\documentclass[12pt]{amsart}
\usepackage{url}
\usepackage{amssymb}
\usepackage{geometry}
\usepackage{verbatim}
\newtheorem{theorem}{Theorem}
\newtheorem{remark}{Remark}
\newtheorem{lemma}{Lemma}
\newtheorem{corollary}{Corollary}
\newtheorem{proposition}{Proposition}
\newcommand{\tmop}[1]{\ensuremath{\operatorname{#1}}}
\numberwithin{equation}{section}
\numberwithin{lemma}{section}
\numberwithin{proposition}{section}
\numberwithin{corollary}{section}
\title[Quantum Unique Ergodicity for half-integral weight forms]{Quantum Unique Ergodicity for half-integral weight automorphic forms}
\author{Stephen Lester and Maksym Radziwi\l\l}
\address{Department of Mathematics, KTH Royal Institute of Technology, SE-100 44 Stockholm, Sweden}
\email{sjlester@gmail.com}
\address{Department of Mathematics, McGill University, 845 Rue Sherbrooke Ouest, Montr\'eal, Qu\'ebec H3A 0G4, Canada}
\email{maksym.radziwill@gmail.com}
\date{\today}
\begin{document}

\begin{abstract} 
We investigate the analogue of the Quantum
Unique Ergodicity (QUE) conjecture for half-integral
weight automorphic forms.
Assuming the Generalized Riemann Hypothesis (GRH) we establish QUE for both
half-integral weight holomorphic Hecke cusp forms for $\Gamma_0(4)$ lying in Kohnen's plus subspace
and for half-integral weight Hecke Maa{\ss} cusp forms for $\Gamma_0(4)$ lying in Kohnen's plus subspace.
By
combining the former result along
with an argument of Rudnick, 
it follows that under GRH the zeros
of these holomorphic Hecke cusp equidistribute with respect to hyperbolic measure
on $\Gamma_0(4)\backslash \mathbb H$ as the weight
tends to infinity.
%
\end{abstract}


\maketitle
\tableofcontents
\section{Introduction}
\subsection{Quantum Unique Ergodicity}
Let $(M,g)$ be a compact, smooth Riemannian manifold without boundary and denote by $-\Delta_g$
the Laplace-Beltrami operator for $M$. Also, let $\tmop{vol}_g$ denote the normalized Riemannian volume form for $M$. Given an orthonormal basis $\{ \phi_{\ell} \}_{\ell}$ of $L^2(M, \text{dvol}_g)$-normalized eigenfunctions of $\Delta_g$ with eigenvalues $\lambda_{\ell} $ the Quantum Ergodicity Theorem of Shnirel'man \cite{Shn}, Colin de Verdi\'ere \cite{YCD} and Zelditch \cite{Z1} implies that if the geodesic flow is ergodic
on the unit cotangent bundle $S^*M$ with respect to the normalized Liouville measure then there exists a  density one subsequence of eigenfunctions $\{ \phi_{\ell_j} \}_j$
 along which the measures $|\phi_{\ell_j}|^2 \tmop{dvol}_g$ weakly converge to $\text{dvol}_g$ as $j \rightarrow \infty$. \footnote{The Quantum Ergodicity Theorem gives the stronger result that $\phi_{\ell_j } $ equidistributes within $S^*M$ as $j \rightarrow \infty$. The weak convergence  $|\phi_{\ell_j}|^2 \tmop{dvol}_g \Rightarrow \text{dvol}_g$ follows from projecting to configuration space. }



If in addition $M$ has negative curvature then more is known about the geodesic flow
beyond ergodicity, e.g.,  exponential decay of correlations, central limit theorem. In this setting Rudnick and Sarnak  \cite{RudnickSarnak}
have conjectured that the quantum limit is unique, that is, $|\phi_{\ell}|^2 \text{dvol}_g$ weakly converges to $ \text{dvol}_g$ along the entire sequence of eigenfunctions $\{ \phi_{\ell} \}_{\ell}$. 
This is known as the Quantum Unique Ergodicity (QUE) conjecture.
The general form of this conjecture appears to be far from
being solved, however, in the special arithmetic setting where $M=\tmop{SL}_2(\mathbb Z) \backslash \mathbb H$
Lindenstrauss \cite{Lindenstrauss} and Soundararajan \cite{SoundQUE}
have shown QUE holds for a Hecke basis of eigenfunctions. 
 Prior
to these breakthroughs, this result was known
to
follow conditionally under the assumption of the
Generalized Riemann Hypothesis (GRH) 
from Watson's formula \cite{Watson}. Additionally,
other cases of arithmetic QUE
have been established by Lindenstrauss \cite{Lindenstrauss} and Silberman-Venkatesh \cite{SilbermanVenkatesh}. 

The hyperbolic surface $\tmop{SL}_2(\mathbb Z) \backslash \mathbb H$
also has the structure of a complex analytic Riemannian surface
and it is natural to wonder if the holomorphic
analogue of QUE holds. The analogue of the eigenfunctions of $\Delta_{\mathbb H}$ lying in $L^2$ are the holomorphic
modular forms, which are intrinsic to number theory. For instance a prototypical example of such a function is the modular discriminant
\[
\Delta(z)=
\sum_{n \ge 1} \tau(n) e(nz),
\]
which is a weight 12 Hecke cusp form,
where $\tau(n)$ is the Ramanujan $\tau$-function. In this context Lindenstrauss's method does not seem to apply. However, by combining
two different approaches, one of which uses Watson's formula \cite{Watson}, Holowinsky and Soundararajan  \cite{Holowinsky, SoundConvex, HolowinskySound} established holomorphic QUE for Hecke cusp forms in the limit as the weight tends to infinity. Their result
has the beautiful corollary, proved by Rudnick \cite{Rudnick} that the zeros of all such Hecke cusp forms equidistribute with respect
to hyperbolic measure within compact subsets of $\tmop{SL}_2(\mathbb Z)\backslash \mathbb H$ with boundary measure zero  as the weight tends to infinity. Holomorphic QUE has been established in other cases as well by Marshall \cite{Marshall}, Nelson \cite{Nelson1, Nelson2},
 and Nelson-Pitale-Saha \cite{NPS}.

\subsection{Half-integral weight automorphic forms}
Let $$\Gamma_0(N)=\left\{ \begin{pmatrix}
a & b \\
c & d
\end{pmatrix}
\in \tmop{SL}_2(\mathbb Z) : c \equiv 0 \pmod N \right\}.$$
A weight  $k$ \textit{modular form} for $\Gamma_0(N)$ is a function $f:\mathbb H \rightarrow \mathbb C$
which, transforms in the following way
\begin{equation} \label{eq:invariance}
f(\gamma z)=\chi(\gamma) \left( \frac{cz+d}{|cz+d|}\right)^{k} f(z),
\qquad \forall \gamma= \begin{pmatrix}
a & b \\
c & d
\end{pmatrix} \in \Gamma_0(N)
\end{equation}
for some character $\chi: \Gamma_0(N) \rightarrow S^1$
and satisfies a suitable growth condition at all cusps of $\Gamma_0(N)\backslash \mathbb H$. If in addition $f$ vanishes at each of the cusps of $\Gamma_0(N)\backslash \mathbb H$
we call $f$ a \textit{cusp form}. 

%

A \textit{half-integral weight Maa{\ss} cusp form} for $\Gamma_0(4)$ is a weight $\tfrac 12$ cusp form that is also  an $ L^2(\Gamma_0(4)\backslash \mathbb H, \tmop{dvol})$-eigenfunction of the following
differential operator
\[
\Delta_{1/2}=y^2\left(\frac{\partial^2}{\partial x^2}+\frac{\partial^2}{\partial y^2}\right)-\frac12 i y \frac{\partial}{\partial x},
\]
and $\chi$ in \eqref{eq:invariance} is given by $\chi(\gamma)=\overline{\varepsilon_d} \Big ( \frac{c}{d} \Big )$
where $\varepsilon_d = 1$ if $d \equiv 1 \pmod{4}$ and $\varepsilon_d = i$ if $d \equiv -1 \pmod{4}$, and $\Big ( \frac{c}{d} \Big ) $ is Jacobi's symbol.
Here $\tmop{dvol}(z)=dxdy/y^2$ denotes hyperbolic measure.

For $k$ a positive integer a \textit{weight $k+\tfrac12$ holomorphic modular form} for $\Gamma_0(N)$ is a holomorphic function $g:\mathbb H \rightarrow \mathbb C$ 
such that $g$ is holomorphic at the cusps of $\Gamma_0(N) \backslash \mathbb H$ and $(\tmop{Im}(z))^{(k+\frac12)/2} g(z)$ is a weight $k+\tfrac12$ modular form with $\chi(\gamma)=\Big( \frac{-1}{d}\Big)^k \overline{\varepsilon_d} \Big( \frac{c}{d}\Big)^{2k+1}$.
These functions emerge naturally in number theory and
an example of such a modular form of weight $\ell+\frac32$ for $\Gamma_0(4)$ is
\[
\theta_P(z)=\sum_{m \in \mathbb Z^3} P(m) e(|m|^2 z), \quad (e(z)=e^{2\pi i z})
\]
where $P$ is a homogeneous harmonic polynomial on $\mathbb R^3$ of degree $\ell$. Other examples of half-integral
weight holomorphic forms are known to encode information on periods of integral weight forms (see \cite{Shintani}), and 
in the level aspect, assuming the Birch and Swinnerton-Dyer conjecture, their coefficients reflect the 
order of the Tate-Shafarevich group of quadratic twists of elliptic curves with rank 0.  

Both the holomorphic and real analytic half-integral weight modular forms play special roles in number theory. 
 For instance, Duke \cite{Duke} established
the equidistribution of CM points on $\text{SL}_2(\mathbb{Z}) \backslash \mathbb H$ by combining a formula of Maa{\ss} (which needed to be extended to the non-compact setting) along with proving a subconvexity bound
for the Fourier coefficients of half-integral weight Maa{\ss} forms. 
This built upon the pioneering work of Iwaniec \cite{Iwaniec} who
had established similar estimates for half-integral weight holomorphic modular forms, which has the application
to the equidistribution of lattice points on a sphere.



%

In this paper we set out to investigate the analogues of QUE in the setting of holomorphic half-integral weight forms for $\Gamma_0(4) \backslash \mathbb{H}$, as well as for half-integral weight Maa{\ss} forms. 
The two cases are rather similar. For this reason, despite stating the holomorphic case
first, we will give a detailed proof only in the case of Maa{\ss} forms. The rationale for our choice is that 
the literature dealing with half-integral weight Maa{\ss} forms is less developed and certain additional difficulties arise. 
We sketch the analogous argument for holomorphic half-integral
weight forms in Section \ref{sec:holomorphic}. 

\subsection{Holomorphic half-integral weight forms}
\subsubsection{Setting}
Write $S_{k+1/2}(\Gamma_0(4))$ for the space of weight $k+\tfrac12$ holomorphic cusp forms for $\Gamma_0(4)$. 
Every $g \in S_{k+1/2}(\Gamma_0(4))$ has a Fourier expansion of the form
\[
g(z)=\sum_{n \ge 1} c(n)e(nz)
\]
and
for odd $p$ the Hecke operator $T_{p^2}$ is defined 
on $S_{k+1/2}(\Gamma_0(4))$ as
\[
T_{p^2} g(z)=\sum_{n \ge 1} \left( c(p^2n)+\left( \frac{(-1)^kn}{p}\right)p^{k-1}c(n)+p^{2k-1}c\left( \frac{n}{p^2}\right)\right)e(nz).
\]
Here we have used the convention
that $c(x)=0$ if $x \notin \mathbb Z$.
We call a half-integral weight cusp form $g$ a \textit{Hecke cusp form} if $T_{p^2} g = \lambda_p g$ for all $p > 2$. 
We recall that the half-integral weight Hecke operators vanish on the primes, thus $T_p = 0 $ for all $p$ (see Shimura \cite{Shimura}, p. 450). We will restrict our attention to Hecke forms in what follows. 

One of the main tools in understanding half-integral
weight forms is the Shimura lift \cite{Shimura}, which to a half-integral weight cusp form associates a cusp form of integer weight.
Following Kohnen-Zagier \cite{KohnenZagier}, we focus on forms belonging to \textit{Kohnen's plus subspace} 
on which the behavior of the Shimura lift is very well understood. Precisely the Kohnen plus space $S_{k+1/2}^{+}(\Gamma_0(4))$ denotes the subspace of $S_{k + 1/2}(\Gamma_0(4))$ of cusp forms whose coefficients satisfy $c(n) = 0$ for $(-1)^k n \equiv 2 , 3 \pmod{4}$
and has a basis consisting of simultaneous eigenfunctions of the operators $T_{p^2}$ for all odd $p$. 
We note that as $k \rightarrow \infty$ asymptotically one-third of half-integral weight Hecke cusp forms belong to the Kohnen space (since $\text{dim } S^{+}_{k + 1/2}(\Gamma_0(4)) \sim k/6$ and $\text{dim }S_{k + 1/2}(\Gamma_0(4)) \sim k / 2$ by \cite{Kohnen1,Niwa} and dimension formulas)

\subsubsection{Results}

Compared to the case of integral weight forms, the study of QUE for half-integral weight forms presents several new difficulties.
First of all, an analogue of Watson's formula \cite{Watson} is not currently known to exist. Consequently, the simple conditional (on GRH) proof of QUE for the integer weight case does not carry over to the half-integral weight setting. 
Secondly, the unconditional techniques of Holowinsky-Soundararajan \cite{HolowinskySound} are not directly applicable since they 
rely crucially on the multiplicativity of the coefficients. The coefficients 
of half-integral weight forms are not multiplicative (except at squares) and in general
correspond to central values of $L$-functions, following Waldspurger \cite{Waldspurger}.  In particular, one does not expect
the analogue of the Riemann Hypothesis to hold for Dirichlet series built from Fourier coefficients of half-integral weight Hecke cusp forms, due to the absence of an Euler product. 
Our first main result establishes QUE for holomorphic half-integral weight Hecke cusp forms
lying in Kohnen's plus space, assuming GRH. 

\begin{theorem} \label{thm:holomorphic}
Assume the Generalized Riemann Hypothesis. Let $g_k$ be a holomorphic half-integral weight cusp form for $\Gamma_0(4)$
of weight $k + \tfrac 12$, with $k$ a positive integer. Suppose that $g_k$ 
\begin{enumerate}
\item is normalized so that $\iint_{\Gamma_0(4) \backslash \mathbb{H}} y^{k + 1/2} |g_k(z)|^2 \tmop{dvol}(z) = 1$,
\item lies in the Kohnen subspace,
\item is a simultaneous eigenfunction of the Hecke operators $T_{p^2}$ , $ p \neq 2$. 
\end{enumerate}
Let $\mathcal{D}$ be a compact subset of $\Gamma_0(4)\backslash \mathbb H$ with boundary measure zero. Then, as $k \rightarrow \infty$,
$$
\iint_{\mathcal{D}} y^{k + \frac12} \ |g_k(z)|^2 \tmop{dvol}(z)\rightarrow \frac{\tmop{vol}(\mathcal{D})}{\tmop{vol}(\Gamma_0(4) \backslash \mathbb{H})}
$$
where $\tmop{dvol}(z) = dx dy / y^2$ is the hyperbolic area measure. 
\end{theorem}

In the  proof of Theorem \ref{thm:holomorphic} we specifically require the Riemann Hypothesis for
$L(s, f)$ with $f \in S_{2k}(\text{SL}_2(\mathbb{Z}))$, all of the quadratic twists $L(s, f \otimes \chi_d)$, with
$d$ fundamental discriminants with $(-1)^k d > 0$ and 
for the symmetric square $L$-function $L(s, \text{Sym}^2 f)$. 


Following the method of Rudnick \cite{Rudnick}, which is closely related to ideas of Shiffman-Zelditch \cite{Shiffman}, 
Theorem \ref{thm:holomorphic} gives an immediate consequence for the distribution of
zeros of holomorphic half-integral weight forms $g_k$ (see also \cite{LesterMatomakiRadziwill} for a
recent refinement of Rudnick's result). 
\begin{corollary} \label{cor:zeros}
Assume the Generalized Riemann Hypothesis. Let $g_k$ be as in Theorem \ref{thm:holomorphic}. 
Then as $k \rightarrow \infty$ the zeros of $g_k$ equidistribute with respect to hyperbolic measure within compact subsets of $\Gamma_0(4) \backslash \mathbb H$  with boundary measure zero. 
\end{corollary}

This means that for $g_k$  as in Theorem \ref{thm:holomorphic} we have
for a compact subset $\mathcal D \subset \Gamma_0(4)\backslash \mathbb H$
with boundary measure zero that 
\[
 \sum_{\substack{g_k(\varrho) = 0 \\ \varrho \in \mathcal D}} 1 = \frac{k}{2}
\cdot \frac{\tmop{vol}(\mathcal{D})}{\tmop{vol}(\Gamma_0(4) \backslash \mathbb{H})}+o(k)
\]
 as $k \rightarrow \infty$, with the zeros counted with multiplicity. 

\subsection{Half-integral weight Maa{\ss} forms}
\subsubsection{Setting}
Let $V$ denote the space generated by the half-integral
weight Maa{\ss} cusp forms for $\Gamma_0(4)$. 
 Every half-integral weight Maa{\ss} cusp form, $g$, with eigenvalue $\lambda=-(\tfrac14+t^2)$ has a Fourier expansion at the cusp at $\infty$
of the following form
\[
g(z)=\sum_{n \neq 0} b_{g,\infty}(n) W_{\frac14 \tmop{sgn}(n), it}(4\pi |n| y) e(nx), \qquad (z=x+iy)
\]
where $W$ is the Whittaker function. 
We will restrict our attention to half-integral weight Maa{\ss} cusp forms which are in addition eigenfunctions of the Hecke operators $T_{p^2}: V \rightarrow V$, $p$ an odd prime. For $p > 2$ the action of a Hecke operator $T_{p^2}$ on a Maa{\ss} cusp form $g$ is given explicitly by
\[
T_{p^2} g(z) =\sum_{n \neq 0} c(n) W_{\frac14 \tmop{sgn}(n), it}(4\pi |n| y) e(nx),
\]
where
\[
c(n)=
p b_{g, \infty}(np^2)+p^{-1/2} \left(\frac{n}{p} \right)b_{g, \infty}(n) +p^{-1}b_{g, \infty}\left( \frac{n}{p^2}\right)
\]
(see Theorem 1.7 of Shimura \cite{Shimura}). 
The operators $T_{p^2}$ are self-adjoint and commute with each other as well as with $\Delta_{1/2}$ (see Theorem 1 of \cite{Kohnen1} for the holomorphic setting and Proposition 1.4 of \cite{KatokSarnak} for the real analytic one).
As before, we will only consider \textit{Hecke Maa{\ss} cusp forms}, which are half-integral weight Maa{\ss} forms which are eigenfunctions
of $T_p^2$, for $p$ odd.

Similarly to the holomorphic case, we have at our disposal the Shimura lift, which is well understood on the Kohnen subspace. In particular, this associates a half-integral weight
 form, $g$, with $b_{g,\infty}(n)=0$ 
for $n \equiv 2,3 \pmod 4$ with a weight $0$ Hecke Maa{\ss} cusp form for $\tmop{SL}_2(\mathbb Z)$ (see Katok-Sarnak \cite{KatokSarnak}).  The Kohnen space $V^+$ is the subspace generated by such forms.

%

\subsubsection{Results}
Before stating our results for Maa{\ss} forms we highlight that it might be possible to apply the ergodic techniques of Lindenstrauss
in the setting of half-integral weight Maa{\ss} forms. Lindenstrauss's method would not however rule out the
possibility of escape of mass into the cusps. In our next theorem we establish, conditionally on GRH, QUE for half-integral weight 
Hecke Maa{\ss} cusp forms and also eliminate the possibility of escape of mass.
%
It is also worth pointing out that our result can be made effective, and yields a slow rate of convergence to equidistribution.
On the other hand making Lindenstrauss's result effective remains a challenging open problem. 
\begin{theorem} \label{thm:maass}
Assume the Generalized Riemann Hypothesis. 
Let $g_j$ be a basis of the Kohnen space $V^{+}$, such that
 $\iint_{\Gamma_0(4) \backslash \mathbb{H}} |g_j(z)|^2  \tmop{dvol}(z) = 1$ for all $j$
and each $g_j$ is a simultaneous eigenfunction of the Hecke operators $T_{p^2}$ for all $p > 2$ as well as of
$\Delta_{1/2}$ with $\Delta_{1/2} g_j = - (\tfrac 14 + t_j^2) g_j$. 
Let $\mathcal{D}$ be a compact subset of $\Gamma_0(4)\backslash \mathbb H$ with boundary measure zero. Then, as $|t_j| \rightarrow \infty$, 
$$
\iint_{\mathcal{D}} |g_j(z)|^2 \tmop{dvol}(z) \rightarrow \frac{\tmop{vol}(\mathcal{D})}{\tmop{vol}(\Gamma_0(4) \backslash \mathbb{H})}
$$
where $\tmop{dvol}(z) = dx dy / y^2$ is the hyperbolic area measure. 
\end{theorem}

In the proof of Theorem \ref{thm:maass} we specifically require the Riemann Hypothesis for 
$L$-functions of weight $0$ Hecke-Maa{\ss} eigencuspforms for the full modular group, all of 
their quadratic twists and the symmetric square $L$-function $L(s,\tmop{Sym}^2 f)$ with $f$ a weight $0$
Hecke-Maa{\ss} eigencuspform for $\text{SL}_2(\mathbb{Z})$. 
The problem of equidistribution of 
half-integral weight
Eisenstein series has been recently addressed by Petridis-Raulf-Risager \cite{Petridis}, where they prove QUE for half-integral weight Eisenstein series under the assumption of a
subconvexity bound for a multiple Dirichlet series. Establishing such a subconvexity bound is still open, apparently even under GRH.
In the weight 0 case
Luo-Sarnak \cite{LuoSarnak} and Jakobson \cite{Jakobson} have unconditionally established QUE for the Eisenstein series.

\subsection{Comments on the proof and further work}

Our proof takes elements from 
Holowinsky's \cite{Holowinsky} unconditional argument for holomorphic integer weight forms and Soundararajan's conditional upper bounds for moments
of $L$-functions \cite{SoundMoments}. 
Similarly to Holowinsky we have to estimate a certain main term and an off-diagonal term consisting of
coefficients of half-integral weight forms. 
Holowinsky's treatment of the off-diagonal is based on sieve estimates, while
ours is based on Soundararajan's work on moments.
It is interesting to notice that Soundararajan's proof could
have been carried out also in the multiplicative setting where it would deliver the corresponding sieve bound. Similarly to Holowinsky's treatment \cite{Holowinsky} we also experience considerable difficulties with estimating the main term. In our case the estimation of the main term corresponds to averaging a Dirichlet polynomial of length $X^2$ over quadratic characters of conductor $\asymp X$. This is a notoriously difficult scenario because an application of the Poisson summation gives back a sum of the same length. We explain how this significant obstruction is resolved in our work in the next section. 



%

In the case of weight zero level 1 Hecke Maa{\ss} forms Watson's formula and the Generalized Lindel\"of Hypothesis imply an upper bound on the rate of convergence in QUE of size $\ll \lambda^{-1/4+\varepsilon}$. 
Currently, no effective
rate is unconditionally known in this setting, but has been 
worked out for the holomorphic case, i.e. Hecke cusp forms for $\text{SL}_2(\mathbb Z)$, in \cite{LesterMatomakiRadziwill}, where the rate obtained is 
a small negative power of the logarithm of the weight.
It also emerges from an inspection of our proof 
that in the half-integral case we only get a bound of $\ll (\log \lambda)^{-\delta}$, for some $\delta>0$, on the rate of convergence to uniform distribution. 
It would be very interesting to obtain a bound 
of size $\ll \lambda^{-\delta}$ for some $\delta>0$.

Finally in a future paper we plan to address the problem of Quantum Ergodicity for half-integral weight forms, that is, to show that the mass of
almost all holomorphic forms in the Kohnen plus space tends to equidistribution as the weight increases (it is also likely that the result
extends to Maa{\ss} forms but for simplicity we have decided not to work out this case). This means that at least as far holomorphic
forms go, the situation in the half-integral case is qualitatively the same as in the integer weight weight case before the breakthroughs of Holowinsky-Soundararajan. It will be therefore very interesting to see if unconditional results
are also possible to obtain. 

\section{Overview of the argument}


\subsection{Reduction to sums of Fourier coefficients} \label{sec:weyl}

For simplicity write $g = g_j$ and $t = t_j$ so that $\Delta_{1/2} g_j = -(\tfrac 14 + t^2) g_j$. 
Note that $t \in \mathbb{R}$ since there are no half-integral Maa{\ss} forms in $V^+$ with exceptional
eigenvalues\footnote{This follows from the fact that the Shimura lift of a Maa{\ss} form $g \in V^+$
with eigenvalue $-(\tfrac 14 + t^2)$ is a Maa{\ss} cusp form for $\text{SL}_2(\mathbb{Z})$ with eigenvalue $-(\tfrac 14 + (2t)^2)$. 
Since $\Delta$ on $\text{SL}_2(\mathbb{Z})$ has no exceptional eigenvalues it follows that neither does $V^+(4)$}.
Since we will take $|t| \rightarrow \infty$ we assume outright that $|t| > 100$. 
By an approximation argument
it suffices to show that for a test function
$\Psi(z) \in C_c^{\infty}(\Gamma_0(4) \backslash \mathbb H)$
\[
\iint_{\Gamma_0(4) \backslash \mathbb H} \Psi(z) |g(z)|^2
\frac{dx \, dy}{y^2} \rightarrow \iint_{\Gamma_0(4) \backslash \mathbb H} \Psi(z) \frac{dx \, dy}{y^2} \qquad (|t|\rightarrow \infty).
\]
For $\Psi(z) \in C_c^{\infty}(\Gamma_0(4) \backslash \mathbb H)$
one has $\Psi(z+1)=\Psi(z)$ and we can define $\widetilde \Psi(z)$
to be the extension of $\Psi$ to $\mathbb H$
by $\Gamma_{\infty}$-periodicity. Additionally, we see that
 $\widetilde \Psi_y(x):=\widetilde \Psi(x+iy)$ has
a Fourier expansion. Using this expansion
and adapting an argument of Luo and Sarnak, (see Section 4 of
\cite{LuoSarnak}, but some care must be taken at the elliptic points of $\Gamma_0(4)\backslash \mathbb H)$, one can expand $\Psi(z)$
into a sum of Poincar\'e series
\[
P_{h,\ell}(z)=\sum_{ \gamma \in   \Gamma_{\infty} \backslash \Gamma_0(4)  } h(\tmop{Im}(\gamma z)) e(\ell \tmop{Re}(\gamma z))
\]
where $h(y)$ is the $\ell th$ Fourier coefficient of $\widetilde \Psi_y(x)$. Thus, we see that Theorem \ref{thm:maass}  follows from estimates for the diagonal term $\ell=0$  (here $P_{h,0}(z)=E(z|h)$ is
the incomplete Eisenstein series)
\begin{equation} \notag
\begin{split}
\iint_{\Gamma_0(4) \backslash \mathbb H} E(z|h) |g(z)|^2 \frac{dx \, dy}{y^2} =&  \frac{1}{\tmop{vol}(\Gamma_0(4) \backslash \mathbb{H})} \iint_{\Gamma_0(4) \backslash \mathbb H} E(z|h) \frac{dx \, dy}{y^2}  +o(1) \qquad (|t|\rightarrow \infty) \\
=&\frac{1}{\tmop{vol}(\Gamma_0(4) \backslash \mathbb H)} \int_0^{\infty} h(y) \frac{dy}{y^2} +o(1)
\end{split}
\end{equation}
along with the bound for the off-diagonal terms, $\ell \neq 0$, 
\begin{equation} \notag
\begin{split}
\iint_{\Gamma_0(4) \backslash \mathbb H} P_{h,\ell}(z) |g(z)|^2 \frac{dx \, dy}{y^2} =  \frac{1}{\tmop{vol}(\Gamma_0(4) \backslash \mathbb{H})} \iint_{\Gamma_0(4) \backslash \mathbb H} P_{h,\ell}(z)\frac{dx \, dy}{y^2}  +o(1)  
=o(1) \quad (|t|\rightarrow \infty).
\end{split}
\end{equation}
%
%
The left-hand sides of both equations above are estimated in terms of 
 the sums of Fourier coefficients of $g(z)$.
By the unfolding technique it is easy to see that
\begin{equation} \label{equ:mainterm}
\iint_{\Gamma_0(4) \backslash \mathbb H} E(z|h) |g(z)|^2 \frac{dx \, dy}{y^2} =\sum_{n \neq 0}
|b_{g, \infty}(n)|^2 \int_0^{\infty}W_{\frac14 \tmop{sgn}(n), it}(4\pi y |n|)^2 h(y) \frac{dy}{y^2}
\end{equation}
and for $\ell \neq 0$
\begin{equation} \label{equ:offdiagonal}
\begin{split}
& \iint_{\Gamma_0(4) \backslash \mathbb H} P_{h, \ell}(z) |g(z)|^2 \frac{dx \, dy}{y^2} \\
& =\sum_{n \neq 0} b_{g, \infty}(n) b_{g, \infty}(n +\ell)\int_0^{\infty} W_{\frac14 \tmop{sgn}(n),it} (4\pi |n|y)
W_{\frac14 \tmop{sgn}(n+\ell),it} (4\pi |n+\ell|y) h(y) \frac{dy}{y^2}.
\end{split} 
\end{equation}
Thus the problem is reduced to estimating (\ref{equ:mainterm}) and (\ref{equ:offdiagonal}) as $|t| \rightarrow \infty$. 

\subsection{Estimating the sums of Fourier coefficients}

At this point we appeal to a consequence of a Waldspurger type formula, obtained in the setting of half-integral
weight Maa{\ss} forms by Baruch and Mao \cite{BaruchMao} (with related works by Bir\'o \cite{Biro} and Katok-Sarnak \cite{KatokSarnak}). It follows from this formula that if  $g \in V^+$ is a half-integral weight Hecke Maa{\ss} cusp form then
for $d$
a fundamental discriminant and $\delta \geq 1$ any integer, 
$$
|b_{g,\infty}(d \delta^2)| \ll_{\varepsilon} \frac{1}{\sqrt{|n|}} \Big ( \frac{L(\tfrac 12, f \otimes \chi_d)}{L(1, \tmop{Sym}^2 f)}
\Big )^{1/2} \cdot \delta^{7/64 + \varepsilon} \cdot |t|^{-\tmop{sgn}(n)/4} \cdot e^{\pi |t|/2} \ , \  n = d \delta^2
$$
where $f$ is the Shimura lift of $g$  and in particular $f$ is a Hecke-Maa{\ss} cusp form for $\tmop{SL}_2(\mathbb Z)$
(see Section \ref{app:one}). 
The exponent $7/64$ is the exponent towards the Ramanujan conjecture obtained by Kim and Sarnak \cite{KimSarnak}. In fact any exponent
$< \tfrac 12$ would do (with some additional work).  
Using the triangle inequality, this reduces the problem of estimating (\ref{equ:offdiagonal}) to showing that for $\ell \neq 0$,
%
$$
\frac{1}{L(1, \tmop{Sym}^2 f)} \sum_{\substack{d_1 \delta_1^2 - d_2 \delta_2^2 = \ell \\ |t| \leq d_1 \delta_1^2 \leq 2|t|}} 
L(\tfrac 12, f \otimes \chi_{d_1})^{1/2} L(\tfrac 12, f \otimes \chi_{d_2})^{1/2} = o(|t|)
$$
where the summation is over fundamental discriminants $d_1, d_2$ and integers $\delta_1, \delta_2$. 
On the other hand the estimation of the main term (\ref{equ:mainterm}) is, through the use of an exact Waldspurger formula, essentially equivalent to showing that
\begin{equation} \label{equ:lfunction}
\frac{1}{L(1, \tmop{Sym}^2 f)} \sum_{\substack{|t| \leq d \leq 2 |t|}} L(\tfrac 12, f \otimes \chi_d) \sim c |t|
\end{equation}
for some constant $c > 0$, 
where the summation is over fundamental discriminants $d$ (although we note that while (\ref{equ:lfunction}) is useful to explain the technical issues that arise, we never deal with (\ref{equ:lfunction}) directly but rather prefer to phrase the problem of estimating (\ref{equ:mainterm}) in terms of the coefficients $|b_{g,\infty}(d)|^2$). 

We obtain the first bound as a consequence of Soundararajan's work on moments of $L$-functions (in fact gaining a small logarithmic saving). 
That such a saving is possible is suggested by the following
``Sato-Tate'' property of the coefficients $L(\tfrac 12, f \otimes \chi_d)$, 
\begin{equation} \notag
\sum_{X \leq |d| \leq 2X} L(\tfrac 12, f \otimes \chi_d)^{1/2} = o(X)
\end{equation}
which is analogous to 
\begin{equation} \notag
\sum_{X \leq n \leq 2X} |\lambda_f(n)| = o(X),
\end{equation}
where $\lambda_f(n)$ are the Hecke eigenvalues of $f$. We remark that for holomorphic forms the first bound is now known unconditionally \cite{RadziwillSound}, while the second has been known for quite some time \cite{Elliott}, and played a critical role in Holowinsky's work \cite{Holowinsky}. 

Let us now turn our attention to the second estimate (\ref{equ:lfunction}) that we need to establish QUE.
The main difficulty here is the following: since $d \asymp |t|$
and $f$ has also spectral parameter $\asymp |t|$ the length of the Dirichlet polynomial approximating $L(\tfrac 12, f \otimes \chi_d)$
is $\asymp |t|^2$. Therefore applying Poisson summation for quadratic characters with conductor $\asymp |t|$ will return a sum of the same length. 
This means that we are confronted with the notorious ``deadlock situation'' discussed for example by Munshi
in \cite{Munshi}, but in any case well recognized by experts. 

We overcome the deadlock situation by using once again another idea of Holowinsky. The starting point is to notice that  
$L(\tfrac 12, f \otimes \chi_d)$ is proportional to $|b_{g,\infty}(d)|^2$ so instead of directly
evaluating the moment \eqref{equ:lfunction}
we head back to the expression for this moment
in terms of the Rankin-Selberg integral \eqref{equ:mainterm}. In order to evaluate the
Rankin-Selberg convolution we consider the following expression 
$$
\iint_{\Gamma_0(4)\backslash \mathbb H} |g(z)|^2 E(z | h) E^Y(z | k)  \tmop{dvol}(z)
$$
where
$$
E^Y ( z | k) = \sum_{\gamma \in \Gamma_{\infty} \backslash \Gamma_0(4)} k( Y \cdot \tmop{Im} (\gamma z)). 
$$
with $k$ a smooth function. Evaluating the above in two different ways shows that the summation in (\ref{equ:lfunction}) can be extended by a factor of $Y$ provided that we have a saving of say, $Y^{100}$, in the shifted convolution problem (\ref{equ:offdiagonal}). This allows us to extend the length of summation in (\ref{equ:lfunction}) by a small power of $\log |t|$, and consequently after applying Poisson summation (which now returns a sum slightly shorter than $|t|$) and using once again the work of Soundararajan \cite{SoundMoments} on moments, we obtain an asymptotic estimate for (\ref{equ:lfunction}). In the actual proof we phrase these details slightly differently by establishing a ``convexity bound'' for the Dirichlet series with coefficients $|b_{g,\infty}(n)|^2$ which then gives an asymptotic evaluation of the average of $|b_{g,\infty}(n)|^2$ when slightly more than $|t|$ terms are added up. 

It is worthwhile to point out that in similar problems involving the second or fourth moment of $L$-functions (as considered by Soundararajan-Young \cite{SoundYoung}) one benefits from the presence of a power greater than the first. This feature is not present in our moment problem, however. Finally we notice that if we could obtain unconditionally a power saving in the shifted convolution problem (\ref{equ:offdiagonal}), then we could extend the length of summation in the first moment by a small power of $|t|$ and then evaluate this first moment unconditionally. This would lead to an unconditional proof of QUE for half-integral weight forms, with a power-saving in the rate of convergence. A power-saving in the shifted convolution problem (\ref{equ:offdiagonal}) seems to however remain far out of reach, in particular due to the sum being averaged in a range which corresponds to $|t|$ where $t$ is the spectral parameter. 


\section{Estimates for character sums}

The results on moments of $L$-functions which appear in the next section depend on a character sum estimate which 
we obtain below. 
\begin{proposition}
\label{prop:charsum}
Let $F$ be a Schwartz function with $\widehat{F}$ compactly supported in $(-A, A)$ for some fixed $A > 0$. 
Also, let $\ell$ be an integer and $a,b$ be integers such that $1 \le a, b \leq X^{\varepsilon}$. Suppose $1 \le r, s \le X^{\varepsilon}$ are odd integers with $(ab, rs) = 1$.
Write $r=r_0q^2$ and $s=s_0w^2$ where $r_0,s_0$ are square-free and 
let
\[
\mathcal M(r_0,q,w,a,b,\ell)=\sum_{\substack{d_1|q, d_2 | w \\ (d_1 d_2 ,r_0) = 1 \\ (ad_1,b d_2) | \ell }}
\frac{\mu(d_1) \mu(d_2)}{[d_1,d_2]}  \frac{c_{r_0}(\ell/(ad_1,bd_2)) }{r_0} 
\]
where $c_r(\ell) := \sum_{(a,r) = 1} e(a \ell / r)$ is a Ramanujan sum.  
Then, for all $X$ large enough with respect to $A$, if $s_0=r_0$
\begin{equation} \notag
\begin{split}
\sum_{\substack{m,n \\ am = bn + \ell}} \Big ( \frac{m}{r} \Big ) \Big ( \frac{n}{s} \Big ) F \Big ( \frac{am}{X} \Big ) =
\left(\frac{a b / (a , b)^2}{r_0}\right) \widehat{F}(0) \frac{X}{[a,b] } \cdot \mathcal M(r_0,q,w,a,b,\ell) 
\end{split}
\end{equation}
and if $r_0 \neq s_0$
\[
\sum_{\substack{m,n \\ am = bn + \ell}} \Big ( \frac{m}{r} \Big ) \Big ( \frac{n}{s} \Big ) F \Big ( \frac{am}{X} \Big ) = 0.
\]
%
%

\end{proposition}

\begin{remark} \label{rem:simplecharsum}
In the simple case where $a=b=1$, $\ell=0$ and $s=1$ this result reduces
to the estimate
\[
\sum_{m} \left( \frac{m}{r}\right) F\left( \frac{m}{X}\right)=
\begin{cases}
\widehat F(0) \varphi(r)  \cdot  \displaystyle \frac{X}{r} & \text{ if }  r=\square, \\
0  & \text{ otherwise},
\end{cases}
\]
which we will also use later.
\end{remark}
Additionally, it is easily seen from the proof that if $F$ is only a Schwartz function
then the results of Proposition \ref{prop:charsum} remain true up to an additional error term of $O_A(X^{-A})$. 
The proof of Proposition \ref{prop:charsum} depends on the following version of Poisson summation, which
we recall here.
\begin{lemma}
\label{le:Poisson}
Let $F$ be a Schwartz function. 
Suppose that $(d,r) = 1$. 
Then,
$$
\sum_{\substack{n \equiv t \pmod{d} \\ 
}} \Big ( \frac{n}{r} \Big ) F(n) = \frac{1}{d r} \Big ( \frac{d}{r} \Big ) 
\sum_{k} \widehat{F} \Big ( \frac{k}{dr} \Big )
e \Big ( \frac{t k \overline{r}}{d} \Big ) \tau_k(r) 
$$
where
$$
\tau_k(n) := \sum_{b \pmod{n}} \Big ( \frac{b}{n} \Big ) e \Big ( \frac{k b}{n} \Big ),
$$
and 
$$
\widehat{F}(x) := \int_{\mathbb{R}} F(u) e(-x u) du
$$
is the Fourier transform of $F$. 
\end{lemma}

\begin{proof}
We have
\begin{align} \notag
\sum_{\substack{n \equiv t \pmod{d}}} \Big ( \frac{n}{r} \Big ) F(n) & = \sum_{b \pmod{r}} \Big ( \frac{b}{r} \Big )
\sum_{\substack{n \equiv b \pmod{r} \\ n \equiv t \pmod{d}}} F(n).
\end{align}
Since $(r,d) = 1$, the congruence condition can be re-written as
$n \equiv b d \overline{d} + t r \overline{r} \pmod{r d}$. 
This way we re-write the inner sum over $n$ as
$$
\sum_{\substack{n \equiv b d \overline{d} + t r \overline{r} \pmod{r d}}} F(n).
$$
By Poisson summation
$$
\sum_{n \equiv b d \overline{d} + t r \overline{r} \pmod{r d}} F(n) = \frac{1}{r d} \sum_{k} \widehat{F} \Big ( \frac{k}{r d} \Big )
e \Big ( \frac{k b \overline{d}}{r} + \frac{k t \overline{r}}{d} \Big ) 
$$
and the claim follows.
\end{proof}


We are now ready to prove the proposition.

\begin{proof}[Proof of Proposition~\ref{prop:charsum}]
Dividing the linear condition $a m = b n + \ell$ by $(a,b)$ we see that
we need to have $(a,b) | \ell$ and that upon substituting $a/(a,b)$ for $a$, 
$b / (a,b)$ for $b$,  $\ell / (a,b)$ for $\ell$ and finally $X / (a,b)$ for $X$, 
we see that we can as well assume that $a$ and $b$ are co-prime from the outset, 
which we will do now. 

Consider first the case where both $r$ and $s$ are square-free.
Write $r = d r^{\star}$ and $s = d s^{\star}$ with $d = (r,s)$. Since $r$ was assumed to be square-free, $(d, r^{\star}) = 1$. 
Given $t \pmod{d}$ if we require that $m \equiv t \pmod{d}$ then the condition $m a = n b + \ell$ implies that
$\overline{b} (t a - \ell) \equiv n \pmod {d}$, with $\overline{b}$ the modular inverse of $b \pmod{d}$. Therefore,
\[
\begin{split}
&\sum_{\substack{m,n \\ am = bn + \ell}} \Big ( \frac{m}{r} \Big ) \Big ( \frac{n}{s} \Big ) F \Big ( \frac{am}{X} \Big ) = \sum_{t \pmod {d}} \Big ( \frac{t \overline{b}(ta- \ell)}{d} \Big ) 
\sum_{\substack{m \equiv t \pmod{d} \\ ma = nb+\ell}} \Big ( \frac{m}{r^{\star}} \Big ) \Big ( \frac{n}{s^{\star}}
\Big ) F \Big ( \frac{am}{X} \Big ).
\end{split}
\]
Moreover given $u \pmod{s^{\star}}$ we see that summing over $m,n$ such that $n \equiv u \pmod{s^{\star}}$ and $m a - nb = \ell$
is the same as summing over $m,n$ such that $n \equiv u \pmod{s^{\star}}$ and $m a - \ell \equiv u b \pmod{b s^{\star}}$. 
%
With this in mind we re-write the inner sum over $m$ as
\[
\begin{split}
S &:= \sum_{u \pmod{s^{\star}}} \Big ( \frac{u}{s^{\star}} \Big ) \sum_{\substack{ m \equiv t \pmod{d} \\ ma-\ell \equiv ub \pmod{b s^{\star}}}} \Big ( \frac{m}{r^{\star}} \Big ) F \Big ( \frac{ma}{X} \Big ) \\
&= \sum_{u \pmod {s^{\star}}} \Big ( \frac{u}{s^{\star}} \Big ) \sum_{m \equiv t \pmod{d}} \Big ( \frac{m}{r^{\star}} \Big ) \cdot \frac{1}{bs^{\star}}
\sum_{0 \leq v < b s^{\star}} e \Big ( \frac{v (ma -ub- \ell)}{bs^{\star}} \Big ) F \Big ( \frac{ma}{X} \Big ) \\
&= \frac{1}{bs^{\star}} \sum_{0 \leq v < bs^{\star}} e \Big ( \frac{-v \ell}{bs^{\star}} \Big ) \sum_{u \pmod{s^{\star}}}
\Big ( \frac{u}{s^{\star}} \Big ) e \Big ( - \frac{v u}{s^{\star}} \Big ) \sum_{m \equiv t \pmod{d}} \Big ( \frac{m}{r^{\star}} \Big ) 
e \Big ( \frac{v ma} {bs^{\star}} \Big ) F \Big ( \frac{ma}{X} \Big ).
\end{split}
\]
Applying Lemma~\ref{le:Poisson} we get
$$
\sum_{m \equiv t \pmod{d}} \Big ( \frac{m}{r^{\star}} \Big ) e \Big ( \frac{a v m}{bs^{\star}} \Big ) F \Big ( \frac{ma}{X} \Big )
= \frac{X}{a r^{\star} d} \Big ( \frac{d}{r^{\star}} \Big ) 
\sum_{k} \widehat{F} \Big ( \frac{X}{a} \Big ( \frac{k}{dr^{\star}} - \frac{va}{bs^{\star}} \Big ) \Big ) e \Big (\frac{k t \overline{r^{\star}}}{d} \Big ) \tau_k(r^{\star})
$$
where we have used the assumption $(d, r^{\star}) = 1$. 
Since $d, a, b, r^{\star}, s^{\star} < X^{\varepsilon}$ and $\widehat{F}$ is compactly supported in $(-A, A)$ the above term is equal to $0$ for all large enough $X$ unless $k b s^{\star}= vad r^{\star}$. Since $s$ is square-free $(s^{\star},d)=1$ so $(bs^{\star}, a d r^{\star}) = 1$. Therefore
$bs^{\star} | v$, but since $0 \leq v < bs^{\star}$ this implies that $v = 0$, and hence that
also $k = 0$. Therefore 
$$
S = \frac{X}{ab r^{\star} s^{\star} d} \widehat{F}(0) \Big ( \frac{d}{r^{\star}} \Big )
\sum_{u \pmod{s^{\star}}} \Big ( \frac{u}{s^{\star}} \Big ) 
\sum_{u' \pmod{r^{\star}}} \Big ( \frac{u'}{r^{\star}} \Big ) . 
$$
The main term is clearly zero unless $s^{\star} = r^{\star}  = 1$, which implies that $r = s$. 
Hence,
\begin{equation} \label{eq:sqfree1}
\sum_{\substack{m,n \\ am = bn + \ell}} \Big ( \frac{m}{r} \Big ) \Big ( \frac{n}{s} \Big ) F \Big ( \frac{am}{X} \Big ) =\delta_{r=s} \frac{X}{ab r} \widehat F(0)  \sum_{t \pmod {r}} \Big ( \frac{t \overline{b}(ta- \ell)}{r} \Big ) .
\end{equation}

To complete the proof in the square-free case
we now estimate the sum on the right-hand side of \eqref{eq:sqfree1}.
First
observe that
\begin{equation} \label{eq:sqfree2}
\frac1r \sum_{0 \le t <r} e\left( \frac{-t\ell}{r}\right) \tau_{at}(r) \overline{\tau_{bt}(r)}
=\sum_{t \pmod r} \left( \frac{\overline{b} t(at-\ell)}{r}\right).
\end{equation}
Next, for odd square-free $r$ a little calculation using
Gauss sums and quadratic reciprocity gives
$$
\tau_v(r) = \Big ( \frac{1 + i}{2} + \Big ( \frac{-1}{r} \Big ) \frac{1 - i}{2} \Big )
\Big ( \frac{v}{r} \Big ) \sqrt{r}.
$$
In particular $\tau_v(r) = 0$ unless $(v,r) = 1$, and, if $(v,r) = 1$ then
$$
\tau_{av}(r)\overline{\tau_{bv}(r)} = \Big ( \frac{av}{r} \Big ) \sqrt{r} \Big ( \frac{bv}{r} \Big ) \sqrt{r} = \Big ( \frac{ab}{r} \Big ) r.
$$
Thus, 
$$
\frac1r \sum_{0 \le t <r} e\left( \frac{-t\ell}{r}\right) \tau_{at}(r) \overline{\tau_{bt}(r)}
=\left( \frac{ab}{r} \right) c_r(\ell)
$$
where $c_r(\ell)$ is a Ramanujan sum. Combining this with \eqref{eq:sqfree1} and \eqref{eq:sqfree2} completes the proof for square-free $r,s$.

Write now $r = r_0 q^2$ and and $s=s_0 w^2$ with $r_0$ and $s_0$ square-free. Then we can re-write the sum as
$$
\sum_{\substack{m,n \\am = bn + \ell \\ (m, q) = (n, w) = 1}} \Big ( \frac{m}{r_0} \Big ) \Big ( \frac{n}{s_0} \Big ) F \Big ( \frac{am}{X} \Big ) . 
$$
Introducing M\"obius inversion this is equal to
$$
\sum_{\substack{d_1| q , d_2 | w \\ (d_1,r_0)=( d_2 ,s_0) = 1}} \mu(d_1) \mu(d_2) \Big ( \frac{d_1 }{r_0} \Big )  \Big( \frac{d_2}{s_0} \Big) \sum_{\substack{m,n \\ad_1 m = bd_2 n + \ell}} \Big ( \frac{m}{r_0} \Big )
\Big ( \frac{n}{s_0} \Big ) F \Big ( \frac{a d_1 m}{X} \Big ) .
$$
By the result for square-free $r$ and $s$, this equals zero unless $r_0=s_0$ and in this case this equals
\[
\begin{split}
&\hat{F}(0) \frac{X}{[a,b] } \sum_{\substack{d_1|q, d_2 | w \\ (d_1 d_2 ,r_0) = 1 \\ (ad_1,b d_2) | \ell}}  \frac{\mu(d_1) \mu(d_2)}{[d_1,d_2]} \Big ( \frac{d_1 d_2}{r_0} \Big ) \cdot \left(\frac{a d_1 b d_2 / (a d_1, b d_2)^2}{r_0}\right)  \frac{c_{r_0}(\ell/(ad_1,bd_2)) }{r_0}  \\
&= \left(\frac{a b / ( a , b)^2 }{r_0}\right) \hat{F}(0) \frac{X}{[a,b] } \cdot \sum_{\substack{d_1|q, d_2 | w \\ (d_1 d_2 ,r_0) = 1 \\ (ad_1,b d_2) | \ell }} \frac{\mu(d_1) \mu(d_2)}{[d_1,d_2]}  \frac{c_{r_0}(\ell/(ad_1,bd_2)) }{r_0},
\end{split}
\]
since $(a d_1, b d_2) = (a, b) (d_1, d_2)$ because $d_1 | r$ and $d_2 | s$ and $(rs, a b) = 1$ by assumption. 
\end{proof}

\section{A moment calculation}
The main result of this section is an estimate for moments of a short Dirichlet polynomial, presented in Lemma \ref{lem:moments12} below. This will be the crucial ingredient in our bound for
\[
\sum_{\substack{d_1,d_2 \\  a|d_1| \le X \\ ad_1=bd_2+\ell}}L(\tfrac 12, f \otimes \chi_{d_1})^{1/2} L(\tfrac 12 , f \otimes \chi_{d_2})^{1/2}
\]
obtained in the next section, where the summation is over fundamental discriminants $d_1,d_2$ and $f$ is an even, arithmetically normalized,  Hecke-Maa{\ss}
cusp form for $\text{SL}_2(\mathbb{Z})$.
\begin{lemma} \label{lem:moments12}
Let $x \leq X^{\varepsilon/k}$ for some small $\varepsilon > 0$ and $1 \le a,b, |\ell| \le (\log X)^{100}$. 
For arbitrary real coefficients $a(p)$ such that $|a(p)| \le p^{1/2-\delta}$ for some $\delta>0$
we have 
\begin{align*}
 \sum_{\substack{d_1, d_2 \\ a |d_1|  \le X \\ a d_1 = b d_2 + \ell}} & \Big ( \sum_{\substack{2 < p \leq x \\ p \nmid a b}} \frac{a(p) (\chi_{d_1}(p) + \chi_{d_2}(p))}{\sqrt{p}} \Big )^{2k}  \\
& \ll \frac{X}{[a,b]} \cdot \frac{(2k)!}{2^k k!} \Big ( 2 \sum_{ p \leq x} \frac{a(p)^2}{p} + O \Big ( 1 + \sum_{p | \ell} \frac{a(p)^2}{p} \Big ) \Big )^k 
\end{align*}
provided that $x$ is chosen so that $\sum_{p \leq x} a(p)^2 / p \gg  1 + \sum_{p | \ell} a(p)^2 / p$ 
and where the sum is taken over fundamental discriminants $d_1, d_2$. 
\end{lemma}

The proof of Lemma \ref{lem:moments12} depends on the following lemma, which uses the character sum estimate obtained in the previous section. 

\begin{lemma} \label{lem:momentsprimes}
Let $F$ be a Schwartz function with $\widehat F$ compactly supported in $(-10, 10)$, and let $a(p)$
be a sequence of real coefficients such that $|a(p)| \le p^{1/2-\delta}$ for some $\delta>0$.  
Suppose that $1 \leq a, b, |\ell| \leq (\log X)^{100}$. 
Let
$$
B(x) = \frac{1}{2} \sum_{p \leq x} \frac{a(p)^2}{p} \text{ and } A(\ell) := 1 + \sum_{p | \ell} \frac{a(p)^2}{p} 
$$
Then, for $k,j \geq 0$ and $x$ such that
\begin{enumerate}
\item $x^{k + j} \leq X^{\varepsilon}$ for some small $\varepsilon > 0$,
\item $B(x) \gg  A(\ell)$ ,
\end{enumerate}
we have,
\begin{align*}
& \sum_{\substack{m, n \\ a m = b n + \ell}}  \Bigg ( \sum_{\substack{2 < p \leq x \\ p \nmid ab}}
 \frac{a(p) \Big ( \frac{m}{p} \Big )}{\sqrt{p}} \Bigg )^{k} 
\Bigg ( \sum_{\substack{2 < p \leq x \\ p \nmid ab}} \frac{a(p) \Big ( \frac{n}{p} \Big )}{\sqrt{p}} \Bigg )^{j}
F \Big ( \frac{a m}{X} \Big ) \ll \\ &
\ll \mathbf{1}_{2 | k - j} \cdot \widehat{F}(0)  \frac{X}{[a,b]} \cdot C(k) C(j)\cdot 
\Big ( 2 B(x) + O(A( \ell)) \Big )^{(k + j)/2 - \eta(k,j)} \cdot A( \ell)^{\eta(k,j)}
\end{align*}
where 
$
C(k) = k! / (\lfloor k/2 \rfloor ! 2^{k/2})
$
and $\eta = \eta(k,j)$ equals $0$ if $k \equiv j \equiv 0 \pmod{2}$ and $\eta = 1$ if 
$k \equiv j \equiv 1 \pmod{2}$. 
\end{lemma}
We will prove Lemma \ref{lem:momentsprimes} and Lemma \ref{lem:moments12} shortly, but beforehand we state and prove
the following fairly standard lemma (see for example \cite[Lemma 3]{SoundMoments}) that will also be required later on.
%

\begin{lemma} \label{lem:multilinear}
Let $J > 0$ be given. Let $F$ be a Schwartz function such that $\widehat F$ has compact support in $(-10,10)$.
Suppose $k_1, \ldots, k_J$ are non-negative integers such that $x^{k_1+\cdots+k_J} \leq X^{\varepsilon}$. If the intervals $I_j$ are disjoint and all contained in $[1,x]$,  then, for real coefficients $a(p)$,  
$$
\Bigg|
\sum_{m} \prod_{j \leq J} \Bigg(  \sum_{p \in I_j} \frac{a(p)  \left( \frac{m}{p}\right)}{\sqrt{p}} \Bigg)^{k_j}  F\left( \frac{m}{X}\right) \Bigg|
\le \widehat F(0) X
\prod_{j \leq J} \Bigg(\widetilde  C(k_j) \cdot \Big ( \sum_{p \in I_j} \frac{a(p)^2}{p} \Big )^{k_j/2} \Bigg),
$$
where $\widetilde C(k)$ is the $k$th moment of a normal random variable with mean 0 and variance $1$, that is $\widetilde C(2\ell) = (2\ell)! / (\ell! 2^{\ell})$ for $\ell$ a non-negative integer, and $\widetilde C(2\ell + 1) = 0$. 
\end{lemma}

\begin{proof}[Proof of Lemma \ref{lem:multilinear}]
First, we
extend $a(p)$ to a completely multiplicative function $a(n)$ by setting $a(p^{\alpha}) = a(p)^{\alpha}$. 
Also, let $p_j(n)$ equal one if $\Omega(n)=j$ and zero otherwise. We also define the multiplicative function
$\nu(p^{\alpha})=\alpha!$. 
In this notation, we expand the moment  as
\[
k_1! \cdots k_J!
\sum_{\substack{r_1, \ldots, r_J \\
p_j | r_j \Rightarrow p_j \in I_j, j=1, \ldots, J} } \frac{a(r_1) \cdots a(r_J) p_{k_1}(r_1) \cdots p_{k_J}(r_J)}{\sqrt{r_1 \cdots r_J} \nu(r_1) \cdots \nu(r_J)}
\sum_{m} \left( \frac{m}{r_1 \cdots r_J}\right)F\left( \frac{m}{X}\right).
\]
Since $r_1 \cdots r_J \le x^{k_1+ \cdots +k_J} \le X^{\varepsilon}$ we can apply 
Proposition \ref{prop:charsum} (see Remark \ref{rem:simplecharsum}) to see that the inner sum over $m$ equals 0
unless $r_1 \cdots r_J=\square$, and this implies that $r_j=\square$ for each $j=1, \ldots, J$, since $(r_i, r_j)=1$ for $i \neq j$.
Since $\Omega(r_j)=k_j$ this completes the proof in the case where $k_j$ is odd for some $j$. 

To handle the remaining case
write $r_j=n_j^2$ and $k_j=2h_j$ so by Proposition \ref{prop:charsum} the above equation equals
\[
X \widehat F(0)
(2h_1)! \cdots (2h_J)! 
\sum_{\substack{n_1, \ldots, n_J \\
p_j | n_j \Rightarrow p_j \in I_j, j=1, \ldots, J} } \frac{a(n_1)^2 \cdots a(n_J)^2 p_{h_1}(n_1) \cdots p_{h_J}(n_J)}
{ n_1 \ldots n_J \nu(n_1^2) \cdots \nu(n_J^2)} \frac{\varphi(n_1^2 \cdot n_J^2)}{n_1^2 \cdots n_J^2}.
\]
Now apply the inequalities $\nu(n^2) \ge 2^{\Omega(n)} \nu(n)$ and $\varphi(n)/n \le 1$ then rearrange the sum to see that
the above is
\begin{equation} \notag
\begin{split}
&\le  X \widehat F(0)(2h_1)! \cdots (2h_J)!   \sum_{\substack{n_1, \ldots, n_J \\
p_j | n_j \Rightarrow p_j \in I_j, j=1, \ldots, J} } \frac{a(n_1)^2 \cdots a(n_J)^2 p_{h_1}(n_1) \cdots p_{h_J}(n_J)}
{ n_1 \ldots n_J 2^{\Omega(n_1) +\cdots +\Omega(n_J)} \nu(n_1) \cdots \nu(n_J)}\\
&= X \widehat F(0)
\prod_{j \le J} \frac{(2h_j)}{2^{h_j} h_j!} \left( \sum_{p \in I_j} \frac{a(p)^2}{p}\right)^{h_j}.
\end{split}
\end{equation}
\end{proof}


We now prove Lemma \ref{lem:momentsprimes}.

\begin{proof} [Proof of Lemma  \ref{lem:momentsprimes}]
We can assume at the outset that $(a,b) | \ell$ otherwise the result is vacuously true. 
In the notation of the proof of the previous lemma we have
\begin{equation} \label{eq:covariance}
\begin{split}
&\sum_{\substack{m,n \\ am=bn+\ell}} \Bigg(\sum_{\substack{2<p \le x \\ p \nmid ab}} \frac{a(p) \Big ( \frac{m}{p} \Big ) }{\sqrt{p}} \Bigg)^{k}
\Bigg(\sum_{\substack{2<p \le x \\ p \nmid ab}} \frac{a(p) \Big ( \frac{n}{p} \Big )}{\sqrt{p}} \Bigg)^{j} F\left(\frac{am}{X} \right)\\
&= k! j! \sum_{\substack{r, s \\ p|rs \Rightarrow 2< p \le x \\ (rs, ab)=1}}
\frac{a(r)a(s)p_k(r)p_{j}(s)}{\sqrt{rs} \nu(r)\nu(s)}\sum_{\substack{m,n \\ am=bn+\ell}} \Big ( \frac{m}{r} \Big ) \Big ( \frac{n}{s} \Big ) 
F\left( \frac{am}{X}\right) .
\end{split}
\end{equation}
Let $\rho(n)$ be the multiplicative function given by
$\rho(p^{\alpha})=(p^{\alpha},\ell)/p^{\alpha}$
and $\varsigma(n)$ be the multiplicative function given
by $\varsigma(p^{\alpha})=(1+3/p)^{\alpha}$.
Notice that the main term $\mathcal M(r_0,q,w,a,b,\ell)$ in Proposition \ref{prop:charsum}
is bounded by
\begin{equation} \notag 
\Bigg| \sum_{\substack{d_1|q, d_2 | w \\ (d_1 d_2 ,r_0) = 1 \\ (ad_1,b d_2) | \ell }} \frac{\mu(d_1) \mu(d_2)}{[d_1,d_2]}  \frac{c_{r_0}(\ell/(ad_1,bd_2)) }{r_0} \Bigg| \le \rho(r_0)\varsigma(qw).
\end{equation}
Write $r=r_0q^2$ and $s=s_0w^2$ where $r_0$ and
$s_0$ are square-free and
apply Proposition \ref{prop:charsum}
to see that only the terms with $r_0 = s_0 = t$ survive and \eqref{eq:covariance} is
\begin{equation} \notag
\begin{split}
\ll& \widehat F(0) \cdot \frac{X}{[a,b]} k!j! \sum_{\substack{t, q, w \\ p|t q w \Rightarrow p\le x}} \frac{a^2(t q w)\rho(t)\varsigma( qw ) p_{k}(t q^2)p_{j}(t w^2) \mu^2(t)}{t qw \nu(t n_1^2)\nu(t m_1^2)}  \\
\end{split}
\end{equation}
Now notice that if $k$ and $j$ are of different parities, then $p_{k}(t q^2) p_j(t w^2) = 0$. Set $\eta = 1$ if $k$ and $j$ are both odd
and set $\eta = 0$ if $k$ and $j$ are both even. 
Without loss of generality assume $k \le j$ and write $k_0=(k-\eta)/2$ and $j_0=(j-\eta)/2$. Thus,
 using the inequalities 
 $1 \le \nu (m) \nu(n) \le \nu (mn)$ and  $\nu(n^2) \geq 2^{\Omega(n)} \nu(n)$ we bound the above sum by
\begin{equation} \label{eq:combinatorics}
\begin{split}
\leq  \widehat F(0) \cdot 
\frac{X}{[a,b]} k! j! \sum_{h=0}^{k_0} \sum_{\substack{t, q, w \\ p|t qw \Rightarrow p\le x}} \frac{a^2(t qw)\rho(t)\varsigma(q)\varsigma(w) p_{2h + \eta}(t) p_{k_0-h}(q)p_{j_0-h}(w)}{t qw 2^{\Omega(q)+\Omega(w)} \nu(t)\nu(q)\nu(w)} .
\end{split}
\end{equation}
Let
\[
A(x)=\sum_{p \le x} \frac{a(p)^2 \rho(p)}{p} \qquad \text{ and } \qquad B(x)= \frac12 \sum_{p \le x} \frac{a(p)^2}{p}.
\]
Note that $A(x) \ll A(\ell)$.
Rearranging \eqref{eq:combinatorics} we see that it equals
\begin{equation} 
\label{eq:finalmomentsbd}
\begin{split}
 = &  \widehat F(0) \cdot  \frac{X}{[a,b]} k!j!
\sum_{h=0}^{ k_0} \frac{ A(x)^{2h+\eta} (B(x)+O(1))^{k_0+j_0-2h}  }{(k_0 -h)!(j_0 -h)!(2h+\eta)!}.
\end{split}
\end{equation}
To bound the above sum use the inequality $ m^n (m-n)! \ge  m!$ to get
\begin{align*}
& \sum_{h = 0}^{ k_0 } \frac{A(x)^{2h + \eta} \cdot (B(x)+O(1))^{k_0 + j_0 - 2h}}{(k_0-h)! (j_0 - h)! (2h+\eta)!} 
\\ & \quad \quad \ll \frac{A(x)^{\eta}(B(x)+O(1))^{k_0+j_0}}{k_0! j_0!} \sum_{h = 0}^{ k_0} \frac{1}{(2h+\eta)!} \cdot 
\left( \frac{\sqrt{k_0j_0} A(x)}{B(x)+O(1)}\right)^{2h} 
\end{align*}
Note that the inner sum over $h$ is
$$
\ll \exp \Big ( \frac{2 \sqrt{h_0 k_0} A(x)}{B(x)} \Big ) \leq \exp \Big (\frac{4 h_0 A(x)}{B(x)} \Big )
\cdot \exp \Big ( \frac{4 k_0 A(x)}{B(x)} \Big ). 
$$
Therefore (\ref{eq:finalmomentsbd}) is
\begin{align*}
& \ll \widehat{F}(0) \cdot \frac{X}{[a,b]} \cdot \frac{k! j!}{k_0! j_0!} 
\cdot (e^{4 A(x) / B(x)} \cdot (B(x) + O(1))^{k_0 + j_0} )\cdot A(x)^{\eta}
\\ & \ll \widehat{F}(0) \cdot \frac{X}{[a,b]} \cdot \frac{k! j!}{k_0! j_0!}
\cdot (B(x) + O(A(x)))^{k_0 + j_0} \cdot A(x)^{\eta}. 
\end{align*}
Since $B(x) \gg 1 + A(x)$ by assumption we have $e^{4 A(x)/ B(x)} = 1 + O(A(x)/B(x))$. 
\end{proof}

%
%

We are finally ready to prove the main result of this section. 

\begin{proof} [Proof of Lemma \ref{lem:moments12}]
For fundamental discriminants $d_1, d_2$ we have, 
$$
\chi_{d_1}(p)  = \Big ( \frac{d_1}{p} \Big ) \text{ and } \chi_{d_2}(p) = \Big ( \frac{d_2}{p} \Big ).
$$
Thus after replacing $\chi_d$ by the corresponding Jacobi symbol, we extend the summation over $d_1, d_2$ to all 
integers, and not just fundamental discriminants.  Also, let $F$ be a Schwartz function such that $\widehat F$ has compact support in $(-10,10)$ and $F(x) \ge \mathbf 1_{[-1,1]}(x)$ for all $x \in \mathbb R$. Using the binomial theorem and the previous lemma we see that the 
$2k$-th moment over fundamental discriminants is bounded by 
\begin{align}  \notag 
\sum_{\substack{m,n \\ am = bn + \ell}} & \Big ( \sum_{\substack{2 < p \leq x \\ p \nmid ab}} \frac{a(p) \Big ( \Big ( \frac{m}{p} \Big ) + \Big ( \frac{n}{p} \Big )  \Big ) }{\sqrt{p}} \Big )^{2k} F \Big ( \frac{a m }{X} \Big ) \\ 
\notag
&= \sum_{j \leq 2k} \binom{2k}{j} \sum_{\substack{m,n \\ a m = bn + \ell}} \Big ( \sum_{\substack{2 < p \leq x \\ p \nmid ab}} 
\frac{a(p) \Big ( \frac{m}{p} \Big )}{\sqrt{p}} \Big )^{j} \Big ( \sum_{\substack{2 < p < x \\ p \nmid a b}} \frac{a(p) \Big ( \frac{n}{p} \Big )}{\sqrt{p}} \Big )^{2k - j}  F \Big ( \frac{am}{X} \Big ).
\end{align}
Writing
\[
B(x)=\frac12  \sum_{p \le x} \frac{a(p)^2}{p} \mbox{  and  } A(\ell)=1+\sum_{p | \ell} \frac{a(p)^2}{p}
\]
and applying the previous lemma we bound this as
\begin{equation} \label{eq:finally}
\ll \frac{X}{[a,b] } \left( 2B(x) +O( A(\ell)  ) \right)^k\sum_{j \leq 2k} \binom{2k}{j} C(2k - j) C(j) \left(\frac{ A(\ell)}{B(x)}\right)^{\eta(2k-j,j)},
\end{equation}
where $\eta(k,j)$ equals one if $k$ and $j$ have the same parity and is zero otherwise.
The contribution of the even $j$ is bounded by 
$$
\ll \frac{(2k)!}{2^k} \sum_{j = 0}^{k} \frac{1}{(k-j)! j!} = \frac{(2k)!}{k!}.
$$
Whereas the odd $j$ in \eqref{eq:finally} contribute
\[
\ll
\frac{A(\ell)}{B(x)} \frac{(2k)!}{2^{k-1}} \sum_{j=0}^{k-1} \frac{1}{(k-1-j)! j!}= \frac{k A(\ell)}{B(x)} \cdot \frac{(2k)!}{k!}
\le \frac{(2k)!}{k!} \left(1+\frac{A(\ell)}{B(x)} \right)^k ,
\]
 where in the last inequality we used the elementary estimate $kw <(1+w)^k$ for $k \ge 1,w>0$. Using the above two estimates in \eqref{eq:finally}
 completes the proof.
\end{proof}

\section{Lemmas on $L$-functions}

Let $f$ be an even, arithmetically normalized, weight 0 Hecke Maa{\ss} cusp form for $\text{SL}_2(\mathbb{Z})$ with Laplace eigenvalue $-(\tfrac14+t^2)$.
We will need rather sharp estimates for the following two moments, 
\begin{align*}
\sum_{|d| \leq X} L(\tfrac 12, f \otimes \chi_d)   \quad \text{ and }  \quad
\sum_{\substack{d_1, d_2 \\ a |d_1| \leq X \\ ad_1 = b d_2 + \ell  }} L(\tfrac 12, f \otimes \chi_{d_1})^{1/2} L(\tfrac 12, f \otimes \chi_{d_2})^{1/2}
\end{align*}
where the summation is over fundamental discriminants. 
Before stating our lemmas let us explain heuristically what to expect.
\subsection{Heuristics}
Heuristically for $ |d| \gg |t|^{\varepsilon}$ we expect 
$\log L(\tfrac 12, f \otimes \chi_d)$ to be approximated by
$$
\sum_{p^k \leq |d|^{\varepsilon}} \frac{(\alpha_p^k + \beta_p^k) \chi_d(p)^k}{k p^{k/2}}
$$
where $\alpha_p, \beta_p$ are the Satake parameters.  We also expect that as $d$ varies over fundamental discriminants $\chi_d(p)$ behaves on average as
an independent random variable $X(p)$ taking values $\pm 1$ with equal probability $\tfrac {1}{2(p + 1)}$ and the value $0$
with probability $\tfrac {1}{p + 1}$. 
We note that in the sum
the terms with $k \geq 3$ contribute $O(1)$, while for $k = 2$ we have $\alpha_p^2 + \beta_p^2 = \lambda_f(p^2) - 1$. 
Therefore the above sum heuristically behaves like
$$
\sum_{p \leq |d|^{\varepsilon}} \frac{\lambda_f(p) X(p)}{\sqrt{p}} + \frac{1}{2} \sum_{ p \leq |d|^{\varepsilon/2}} \frac{\lambda_f(p^2)}{p} - 
\frac{1}{2} \log\log |d| + O(1)
$$
which is the same as
$$
\sum_{p \leq |d|^{\varepsilon}} \frac{\lambda_f(p) X(p)}{\sqrt{p}} + \frac{1}{2} \log L(1, \tmop{Sym}^2 f) - \frac{1}{2} \log\log |d| + O(1). 
$$
In addition the sum over $ p \leq |d|^{\varepsilon}$ behaves approximately like a Gaussian random variable with mean $0$
and variance $\sum_{p \leq |d|^{\varepsilon}} \lambda_f(p)^2 / p = \log\log |d| + \log L(1, \tmop{Sym}^2 f) + O(1)$, 
since $\lambda_f(p)^2 = \lambda_f(p^2) + 1$.  Note that
we have to include $L(1, \tmop{Sym}^2 f)$ since \textit{a priori} its contribution could easily overwhelm
$\log\log |d|$ when $|t|^{\varepsilon} \leq |d| \leq t^2$ (this 
is due to our lack of knowledge of the Ramanujan-Petersson conjecture in the case of weight 0 Maa{\ss} forms). 
It follows from these considerations that we expect heuristically, in the range $|t|^{\varepsilon} \leq X$,  that
\begin{align*}
\sum_{X \leq |d| \leq 2X}  L(\tfrac 12, f \otimes \chi_d) & \asymp  X \cdot \mathbb{E} \Big [ \exp \Big ( \sum_{2 < p \leq X} \frac{\lambda_f(p) X(p)}{\sqrt{p}}
\Big ) \Big ] (\log X)^{-1/2} \cdot L(1, \tmop{Sym}^2 f)^{1/2} \\
&\asymp L(1, \tmop{Sym}^2 f) X
\end{align*}
using the fact that $\mathbb{E}[e^{\lambda X}] = \exp(\tfrac 12 \sigma^2 \lambda^2)$ for a Gaussian random variable
with variance $\sigma^2$ and mean $0$. In the shifted moment the same heuristic applies, but in addition we expect
$L(\tfrac 12, f \otimes \chi_{d_1})$ to behave approximately independently from $L(\tfrac 12, f \otimes \chi_{d_2})$. 
This leads to the following conjecture, for $1 \le  a,b, |\ell| < X^{1 - \varepsilon}$
\begin{align*}
& \sum_{\substack{ d_1,d_2 \\ X \leq a |d_1| \leq 2X \\ a d_1 = b d_2 + \ell }} 
 L(\tfrac 12, f \otimes \chi_{d_1})^{1/2}  L(\tfrac 12, f \otimes \chi_{d_2})^{1/2} \\
& \qquad \qquad \asymp 
\frac{X}{[a,b]} \ \Big [ \mathbb{E}  \exp \Big ( \frac 12 \sum_{p \leq X} 
\frac{\lambda_f(p)X(p)}{\sqrt{p}} \Big ) \Big ]^2 \cdot (\log X)^{-1/2} L(1, \tmop{Sym}^2 f)^{1/2} \\
& \qquad \qquad \asymp \frac{X}{[a,b]} \cdot L(1, \tmop{Sym}^2 f)^{3/4} \cdot (\log X)^{-1/4}.
\end{align*}

\subsection{Rigorous results under GRH} We establish the following two lemmas which combine Soundararajan's method for moments of $L$-functions along with our character sum estimate.  
\begin{lemma} \label{lem:shiftedmoment}
Assume the Generalized Riemann Hypothesis. Let $f$ be an even, arithmetically normalized, Hecke-Maa{\ss} eigencuspform with eigenvalue
$-(\tfrac 14 + t^2)$. Then, uniformly for $|t|^{\varepsilon} <  X $, $0 \neq |\ell| \leq \log X$, and $1 \le a,b\le (\log X)^{100}$, 
$$
\sum_{\substack{d_1,d_2 \\ a|d_1| \le X  \\ ad_1=bd_2+\ell}}
L(\tfrac 12, f \otimes \chi_{d_1})^{1/2} L(\tfrac 12, f \otimes \chi_{d_2})^{1/2} \ll_{\varepsilon} L(1, \tmop{Sym}^2 f)^{3/4+\varepsilon} \frac{X}{[a,b]} (\log X)^{-1/4+\varepsilon} ,
$$ 
where the summation is over fundamental discriminants $d_1,d_2$.
\end{lemma}

When using this lemma in combination with a Waldspurger type formula we will be dividing by $L(1, \tmop{Sym}^2 f)$. 
We therefore record the following simple (but useful!) observation.
\begin{lemma} \label{lem:grhL1bd}
Assume the Generalized Riemann Hypothesis. Let $f$ be an even, arithmetically normalized, Hecke-Maa{\ss} eigencuspform with eigenvalue
$-(\tfrac 14 + t^2)$. Then, 
$$
L(1, \tmop{Sym}^2 f) \gg \frac{1}{\log\log |t|}. 
$$
\end{lemma}
\begin{proof}
%

By Perron's formula, for $x \ge 1$
\[
 \frac{1}{2\pi i} \int_{1-ix \log |t|}^{1+i x \log |t|} \log \Big( L(s+1, \tmop{Sym}^2 f) \Big) \frac{x^s}{s} \, ds=\sum_{p \le x} \frac{\lambda_f(p^2)}{p} +O(1).
\]
Under GRH, $\log \Big( L(s+1, \tmop{Sym}^2 f) \Big)$ is analytic in the half-plane $\tmop{Re}(s) > -\tfrac12$.
Hence shifting contours to $\tmop{Re}(s)=-\tfrac12+\frac{1}{\log x}$, collecting a simple at $s=0$ with residue
equal to $\log L(1, \tmop{Sym}^2 f)$, and bounding the horizontal contours by $O(1)$ we see that
\begin{equation} \label{eq:perron}
\sum_{p \le x} \frac{\lambda_f(p^2)}{p}
=\log L(1, \tmop{Sym}^2 f)+O\Big(1 +x^{-1/2} \int_{-x \log |t|} ^{x \log |t|}\frac{\log (|t|+|\tau|) }{1+\tau} \, d\tau\Big).
\end{equation}
Since $\lambda_f(p^2) = \lambda_f(p)^2 - 1 \ge -1$ it follows that the LHS is at least
$-\log\log x + O(1)$. Choosing $x = (\log |t|)^{2 + \varepsilon}$ the claim follows. 
\end{proof}

We also need an upper bound for the first moment,
however Soundararajan's method only gives
\[
\sum_{|d| \le X} L(\tfrac12, f \otimes \chi_d) \ll L(1,\tmop{Sym}^2 f)^{1 + \varepsilon}
\cdot X (\log X)^{\varepsilon}
\]
where the summation is over fundamental discriminants.
This bound differs from the one predicted in the heuristic of the previous section by an extra factor of $L(1, \tmop{Sym}^2 f)^{\varepsilon} (\log X)^{\varepsilon}$ and the term $L(1, \tmop{Sym}^2 f)^{\varepsilon} $ 
is especially problematic in our range of interest $|t|^{\varepsilon} \leq X \leq t^2$ since even on GRH we only know that 
\begin{equation} \label{eq:lisbon}
L(1, \tmop{Sym}^2 f) \ll \exp \Big ( (\log |t|)^{1/8 + \varepsilon} \Big ),
\end{equation}
which is proved by Li \cite{Li}.
However the loss of $\varepsilon$ on the exponent $L(1, \tmop{Sym}^2 f)$ is due to a non-optimal
treatment of the small primes $p \ll (\log  |t|)^{2 + \varepsilon}$ in Soundararajan's method \cite{SoundMoments}. We
are in luck since small primes are analytically easy to deal with, and we remedy this loss
by a small refinement of Soundararajan's method which is more efficient on small primes. 
We also note that in principle the loss of $(\log X)^{\varepsilon}$ could have been also
averted by the use of Harper's \cite{Harper} refinement of Soundararajan's method. 
\begin{lemma} \label{lem:firstmoment}
Assume the Generalized Riemann Hypothesis. 
Let $f$ be an even,  arithmetically normalized Hecke-Maa{\ss} eigencuspform with eigenvalue $-(\tfrac 14 + t^2)$. Then,
uniformly in $|t|^{\varepsilon}<X$, 
$$
\sum_{|d| \le X} L(\tfrac 12, f \otimes \chi_d) \ll L(1, \tmop{Sym}^2 f) X (\log X)^{\varepsilon}
$$
where the summation is over fundamental discriminants.
\end{lemma}

For the proof of both Lemma \ref{lem:firstmoment} and Lemma \ref{lem:shiftedmoment} we will use a theorem of Chandee \cite{Fai} which is a generalization of an inequality first obtained by Soundararajan for the Riemann zeta-function in \cite{SoundMoments}. 
\begin{lemma} \label{lem:chandee} Assume the Generalized Riemann Hypothesis. 
Let $f$ be an even, arithmetically normalized Hecke-Maa{\ss} eigencuspform with eigenvalue $-(\tfrac 14 + t^2)$ for the full modular group. Let 
$d$ be a fundamental discriminant. Then,
$$ 
\log L(\tfrac 12, f \otimes \chi_d)
\leq \sum_{p^n \leq x} \frac{\chi_{d}(p^n) (\alpha_p^{n} + \beta_p^n)}{n p^{n(1/2 + 1/\log x)}} \cdot
\frac{\log (x/p^n)}{\log x} + O \Big ( \frac{\log (|d| + |t|)}{\log x} + 1 \Big ) 
$$
where $\alpha_p, \beta_p$ are the Satake parameters. 
\end{lemma}

\begin{proof}
See Theorem 2.1 of \cite{Fai}. 
\end{proof}
\begin{remark}
Note that $\alpha_p^{n} + \beta_p^{n} \in \mathbb{R}$ for all $p$ and $n \geq 1$, regardless of the truth of the
Ramanujan-Petersson conjecture. 
Since $\lambda_f(p) = \alpha_p + \beta_p$ and $\alpha_p \beta_p = 1$, it follows that $\alpha_p$ and $\beta_p$
are roots of $x^2 - \lambda_f(p) x + 1$. If $|\lambda_f(p)| > 2$ then the discriminant of this polynomial is positive, hence
its roots $\alpha_p$, $\beta_p$ are real, and therefore $\alpha_p^n + \beta_p^n$ is real. 
On the other hand if $|\lambda_f(p)| \leq 2$ then the roots are complex, lie on the unit circle and are conjugates of each other,
and hence $\alpha_p^{n} + \beta_p^{n} = 2 \cos(n \theta_p)$ for some $\theta_p \in \mathbb{R}$. 
%
\end{remark}

Additionally the proof of Lemma \ref{lem:firstmoment} and Lemma \ref{lem:shiftedmoment} will require the following very weak estimate for the second moment. This relies on the following lemma.
\begin{lemma} \label{lem:easybound}
  Assume the Generalized Riemann Hypothesis. Let $f$ be an even, arithmetically normalized Hecke-Maa{\ss} eigencuspform with eigenvalue $(-\tfrac 14 + t^2)$. Then, uniformly in $|t|^{\varepsilon} \leq X$, and $x \leq X$, 
$$
\sum_{p \leq x} \frac{\lambda_f(p)^2}{p} \ll (\log X)^{1/3}. 
$$
\end{lemma}
\begin{proof}
If $x \leq (\log |t|)^{32/23}$ then we apply the Kim-Sarnak bound $|\lambda_f(p)| \leq p^{7/64}$, to see that, 
$$
\sum_{p \leq x} \frac{\lambda_f(p)^2}{p} \ll x^{7/32+\varepsilon} \ll (\log |t|)^{7/23 +\varepsilon} \ll (\log X)^{1/3 }. 
$$
If $x > (\log |t|)^{32/23}$ then, by \eqref{eq:perron} and using Li's bound \eqref{eq:lisbon}, we see that, 
\begin{align*}
\sum_{p \leq x} \frac{\lambda_f(p)^2}{p}=\sum_{p \le x} 
\frac{1+\lambda_f(p^2)}{p} & \ll \log\log x + (\log |t|)^{1/8 + \varepsilon} + x^{-1/2+\varepsilon} (\log |t|)^{1 + \varepsilon}
\\ & \ll \log \log x+ (\log |t|)^{7/23 + \varepsilon} \ll (\log X)^{1/3 } 
\end{align*}
for $x \le X$, as needed.

\end{proof}

\begin{lemma} \label{lem:secondmomenttechnical}
Assume the Generalized Riemann Hypothesis. Let $f$ be an even, arithmetically normalized Hecke-Maa{\ss} eigencusp form with eigenvalue $-(\tfrac 14 + t^2)$. Then, uniformly in $|t|^{\varepsilon} \leq X$, 
$$
\sum_{|d| \leq X} L(\tfrac 12, f \otimes \chi_d)^2 \ll X \exp((\log X)^{1 / 3 + \varepsilon})
$$
where the summation is over fundamental discriminants. 
\end{lemma}
\begin{proof}[Proof of Lemma \ref{lem:secondmomenttechnical}]
Let
\[
\mathcal S(X,\mathcal V)=\{ |d| \le X
: \log L( \tfrac12, f \otimes \chi_d) > \mathcal V\}.
\]
Observe that
\begin{equation} \notag 
\begin{split}
\sum_{|d| \leq X} L(\tfrac 12, f \otimes \chi_d)= 
\int_{\mathbb R} e^{\mathcal V} \# \mathcal S(X,\mathcal V) \, d\mathcal V .
\end{split}
\end{equation}
Thus, it will be enough to show that for $(\log X)^{1/3 + \delta} < \mathcal{V}$, 
with $\delta > 0$ arbitrarily small but fixed, we have, 
\begin{equation} \label{equ:measbdtech}
\# \mathcal{S}(X, \mathcal{V}) \ll X \exp( - c \mathcal{V} \log \mathcal{V}),
\end{equation}
with the constant $c > 0$ depending at most on $\delta$. 
Note that this bound is vacuous once $\mathcal{V} > A \log X / \log \log X$ for some large enough constant $A > 0$, since
choosing $x = (\log |d| t)^{\varepsilon}$ in Lemma \ref{lem:chandee} shows that 
$\log L(\tfrac 12, f \otimes \chi_d) \ll \log X / \log\log X$. 
To obtain \eqref{equ:measbdtech} take $x = X^{1/(\varepsilon \mathcal{V})}$ . By Lemma \ref{lem:chandee},
$$
\log L(\tfrac 12, f \otimes \chi_d) \leq \sum_{p^n \leq x} \frac{\chi_{d}(p^n)  (\alpha_p^{n} + \beta_p^n)}{n p^{n (\tfrac 12 + 1/\log x)}} \cdot \frac{\log (x/p^n)}{\log x} + O \Big ( \frac{\log (|d| + |t|)}{\log x} + 1 \Big ) . 
$$
In the sum the $p^n$ with $n \geq 3$ contribute $O(1)$. On the other hand the contribution from the $p^2 \leq x$ is 
according to Lemma \ref{lem:easybound} less than, 
$$
\sum_{2 < p \leq \sqrt{x}} \frac{\lambda_f(p)^2 + 2}{p} \ll (\log X)^{1/3 + \delta / 2}. 
$$ 
We conclude that if 
$\log L(\tfrac 12, f \otimes \chi_d) > \mathcal{V}$, then
$$
\frac{1}{2} \mathcal{V} < \sum_{p \leq x} \frac{\lambda_f(p) \chi_d(p)}{p^{\frac12+1/\log x}} \frac{\log x/p}{\log x} .
$$
Using Markov's inequality and Lemma \ref{lem:multilinear}
it follows that the number of fundamental discriminants $d$ for which this holds
is
\begin{equation} \label{eq:markov1}
\ll  \frac{2^{2k}}{\mathcal{V}^{2k}} \sum_{|d| \leq X} \Big (  \sum_{p \leq x} \frac{\lambda_f(p) \chi_d(p)}{p^{\frac12+1/\log x}} \frac{\log x/p}{\log x} \Big )^{2k} \ll X \cdot \Big ( \frac{8k }{e \mathcal{V}^2} \sum_{p \le x} \frac{\lambda_f(p)^2}{p}\Big )^{k}.
\end{equation}
By Lemma \ref{lem:easybound}
the sum over primes on the RHS is
 $\ll (\log X)^{1/3 }$. 
Picking $k = \lfloor \varepsilon \mathcal{V} \rfloor$, we see that since $\mathcal{V} > (\log X)^{1/3 + \delta}$, there is
an absolute constant $c > 0$, such that
the RHS of \eqref{eq:markov1} is 
$
\ll X e^{- c \delta \mathcal{V} \log \mathcal{V}}
$
which establishes \eqref{equ:measbdtech}, thus giving the lemma. 
\end{proof}

We are now ready to prove the main results of this section, namely Lemma \ref{lem:shiftedmoment} and Lemma \ref{lem:firstmoment}. 

\begin{proof}[Proof of Lemma \ref{lem:shiftedmoment}] 


Using the moment calculation we now deduce the lemma
by the method of Soundararajan \cite{SoundMoments}, with appropriate modifications.
For fundamental discriminants $d_1,d_2$ write $\mathcal L_f(d_1,d_2)=L(\tfrac12,f\otimes \chi_{d_1})L(\tfrac12,f\otimes \chi_{d_2})$, 
 and set
\[
\mathcal S(X,\mathcal V)=\{ d_1, d_2: a|d_1|\le X, ad_1=bd_2+\ell,  \log \mathcal L_f(d_1,d_2)\ge \mathcal V\}.
\]
Let 
$$
\mu(f,X)=-\log \log X+\log L(1, \tmop{Sym}^2f),$$ 
$$\sigma^2(f,X)=2\log \log X+2\log L(1, \tmop{Sym}^2 f).$$
Observe that for $\mathcal{U} = \exp(\sqrt{\log X})$, 
\begin{align*}
\sum_{\substack{d_1, d_2 \\ a |d_1| \leq X \\ a d_1 = b d_2 + \ell \\ \mathcal{L}_f (d_1, d_2) > \mathcal{U}}} 
(\mathcal{L}_f (d_1, d_2))^{1/2} &  \leq \mathcal{U}^{-1/2} \sum_{\substack{d_1, d_2 \\ a |d_1| \leq X \\ a d_1 = b d_2 + \ell}} 
\mathcal{L}_f(d_1, d_2) \\ \ll & X \mathcal{U}^{-1/2} \cdot \exp((\log X)^{1/3 + \delta}) \ll \frac{X}{\log X}
\end{align*}
by Lemma \ref{lem:secondmomenttechnical}. 
On the other hand, 
\begin{equation} \notag
\begin{split}
\sum_{\substack{ d_1,d_2 \\ a |d_1| \le X \\ ad_1=bd_2+\ell \\ \mathcal{L}_f(d_1, d_2) \leq \mathcal{U}}} (\mathcal L_f(d_1,d_2))^{1/2} = & \frac 12 \cdot
\int_{-\infty}^{\sqrt{\log X}} e^{\frac12\cdot \mathcal V} \# \mathcal S(X,\mathcal V) \, d\mathcal V \\
\le& \frac{L(1,\tmop{Sym}^2 f)^{1/2}}{2 (\log X)^{1/2}} \int_{-\infty}^{2\sqrt{\log X}} e^{\frac12\cdot \mathcal V} \# \mathcal S(X,\mathcal V+\mu(f,X)) \, d\mathcal V.
\end{split}
\end{equation}
The portion of the integral over $\mathcal V \le \varepsilon \log \log X$ is bounded by
\[
\ll \frac{X}{[a,b]} (\log X)^{-1/2+\varepsilon} L(1,\tmop{Sym}^2 f)^{1/2},
\]
which is negligible since $L(1, \text{Sym}^2 f)^{1/2} \ll (\log X)^{\varepsilon} L(1, \text{Sym}^2 f)^{3/4}$ by Lemma \ref{lem:grhL1bd}.

Thus, our present lemma follows once we have shown that
\begin{equation} \label{eq:largedev}
 \# \mathcal S(X,\mathcal V+\mu(f,X))
\ll  \frac{X}{[a,b]} \cdot
\begin{cases}
\exp\left( -\frac{\mathcal V^2}{2\sigma^2(f,X)}(1+o(1))\right) & \text{ if }  \mathcal V \in I_1(X), \\
\exp(-c \mathcal V \log \frac{\mathcal V}{\sigma^2(f,X)}) & \text{ if } \mathcal V  \in I_2(X) ,
\end{cases}
\end{equation}
where $c>0$ is an absolute constant, $I_1(X)=(\varepsilon \log \log X, \varepsilon  \sigma^2(f,X) \log \sigma(f,X)]$
and $I_2(X)=(\varepsilon \sigma^2(f,X)\log \sigma(f,X), 2 \sqrt{\log X}]$. 

It remains to establish (\ref{eq:largedev}).
By Lemma \ref{lem:chandee}, 
\begin{align*} \label{eq:logbd}
\log \mathcal L_f(d_1,d_2)
\le  \sum_{p^n \le x} \frac{(\chi_{d_1}(p^n)+\chi_{d_2}(p^n))(\alpha_p^n +\beta_p^n)}{n p^{n(\frac12+1/\log x)}}
& \frac{\log (x/ p^n)}{\log x}\\ & +O\left( \frac{\log (|d_1d_2|+|t|)}{\log x}+1\right).
\end{align*}
In the sum over prime powers the terms with $n \ge 3$ contribute $O(1)$, which is negligible.
Using the Hecke relations the sum over prime squares is for $x>(\log |t|)^{2+\varepsilon}$
\begin{align*}
 \sum_{ p \leq \sqrt{x} }
\frac{(\chi_{d_1}(p^2)+\chi_{d_2}(p^2))(\lambda_f^2(p)-2)}{2 p^{1+2/\log x}} \frac{\log x/p^2}{\log x}  \le &  \sum_{p \le x} \frac{\lambda_f^2(p)}{p}-2 \log \log x+O\Big(\sum_{ p |d_1d_2} \frac{1}{p}\Big) \\
= &
\mu(f,x)+O(\log \log \log  (|d_1d_2|+e^e)),
\end{align*}
by \eqref{eq:perron}.
Hence, for $x>(\log |t|)^{2+\varepsilon}$
\begin{equation} \notag
\begin{split}
\log \mathcal L_f(d_1,d_2)-\mu(f,x)
\le & \sum_{2 < p \leq x} \frac{(\chi_{d_1}(p)+\chi_{d_2}(p))\lambda_f(p)}{p^{\frac12+1/\log x}} \frac{\log x/p}{\log x} \\
 &+O\left(\frac{\log (|d_1d_2|+|t|)}{\log x}+\log \log\log (|d_1d_2|+e^e) \right).
\end{split}
\end{equation}
Also, 
\[
\sum_{p|ab}\frac{|\lambda_f(p)|}{\sqrt{p}} < \omega(ab) =o(\log \log X) .
\]
since $a,b \leq (\log X)^{100}$. 
Combining the previous two observations and taking $x=X^{1/(\varepsilon \mathcal V)}$
it follows that for all $X$ large enough, 
\begin{equation} \label{eq:logbd2}
\begin{split}
\log \mathcal L_f(d_1,d_2)-\mu (f,X) \le & \sum_{\substack{2 < p \leq x \\ p \nmid ab}} \frac{(\chi_{d_1}(p)+\chi_{d_2}(p))\lambda_f(p)}{p^{\frac12+1/\log x}} \frac{\log x/p}{\log x}+\frac{\varepsilon}{2} \mathcal V.
\end{split}
\end{equation}

Let $z=x^{1/\log \log X}$ and note that if the LHS of \eqref{eq:logbd2} is at least $\mathcal V$
then 
\begin{equation} \label{eq:event1}
\Bigg| \sum_{\substack{2 < p \leq z \\ p \nmid ab}} \frac{(\chi_{d_1}(p)+\chi_{d_2}(p))\lambda_f(p)}{p^{\frac12+1/\log x}} \frac{\log x/p}{\log x}\Bigg| \ge \mathcal V(1-\varepsilon)
\end{equation}
or
\begin{equation}\label{eq:event2}
 \Bigg| \sum_{\substack{z < p \le x \\ p \nmid ab}} \frac{(\chi_{d_1}(p)+\chi_{d_2}(p))\lambda_f(p)}{p^{\frac12+1/\log x}} \frac{\log x/p}{\log x}\Bigg| \ge \frac{\varepsilon}{2} \mathcal V.
\end{equation}
Now we use Markov's inequality
and apply Lemma \ref{lem:moments12} to see that
for $0\ne |\ell| \le \log X$ the number of $d_1,d_2$ with $a|d_1|\le X$ and $ad_1=bd_2+\ell$, for which \eqref{eq:event1} holds is bounded by, for any $k \leq \varepsilon \mathcal V \log\log X$, 
\begin{equation} \label{eq:markoffbound}
\begin{split}
& \frac{1}{(\mathcal V(1-\varepsilon))^{2k}}
\sum_{ \substack{d_1,d_2 \\ a|d_1| \le X \\ ad_1=bd_2+\ell}}   \Bigg( \sum_{\substack{2 < p \leq z \\ p\nmid ab} }\frac{(\chi_{d_1}(p) +\chi_{d_2}(p))\lambda_f(p)}{p^{\frac12+1/\log x}} \frac{\log x/p}{\log x} \Bigg)^{2k}  \\
& \ll \frac{X (2k)!}{[a,b] (2 \mathcal V^2)^k k!} \left(\sigma^2(f,X)(1+o(1)) \right)^k 
\ll \frac{X}{[a,b]} \left( \frac{2k \sigma^2(f,X)(1+o(1))}{e \mathcal V^2} \right)^k .
\end{split}
\end{equation}
Note that the use of Lemma \ref{lem:moments12} is justfied, since for $|\ell| \leq \log X$, $$\sum_{p | \ell} \lambda_f(p)^2 / p \ll \sum_{p | \ell} p^{-2\delta} \ll (\log\log X)^{1 - 2\delta}$$ using the bound $|a(p)| \ll p^{1/2 - \delta}$ (where importantly we know that $\delta > 0$ following for example the works of \cite{Ramanujan}) 
while $\sum_{p \leq x} \lambda_f(p)^2 / p \gg \log\log x$  as a consequence of GRH using \eqref{eq:perron}, since $x=X^{1/(\varepsilon\mathcal V)}$ and $\mathcal{V} \leq 2 \sqrt{\log X}$.
 Also, we have $\log\log x \gg \log\log X$.

If $\mathcal V \leq \varepsilon \sigma^2(f, X) \log\log X$ then we choose $k = \lfloor \mathcal V^2 / ( 2 \sigma^2(f, X) ) \rfloor$ in (\ref{eq:markoffbound}), and 
in the remaining case we choose $k = \lceil \mathcal V/2 \rceil$. This leads to the following bound, 
\begin{equation} \label{eq:largedevbd1}
\begin{split}
\ll \frac{X}{[a,b]}
\begin{cases}
\exp\left(\frac{-\mathcal V^2(1+o(1))}{2 \sigma^2(f,X)} \right) & \text{ if }  \mathcal V \le \varepsilon  \sigma^2(f,X) \log \log X, \\
\exp\left(- \left(\frac{1}{2}+o(1)\right) \mathcal V \log \frac{\mathcal V}{\sigma^2(f,X)} \right) & \text{ if }   \varepsilon \sigma^2(f,X) \log \log X < \mathcal V \le 2\sqrt{\log X},
\end{cases}
\end{split}
\end{equation}
for the number of fundamental discriminants $d_1, d_2$ with $a |d_1| \leq X$, and $a d_1 = b d_2 + \ell$ for which (\ref{eq:event1})
holds.

It remains to bound the number of fundamental discriminants $d_1, d_2$, with $a |d_1| \leq X$, and $a d_1 = b d_2 + \ell$ for which  (\ref{eq:event2}) holds. Since $\mathcal V \le 2 \sqrt{\log X}$ it follows that $z=X^{1/(\varepsilon \mathcal V \log \log X)} > (\log |t|)^{2+\varepsilon}$. Hence, under GRH it follows from \eqref{eq:perron} that
\[
\sum_{z < p \le x} \frac{\lambda_f^2(p)}{p}= \sum_{p \leq x} \frac{\lambda_f(p)^2}{p} - \sum_{p \leq z} \frac{\lambda_f(p)^2}{p} = \log \frac{\log x}{\log z}+O(1 ).
\]
 We will now apply Markov's inequality and Lemma \ref{lem:moments12}. In Lemma \ref{lem:moments12} set $a(p) = 0$ for $p < z$ and $a(p) = \lambda_f(p)$ otherwise. This way the conditions of the lemma are trivially verified, 
since 
$\sum_{p | \ell} a(p)^2 / p = 0$ because $|\ell| \leq \log X$ according to our assumption. 
Thus applying Markov's inequality and Lemma \ref{lem:moments12} as in \eqref{eq:largedevbd1}
it follows that, for $\varepsilon \log \log X< \mathcal V< 2\sqrt{\log X}$,
 the number of $d_1,d_2$ with $a|d_1|\le X$ and $ad_1=bd_2+\ell$ for which
 \eqref{eq:event2} holds 
is
\begin{equation} \label{eq:largedevbd2}
\ll \frac{X}{[a,b]}\left(  \frac{2 k}{e (\varepsilon \mathcal V)^2} \log \log \log X(1+o(1))\right)^{k} \ll \frac{X}{[a,b]}\exp\left( -\varepsilon \mathcal V \log \mathcal V\right)
\end{equation}
where we chose $k =\lceil \varepsilon \mathcal V \rceil$. 
Combining  \eqref{eq:largedevbd1}, \eqref{eq:largedevbd2} gives \eqref{eq:largedev}, thereby completing the proof.


\end{proof}

Finally we are ready to prove Lemma \ref{lem:firstmoment}. 
\begin{proof}[Proof of Lemma \ref{lem:firstmoment}]
Given a parameter $x$, an integer $m$, and interval $I \subset [1, x]$, let 
$$
\mathcal P_I (m, x) = \sum_{p \in I} \frac{\lambda_f(p) \Big ( \frac{m}{p} \Big )}{p^{1/2 + 1/\log x}} \cdot \frac{\log (x/p)}{\log x}.
$$
Let $F$ be a Schwartz function with $\widehat F$ compactly supported in $(-10,10)$.
We assume that in addition $F$ satisfies
$F(x) \ge \mathbf 1_{[-1,1]}(x)$ for all $x \in \mathbb R$.
First, by Lemma \ref{lem:secondmomenttechnical} for $\mathcal{U} = e^{\frac12 (\log X)^{5/6}}$ we have
\begin{equation} \label{eq:smoothing1}
\sum_{\substack{|d| \leq X \\ \log L(\frac 12 , f \otimes \chi_d) > \mathcal{U} }} L(\tfrac 12, f \otimes \chi_d)
\ll \mathcal{U}^{-1} \sum_{\substack{|d| \leq X}} L(\tfrac 12, f \otimes \chi_d)^2 \ll X
\end{equation}
Secondly,
let $I = (2, (\log X)^5]$ and suppose that $x> e^{(\log X)^{1/6}}$. By Lemma \ref{lem:multilinear}
the number of 
fundamental discriminants $|d| \leq X$
for which $|\mathcal{P}_I(d;X)| \geq \frac12 (\log X)^{5/6} $ is for $k = \lfloor (\log X)^{1-\varepsilon} \rfloor$
\begin{equation} \label{eq:Pbound}
\begin{split}
 \le \frac{4^k}{(\log X)^{5k/3}} \sum_{m \in \mathbb Z} |\mathcal{P}_I(m;X)|^{2k} F\left( \frac{m}{X}\right) & \ll  X \frac{(2k)!}{k! 2^{k}} \cdot \Big ( 
\frac{4 }{(\log X)^{5/3} } \cdot \sum_{p \in I} \frac{\lambda_f(p)^2}{p} \Big )^{k}  \\
 & \ll X \Big ( \frac{8k (\log X)^{1/3} }{e(\log X)^{5/3}} \Big )^{k} \ll X e^{-(\log X)^{1-\varepsilon}}
\end{split}
\end{equation}
where we used
 Lemma \ref{lem:easybound} in the second bound.
It follows that, 
$$
\sum_{\substack{|d| \leq X \\ |\mathcal{P}_I(d, X)| >  \frac12 (\log X)^{5/6}}} L(\tfrac 12, f \otimes \chi_d) 
\leq \Big ( \sum_{\substack{|d| \leq X}}  L(\tfrac 12, f \otimes \chi_d)^2 \Big )^{1/2} \cdot  X^{1/2}\exp \Big ( - (\log X)^{1-\varepsilon} \Big ) \ll X
$$
on using Lemma \ref{lem:secondmomenttechnical}. 
Combining this along with \eqref{eq:smoothing1} we have that
\begin{equation} \label{eq:sizerestriction}
\sum_{|d| \leq X} L(\tfrac 12 , f \otimes \chi_d) = \sum_{\substack{|d| \leq X \\ \log L (\frac 12, f \otimes \chi_d) \leq \frac12 (\log X)^{5/6} \\ |\mathcal{P}_I (d,X)| \leq  \frac12 (\log X)^{5/6} }} L(\tfrac 12, f \otimes \chi_d) + O(X).
\end{equation}

Define $\mathcal L_f(d, X)=L(\tfrac12, f\otimes \chi_d) e^{-\mathcal P_I(d,X)}$ and
\[
\widetilde{\mathcal S} (X, \mathcal V)=\sum_{\substack{|d| \le X \\ \mathcal P_I(d,X) \le \frac12 (\log X)^{5/6} \\ \log \mathcal L_f(d,X)>\mathcal V}} e^{\mathcal P_I(d,X)}.
\]
Then
\begin{equation} \label{eq:fubini}
\begin{split}
\sum_{\substack{|d| \le X \\ \log L(\frac 12, f \otimes \chi_d) \leq \frac12 (\log X)^{5/6} \\ |\mathcal{P}_I(d,X)| \leq \frac12 (\log X)^{5/6}} } L(\tfrac12, f \otimes \chi_d) \le & \sum_{\substack{|d| \le X \\ \log \mathcal L_f(d,X) \leq (\log X)^{5/6} \\ |\mathcal{P}_I(d,X)| \leq \frac12 (\log X)^{5/6}} } \mathcal L_f(d,X) \cdot e^{\mathcal P_I(d,X)}
\\
=& \int_{-\infty}^{(\log X)^{5/6}} e^\mathcal V  \cdot \widetilde{ \mathcal S}(X,\mathcal V) d\mathcal V.
\end{split}
\end{equation}
In particular, it follows from this along with \eqref{eq:sizerestriction} that
$$
\sum_{|d| \leq X} L(\tfrac 12, f \otimes \chi_d) \leq \int_{-\infty}^{(\log X)^{5/6}} e^{\mathcal{V}} \cdot 
\widetilde{\mathcal S}(X, \mathcal V) d \mathcal V + O(X).
$$
Hence, it suffices to show that the above integral is $\ll  L(1,\tmop{Sym}^2 f)  X(\log X)^{\varepsilon}$.
Letting as usual $\widetilde \mu(f, X) = -\frac 12 \log\log X + \frac 12 \log L(1, \tmop{Sym}^2 f)$ it 
will suffice to show that
\begin{equation} \label{eq:goal}
\begin{split}
& \widetilde{\mathcal{S}}(X, \mathcal{V} + \widetilde \mu (f,X)) \ll 
\\ & \quad \quad X (\log X)^{\varepsilon} \cdot L(1, \tmop{Sym}^2 f)^{1/2} \cdot  
\begin{cases}
e^{-\frac{\mathcal{V}^2}{2 \log\log X} (1 + o(1))} & \text{ if } \varepsilon \log\log X \leq \mathcal{V} \leq \varepsilon (\log\log X)^2 \\
e^{-(\varepsilon + o(1)) \mathcal{V} \log \mathcal{V}} & \text{ if } \varepsilon (\log\log X)^2 \leq \mathcal{V} \leq ( \log X)^{5/6}.
\end{cases}
\end{split}
\end{equation}

Given $\varepsilon \log \log X \le \mathcal{V} \le 2(\log X)^{5/6}$ set $x = X^{1/(\varepsilon \mathcal{V})}(>e^{(\log X)^{1/6}})$ and $z = x^{1/\log\log X}$. 
Define also $J_1 = ((\log X)^{5}, z]$ and $J_2 = (z, x]$. 
Using Lemma \ref{lem:chandee} we see that for a fundamental discriminant $d$, GRH implies
\begin{equation}\notag
\log \mathcal L_f(d,X)-\widetilde \mu(f,X)
\le \mathcal P_{J_1}(d,x)+\mathcal P_{J_2}(d,x)+\left( \mathcal P_I(d,x)-\mathcal P_I(d,X)\right)+\frac{\varepsilon}{2} \mathcal V.
\end{equation}
If the LHS of the above inequality is $> \mathcal V$ then we must have at least one of the following events occur: $\mathcal P_{J_1}(d,X))> \mathcal V(1-\varepsilon)$, $\mathcal P_{J_2}(d,x)>\frac{\varepsilon}{4} \mathcal V$ or $\mathcal P_I(d,x)-\mathcal P_I(d,X) > \frac{\varepsilon}{4} \mathcal V$.  Let $\mathcal{V}_1 = (1 - \varepsilon) \mathcal{V}$ and $\mathcal{V}_2 =\mathcal V_3= \frac{\varepsilon}{4} \mathcal{V}$. Therefore, for any $k_1, k_2, k_3 \geq 1$, 
 $ \widetilde{\mathcal S}(X,\mathcal V+\widetilde \mu(f,X)) $ is bounded above by
\begin{align} \label{eq:upbd1}
& \sum_{\substack{m \\ |\mathcal P_I(m, X)| \le \frac12 (\log X)^{5/6} }}\left( \left( \frac{\mathcal  P_{J_1}(m,X)}{\mathcal V_1} \right)^{2k_1}
+\left( \frac{ \mathcal P_{J_2}(m)}{\mathcal {V}_2} \right)^{2k_2}\right) e^{ \mathcal P_I(m,X)} F\left( \frac{m}{X}\right) \\
&+ e^{\frac12 (\log X)^{5/6}} \sum_m \left( \frac{\mathcal P_I(d,x)-\mathcal P_I(d,X)}{\mathcal V_3}\right)^{2k_3} F\left( \frac{m}{X}\right) \label{eq:technicalterm}.
\end{align}
To estimate \eqref{eq:technicalterm} we note that $|p^{-1/\log x}\frac{\log p/x}{\log x}-p^{-1/\log X}\frac{\log p/X}{\log X}| \le C \mathcal V (\log X)^{-1+\varepsilon}$
for $p \le (\log X)^5$, and $C>1$ is a sufficiently large absolute constant. So
applying Lemmas \ref{lem:multilinear} and \ref{lem:easybound} we get for $k_3 = \lfloor (\log X)^{1-\varepsilon} \rfloor$
that \eqref{eq:technicalterm} is
\begin{equation} \label{eq:technicalbounded}
\begin{split}
\ll &
X \frac{e^{\frac12(\log X)^{5/6}}}{\mathcal V_3^{2k_3}} \frac{(2k_3)!}{2^{k_3} k_3!} \left(  \frac{C^2 \mathcal V^2}{(\log X)^{2-\varepsilon}}\sum_{ p \in I} \frac{\lambda_f(p)^2}{p}\right)^{k_3} \\
\ll & X e^{\frac12 (\log X)^{5/6}} \left( \frac{32C^2}{ \varepsilon^2 e(\log X)^{2/3-\varepsilon}} \right)^{(1+o(1))(\log X)^{1-\varepsilon}} \ll X e^{-(1+o(1))(\log X)^{1-\varepsilon}}.
\end{split}
\end{equation}

In order to bound \eqref{eq:upbd1} first observe that for $x \le N/e^2$, where $N$ is an even integer
$$
e^x \leq 2 \sum_{h = 0}^{N} \frac{x^h}{h!}.
$$
Thus setting $M = \frac12 (\log X)^{5/6}$, we have that \eqref{eq:upbd1} is
\begin{equation} \label{eq:truncated}
\begin{split}
& \leq 2 \sum_{i = 1}^2 \sum_{h \le 2 \lfloor e^2M \rfloor} \frac{1}{h!} \frac{1}{\mathcal {V}_i^{2k_i}}  \sum_{\substack{m \\ |\mathcal P_I(m,X)| \le M }} \mathcal P_I(m,X)^h \cdot \mathcal P_{J_i}(m,x)^{2k_i} F\left( \frac{m}{X}\right). \\
\end{split}
\end{equation}
We will now extend the inner sum  over $m$ to all integers. To do this, we observe that for $k_1,k_2 \le \varepsilon \mathcal V < M $
the contribution of an identical sum over terms $m$ with $|\mathcal{P}_I(m,x)| > M$ is by
Cauchy-Schwarz, Lemmas \ref{lem:multilinear} and \ref{lem:easybound}, and \eqref{eq:Pbound} seen to be
\begin{equation} \notag
\begin{split}
& \ll
X e^{-(\log X)^{1-\varepsilon}} \cdot \sum_{i=1,2} \sum_{h \le 2 \lfloor e^2M \rfloor} h^{h/2} \Big(\sum_{p \in I} \frac{\lambda_f(p)^2}{p} \Big)^{h/2} (2k_i)^{k_i} \Big(\sum_{p \in J_i} \frac{\lambda_f(p)^2}{p} \Big)^{k_i} \\
& \ll X e^{-(\log X)^{1-\varepsilon}} e^{M (\log M)^2}
\ll X e^{-(\log X)^{1-\varepsilon}}.
\end{split}
\end{equation}
Hence, we can extend the inner sum in \eqref{eq:truncated} to all of $m \in \mathbb Z$ at the cost of an error of size $X e^{-(\log X)^{1-\varepsilon}}$. 
Using
Lemma \ref{lem:multilinear} we evaluate the moments in the sum over all $m$ provided that 
$k_1 \leq \varepsilon \mathcal{V} \log\log X$ and $k_2 \leq \varepsilon \mathcal{V}$. Extending the summation
in the resulting main terms to all integers $h \ge 0$ we get the following bound for (\ref{eq:truncated}),
\begin{equation} \label{eq:extendedbd}
\begin{split}
& 2 \sum_{i=1,2}\sum_{h \le 2 \lfloor e^2M \rfloor} \frac{1}{h!} \frac{1}{\mathcal {V}_i^{2k_i}} \sum_{m} \mathcal P_I(m,X)^h \mathcal P_{J_i}(m,x)^{2k_i} F\left(\frac{m}{X}\right)+O(Xe^{- (\log X)^{1-\varepsilon}}) \\
&\ll  X \sum_{i=1,2} \left( \frac{2k_i}{e\mathcal V_i^2} \right)^{k_i} \left(\sum_{p \in J_i}
\frac{\lambda_f(p)^2}{p} \right)^{k_i}  
\sum_{h=0}^{\infty} \frac{1}{(2h)!} 
\cdot \frac{(2h)!}{2^{h}h!}\left(\sum_{p \in I}
\frac{\lambda_f(p)^2}{p} \right)^h+Xe^{-(\log X)^{1-\varepsilon}}.
\end{split}
\end{equation}
The sum over $I$ equals $\log L(1,\tmop{Sym}^2 f) +O(\log\log\log X)$ so that the sum over $h$ is $\ll L(1, \tmop{Sym}^2 f)^{1/2} (\log X)^{\varepsilon}$. 
We now specialize our choice of $k_1$ and $k_2$. 
Using the Hecke relations and applying \eqref{eq:perron}
$$
\sum_{p \in J_1} \frac{\lambda_f(p)^2}{p} \le \log\log X+O(1), \ \quad \ \sum_{p \in J_2} \frac{\lambda_f(p)^2}{p} = \log\log\log X+O(1).
$$
First for $k_1$, if $\varepsilon \log\log X \leq \mathcal{V} \leq \varepsilon (\log\log X)^2$ then choose
$k_1 = \lfloor \mathcal{V}^2 / (2 \log\log X) \rfloor$. While if $\mathcal{V} \geq \varepsilon (\log\log X)^2$ then choose
$k_1 = \lfloor \varepsilon \mathcal{V} / 2 \rfloor$. Secondly, for $k_2$ set $k_2 = \lfloor \varepsilon \mathcal{V} \rfloor$. 
This along with our earlier estimate for the sum over $h$ shows that for this choice of $k_1$ and $k_2$, the RHS of \eqref{eq:extendedbd} (and thus \eqref{eq:upbd1} as well) is bounded by
\begin{align*}
\ll 
 \begin{cases}
 L(1, \tmop{Sym}^2 f)^{1/2} X (\log X)^{\varepsilon} e^{-\frac{\mathcal V^2}{2 \log\log X}(1-o(1))} & \text{ if } \varepsilon \log \log X < \mathcal V < \varepsilon (\log \log X)^2, \\
 L(1, \tmop{Sym}^2 f)^{1/2} X (\log X)^{\varepsilon} e^{-(\varepsilon+o(1))\mathcal V\log \mathcal V} & \text{ if }
\varepsilon (\log \log X)^2 \le \mathcal V \le (\log X)^{5/6}.
\end{cases}
\end{align*}
 Therefore, using this bound for \eqref{eq:upbd1} along with the estimate \eqref{eq:technicalbounded} for \eqref{eq:technicalterm} gives \eqref{eq:goal}, which completes the proof.
\end{proof}

\section{Averages of coefficients of half-integral weight Maa{\ss} forms} \label{sec:smooth}

Let $g_j$ be a basis of $V^+$ consisting of simultaneous eigenfunctions of $\Delta_{1/2}$
and of the Hecke operators $T_{p^2}$, $p \neq 2$, and with each $g_j$ normalized so that $\iint_{\Gamma_0(4) \backslash \mathbb{H}} |g_j(z)|^2  \tmop{dvol}(z) = 1$. 
Write $g = g_j$ with $\Delta_{1/2} g = - (\tfrac 14 + t^2) g$. 
As usual $g(z)$ has a Fourier expansion of the form
$$
g(z) = \sum_{n \neq 0} b_{g,\infty}(n) W_{\tfrac 14 \tmop{sgn}(n), it}(4\pi |n|y)e(n x) , \qquad z = x + i y.
$$
In this section our goal is to estimate the following averages of the coefficients $b_{g,\infty}$, 
\begin{equation} \label{equ:coeffbound}
\sum_{X \leq n \leq 2X} |b_{g,\infty}(\pm n)|^2 \text{ and } \sum_{X \leq n \leq 2X} |b_{g,\infty}(\pm n) b_{g,\infty}(\pm n + \ell)| .
\end{equation}
We do this by relating the behavior of the coefficients to central values of $L$-functions through a 
Waldspurger type result which follows after collecting the result of Baruch and Mao \cite{BaruchMao}
and a recent result of Duke-Imamo$\overline{\text{g}}$lu-T\'oth \cite{DukeImamogluToth} (which strengthens the earlier work of Bir\'o \cite{Biro}, and Katok-Sarnak 
\cite{KatokSarnak}).
We will postpone the proof of the proposition until the appendix. We also recall the standard definition of the inner product on $\Gamma_0(N) \backslash \mathbb{H}$, 
$$
\langle f, g \rangle_{\Gamma_0(N)} := \frac{1}{[\tmop{PSL}_2(\mathbb{Z}) : \Gamma_0(N)]} \iint_{\Gamma_0(N) \backslash \mathbb{H}} f(z) \overline{g(z)} \cdot \frac{dx dy}{y^2} .
$$ 
\begin{proposition} 
\label{prop:shimura}
Let $g \in V^{+}$ with $g \neq 0$. Suppose that $g$ is a simultaneous eigenfunction of $\Delta_{1/2}$ and of the Hecke operators
$T_{p^2}$, $p \neq 2$. Then, there exists an even weight $0$ Maa{\ss} (Hecke normalized eigen-)form $f$ with 
$\Delta f = - (\tfrac 14 + (2t)^2) f$ 
such that for any fundamental discriminant $d$, 
$$
\frac{|b_{g,\infty}(d)|^2}{\langle g , g \rangle_{\Gamma_0(4)}} = \frac{L(\tfrac 12, f \otimes \chi_d)}{\langle f , f \rangle_{\text{SL}_2(\mathbb{Z})}} 
\cdot (\pi |d|)^{-1} \Big | \Gamma \Big ( \frac{1}{2} - \frac{\tmop{sgn}(d)}{4} - it \Big ) \Big |^2 .
$$
Additionally, for an integer $n = d \delta^2$ with $d$ a fundamental discriminant, 
$$
|b_{g,\infty}(n)|^2 \ll_{\varepsilon} |b_{g,\infty}(d)|^2 \cdot \delta^{2 \theta - 2 + \varepsilon}
$$
where $\theta$ is the best currently known exponent towards the Ramanujan conjecture in the case of
weight $0$ Maa{\ss} forms. 
\end{proposition}
Note that in our setting $\langle g, g \rangle_{\Gamma_0(4)}  = \tfrac 16$ since $g$ is normalized so that
$$
\iint_{\Gamma_0(4) \backslash \mathbb{H}} |g(z)|^2  \tmop{dvol}(z) = 1. 
$$
In addition since the corresponding form $f$ of weight $0$
is Hecke normalized, its $L^2$ norm is determined by the formula
$$
(\cosh 2 \pi t) \cdot \langle f , f \rangle_{\text{SL}_2(\mathbb{Z})}  \asymp  L(1, \tmop{Sym}^2 f) .
$$
Combining the above with Proposition \ref{prop:shimura} and a bound towards the Ramanujan conjecture
due to Kim and Sarnak \cite{KimSarnak}, we obtain the following corollary. 
\begin{corollary} \label{cor:bound} 
For $n = d \delta^2$ with $d$ a fundamental discriminant, we have
\begin{equation} \label{equ:coefficientbound}
|b_{g,\infty}(n)| \ll_{\varepsilon} \frac{1}{\sqrt{|n|}} \cdot \Big ( \frac{L(\tfrac 12, f \otimes \chi_d)}{L(1, \tmop{Sym}^2 f)} \Big )^{1/2} \cdot |\delta|^{7/64 + \varepsilon}
\cdot  |t|^{-\tmop{sgn}(n)/4} \cdot e^{\pi |t|/2}.
\end{equation}
\end{corollary}

The bound of Corollary \ref{cor:bound} reduces the problem of understanding (\ref{equ:coeffbound}) to corresponding questions on $L$-functions
which we have addressed in the previous section. In this way we obtain the following two lemmas.

\begin{lemma} \label{lem:secondmoment} Assume the Generalized Riemann Hypothesis.
Then for $|t|^{\varepsilon} \leq X$
$$
\sum_{X \leq n \leq 2X} |b_{g,\infty}(\pm n)|^2 \ll_{\varepsilon} (\log X)^{\varepsilon} \cdot |t|^{\mp \frac12}
\cdot e^{\pi |t|} .
$$
In addition for $X \leq |t|^{\varepsilon}$ we have the trivial bound $\ll |t|^{\mp \tfrac 12 + \varepsilon} \cdot e^{\pi |t|}$. 
\end{lemma}
The trivial bound in Lemma \ref{lem:secondmoment} is a fairly direct consequence of the Generalized Riemann Hypothesis and
Corollary \ref{cor:bound}. 
In the next lemma we estimate the second sum in \eqref{equ:coeffbound}. The proof builds in part on the argument used to establish Lemma \ref{lem:secondmoment}. We notice however
that a trivial application of Cauchy-Schwarz and Lemma \ref{lem:secondmoment} gives only a bound of size $|t|^{\mp \frac12} e^{\pi |t|} (\log X)^{\varepsilon}$ for the second sum in \eqref{equ:coeffbound}, which does not suffice for our argument. 

%
\begin{lemma} \label{lem:offdiagonal}
Assume the Generalized Riemann Hypothesis. Then, for $|t|^{\varepsilon} \leq X$, 
and $0 \neq |\ell| \leq \log X$, 
$$
\sum_{X \leq n \leq 2X} |b_{g,\infty}(\pm n) b_{g,\infty}( \pm n + \ell)| \ll_{\varepsilon} 
 |t|^{\mp \frac12} \cdot e^{\pi |t|}
\cdot 
(\log X)^{-25/156 + \varepsilon} .
$$
\end{lemma}


In Lemma \ref{lem:offdiagonal},
we expect that the correct exponent on the logarithm is $-\tfrac 14$. 
Note that  $\tfrac{25}{156} = \frac{1}{4} - 0.089 \ldots$ and that this weaker exponent is due to the lack of the
Ramanujan bound for weight $0$ Maa{\ss} forms. It is however quite possible that with a more involved argument the 
sharp exponent $-\tfrac 14$ can be reached. 

\begin{proof}[Proof of Lemma \ref{lem:secondmoment}]
Write $n = d \delta^2$ with $d$ a fundamental discriminant for simplicity we only consider
terms with $d>0$, as those with $d<0$ are handled in the same way and satisfy the analogous bound.
Let us first prove the ``trivial bound''. On the Generalized Riemann Hypothesis we have, 
$$
\frac{L(\tfrac 12, f \otimes \chi_d)}{L(1, \text{Sym}^2 f)} \ll |t d|^{\varepsilon}
$$
Therefore, by the bound of Corollary \ref{cor:bound}, 
\begin{align*}
& \sum_{X \leq n \leq 2X} |b_{g,\infty}(n)|^2 \ll \frac{|t|^{-\tfrac 12 + \varepsilon} \cdot e^{\pi |t|}}{X} \sum_{X \leq d \delta^2 \leq 2X}  |d|^{\varepsilon} \cdot |\delta|^{14/64 + \varepsilon} \\ & \quad \ll  X^{\varepsilon} \cdot \frac{|t|^{-1/2} e^{\pi|t|}}{X} \sum_{\delta \leq \sqrt{X}} 
|\delta|^{16/64} \cdot \frac{2X}{\delta^{2}} \ll X^{\varepsilon} |t|^{-1/2} e^{\pi|t|}. 
\end{align*}
When $X \leq |t|$ this of course implies the claimed trivial bound $|t|^{-1/2 + \varepsilon} e^{\pi|t|}$. 

Now let's consider the case when $X \geq |t|^{\varepsilon}$, in which case we can do better.
Inserting the bound (\ref{equ:coefficientbound}) we get
\begin{align*}
\sum_{X \leq n \leq 2X} |b_{g,\infty}(n)|^2 & \ll \frac{|t|^{- \frac12} \cdot e^{\pi |t|}}{X} \sum_{X \leq d \delta^2 \leq 2X} \frac{L(\tfrac 12, f \otimes \chi_{d}) }{L(1, \tmop{Sym}^2 f)}
\cdot \delta^{14/64 + \varepsilon}.
\end{align*}
Pick any $0 < \theta < \tfrac 12$. 
In the sum we split according to $\delta < X^{\theta}$ or $\delta > X^{\theta}$.
This gives
\begin{align} \label{eq:fourierbd11}
\frac{1}{X} \sum_{X \leq d \delta^2 \leq 2X} & \frac{L(\tfrac 12, f \otimes \chi_d)}{L(1, \tmop{Sym}^2 f)} \cdot \delta^{14/64 + \varepsilon} 
 \ll \frac{1}{X} \sum_{\delta < X^{\theta}} \delta^{14/64 + \varepsilon} \sum_{X / \delta^2 \leq d \leq 2X / \delta^2}
\frac{L(\tfrac 12, f \otimes \chi_d)}{L(1, \tmop{Sym}^2 f)} + \\ & \ \ \ \ + \frac{1}{X} \sum_{\delta > X^{\theta}} \delta^{14/64+\varepsilon} \sum_{d \leq 2X / \delta^2} \frac{L(\tfrac 12, f \otimes \chi_d)}{L(1, \tmop{Sym}^2 f)}. \label{eq:fourierbd12}
\end{align}
We now use Lemma \ref{lem:firstmoment} to bound the sum on the RHS of \eqref{eq:fourierbd11} in the following way
\[
\frac{1}{X} \sum_{\delta < X^{\theta}} \delta^{14/64 + \varepsilon} \sum_{X / \delta^2 \leq d \leq 2X / \delta^2}  \frac{L(\tfrac 12, f \otimes \chi_d)}{L(1, \tmop{Sym}^2 f)} \ll (\log X)^{\varepsilon} \sum_{\delta < X^{\theta}} \frac{\delta^{14/64+\varepsilon}}{\delta^2} \ll (\log X)^{\varepsilon}.
\]
To estimate the sum in \eqref{eq:fourierbd12} note that the Lindel\"of Hypothesis  implies $L(\tfrac 12, f \otimes \chi_d)/L(1, \tmop{Sym}^2 f) \ll (|t| d)^{\varepsilon}$. This gives
$$
\frac{1}{X} \sum_{\delta > X^{\theta}} \delta^{14/64+\varepsilon} \sum_{d \leq 2X / \delta^2} \frac{L(\tfrac 12, f \otimes \chi_d)}{L(1, \tmop{Sym}^2 f)}
\ll  X^{\varepsilon} \sum_{\delta>X^{\theta}} \frac{\delta^{14/64+\varepsilon}}{\delta^2}\ll (\log X)^{-100}.
$$
Using this estimate in \eqref{eq:fourierbd12} completes the proof.
\end{proof} 

\begin{proof}[Proof of Lemma \ref{lem:offdiagonal}]
Again let us write, $n = d_0 \delta_0^2$ and $n + \ell = d_1 \delta_1^2$
and as before we will only consider the case $d_1,d_2>0$.  
We first show that the contribution of those $n$ for which $\delta_0 > (\log X)^{\kappa}$
or $\delta_1 > (\log X)^{\kappa}$ is negligible, for some $0 < \kappa < 10$ to be fixed later. 
Indeed, by Cauchy-Schwarz, and Lemma \ref{lem:secondmoment}, the contribution of those integers
for which $\delta_0 > (\log X)^{\kappa}$ is bounded by 
\begin{align} \label{eq:cauchy}
\Big ( \sum_{\substack{ X \leq d \delta^2 \leq 2X  \\ 
\delta > (\log X)^{\kappa}}} |b_{g,\infty}(d \delta^2)|^2 \Big )^{1/2} \cdot \Big ( \sum_{X/2 \leq n \leq 3X}
|b_{g,\infty}(n)|^2 \Big )^{1/2} .
\end{align}
By Lemma \ref{lem:secondmoment} the second term is
$ \ll |t|^{-1 / 4} \cdot e^{\frac{\pi}{2}|t|} \cdot (\log X)^{\varepsilon}$. On the other hand using (\ref{equ:coefficientbound}) and splitting into
two ranges $(\log X)^{\kappa} \leq \delta < X^{\theta}$ and $X^{\theta} \leq \delta \leq \sqrt{2X}$, with any fixed $0<\theta < \tfrac12$ , we see
that the sum in the first term of \eqref{eq:cauchy} is bounded by
\begin{align*}
\frac{|t|^{-\frac12} \cdot e^{\pi |t|}}{X} & \sum_{(\log X)^{\kappa} < \delta < X^{\theta}} \delta^{14/64 + \varepsilon} 
\sum_{\substack{X/\delta^2 \leq d \leq 2X / \delta^2}}
\frac{L(\tfrac 12, f \otimes \chi_d)}{L(1, \tmop{Sym}^2 f)} \\ & + \frac{|t|^{-\frac12} \cdot e^{\pi |t|}}{X} \sum_{\delta > X^{\theta}} \delta^{14 / 64 + \varepsilon}
\sum_{X / \delta^2 \leq d \leq 2X / \delta^2} \frac{L(\tfrac 12, f \otimes \chi_d)}{L(1, \tmop{Sym}^2 f)}. 
\end{align*}
Proceeding as in the proof of Lemma \ref{lem:secondmoment} we use Lemma \ref{lem:firstmoment} to bound the first term above
and the Lindel\"of Hypothesis bound $L(\tfrac 12, f \otimes \chi_d) \ll (d |t|)^{\varepsilon}$ to bound the second
term. This gives the following total bound for the sum in the first term of \eqref{eq:cauchy}
$$
\ll |t|^{-\frac12} \cdot e^{\pi |t|} \Bigg( (\log X)^{\varepsilon} \sum_{\delta > (\log X)^{\kappa}} \frac{\delta^{14 / 64 + \varepsilon}}{\delta^2}
+ (\log X)^{-100} \Bigg) \ll |t|^{-\frac12} \cdot e^{\pi |t|} \cdot (\log X)^{-25\kappa /32 + \varepsilon}.
$$
Combining estimates, we bound \eqref{eq:cauchy} as $ O(|t|^{-1/2} e^{\pi |t|} (\log X)^{-25 \kappa/64+\varepsilon})$.
Similarly we can bound the contribution of those integers $n$ for which $n + \ell = d_1 \delta_1^2$
with $d$ fundamental, and $\delta_1 > (\log X)^{\kappa}$. 

Therefore, 
\begin{align*}
& \sum_{X \leq n \leq 2X}  |b_{g,\infty}(n) b_{g,\infty}(n + \ell)|   \\
& \ll  |t|^{-\frac12} \cdot e^{\pi |t|} 
\Bigg(\frac{(\log X)^{7\kappa/32+\varepsilon}}{X}\sum_{\substack{X/2 \leq d_0 \delta_0^2 \leq 3X \\ d_0 \delta_0^2 - d_1 \delta_1^2 = -\ell \\ 1 \le \delta_0, \delta_1 \leq
(\log X)^{\kappa}}}  \frac{L(\tfrac 12, f \otimes \chi_{d_0})^{1/2} L(\tfrac 12, f \otimes \chi_{d_1})^{1/2}}{L(1, \tmop{Sym}^2 f)} \\ & \quad \quad \quad \quad +(\log X)^{-25 \kappa/64+\varepsilon}  \Bigg).
\end{align*}
Applying 
 Lemmas \ref{lem:shiftedmoment} and \ref{lem:grhL1bd} the above sum is 
$$
\ll X  (\log X)^{- 1/4+\varepsilon} \sum_{1 \le \delta_0,\delta_1  \le (\log X)^{\kappa}} \frac{1}{[\delta_0,\delta_1]^2}
\ll X  (\log X)^{- 1/4+\varepsilon}.
$$
 We conclude that
$$
\sum_{X \leq n \leq 2X} |b_{g,\infty}(n) b_{g,\infty}(n + \ell)| \ll |t|^{-\frac12} \cdot e^{\pi |t|} \cdot ( (\log X)^{- 25 \kappa /64} +
(\log X)^{7 \kappa /32 - 1/4} ) \cdot (\log X)^{\varepsilon}.
$$
Choosing $\kappa = \tfrac {16}{39}$ we get $\ll |t|^{-\frac12} \cdot e^{\pi |t|} (\log X)^{-25/156 + \varepsilon}$. 
\end{proof}

\section{Extending the length of summation}

As before let $g$ be an element of a basis $g_j$ of $V^+$
consisting of simultaneous eigenfunctions of $\Delta_{1/2}$ and $T_{p^2}$, $p \neq 2$,  
with each $g_j$ normalized so that $\iint_{\Gamma_0(4) \backslash \mathbb{H}} |g_j(z)|^2  \tmop{dvol}(z) = 1$
and denote by $b_{g,\infty}(n)$ the Fourier coefficients of $g$. As usual, we write
the eigenvalue of $g$ as $\lambda=-(\tfrac 14 + t^2)$. 

Following an idea of Holowinsky \cite{Holowinsky} we show in this section that the average of
$|b_{g,\infty}(n)|^2$ over integers $n \asymp |t|$ can be related to a corresponding
average of $|b_{g,\infty}(n)|^2$ over integers $n \asymp |t| Y$ with $Y$ a parameter
whose size roughly depends on the saving in the shifted convolution problem, 
$$
\sum_{n \asymp |t|} b_{g,\infty}(n) b_{g,\infty}(n + \ell)
$$
in the range $n \asymp |t|$.  The advantage of such an ``extension'' of the length of summation will become
clear in the next section, where we will show that the average of $|b_{g,\infty}(n)|^2$ can be estimated as soon
as we sum slightly more than $|t|$ terms. 

\begin{lemma} \label{lem:extend}
Let $h, k$ be two smooth compactly supported functions on $\mathbb R_{+}$ and with $k$ non-negative. Let $H(s) := \int_0^{\infty} h(x) x^{s-1} dx$ and $K(s) := \int_{0}^{\infty} k(x) x^{s-1} \, dx$
denote their respective Mellin transforms.
 Then, for $Y \ge 1$
\begin{align} \label{equ:extend}
K(-1) Y \mathcal{S}_{1}(0, t; h) = & H(-1) (1  + O(Y^{-1/2})) \mathcal{S}_Y(0, t; k) + O(\sqrt{Y}) \\ & \nonumber + O_A \Big ( \frac{1}{\sqrt{Y}} \sum_{\ell \neq 0} \frac{(\nu(\ell)+1) 2^{\nu(\ell)} d(|\ell|)}{1 + |\ell / Y|^A} 
\cdot \mathcal{S}_Y(\ell, t; k) \Big )
\end{align}
where for a smooth, compactly supported function $\psi$, the quantity $\mathcal{S}_Y(\ell, t; \psi)$ is defined as
\begin{align*}
\int_{0}^{\infty} \psi(y Y) \Big | \sum_{n \neq 0, -\ell} b_{g,\infty}(n) b_{g,\infty}(n + \ell) W_{\tfrac 14 \tmop{sgn}(n), it} 
(4\pi |n| y) W_{\tfrac 14 \tmop{sgn}(n+\ell), \text{it}} (4\pi |n + \ell| y) \Big | \cdot \frac{dy}{y^2} 
\end{align*}
and $\nu(\ell)$ is the highest power of two dividing $\ell$. 
\end{lemma}
Notice here that when $\ell = 0$ the term appearing inside the absolute value in $\mathcal{S}_Y(0, t; \psi)$ is
non-negative, since $W_{\kappa, it}(x) \in \mathbb{R}$ for $|\kappa| < \tfrac 12$ and $t \in \mathbb{R}$. 
The proof of Lemma \ref{lem:extend} is rather involved so we explain now the general principle behind its proof. 
Consider the incomplete Eisenstein series
\begin{equation} \label{eq:incompleteEisenstein1}
E(z | h) = \sum_{\gamma \in \Gamma_{\infty} \backslash \Gamma_0(4)} h(\tmop{Im} (\gamma z))
\end{equation}
and
\begin{equation} \label{eq:incompleteEisenstein2}
E^Y(z | k) = \sum_{\gamma \in \Gamma_{\infty} \backslash \Gamma_0(4)} k(Y \cdot \tmop{Im} (\gamma z))
\end{equation}
where $\Gamma_{\infty}$ is the stabilizer group of the cusp at infinity. In the proof of Lemma \ref{lem:extend}
we will evaluate
\begin{equation} \label{equ:start}
\iint_{\Gamma_0(4)\backslash \mathbb H} E(z | h) E^Y(z | k) |g(z)|^2 \tmop{dvol}(z)
\end{equation}
in two different ways. First, we will express $E^Y(z | k)$ as a Perron integral involving the real analytic Eisenstein
series $E(z,s)$, which we recall is defined as
$$
E(z,s) = \sum_{\gamma \in \Gamma_{\infty} \backslash \Gamma_0(4)} (\tmop{Im} (\gamma z))^s.
$$ 
Shifting contours and
collecting the residue from the simple pole of $E(z,s)$ at $s = 1$ will lead to the main term on the left-hand side of (\ref{equ:extend}).
Second, we will use the unfolding technique with $E^Y(z | k)$ and then expand $E(z | h)$ as a Fourier series. This will lead to the right-hand side of (\ref{equ:extend}).
The Fourier development is described in the next lemma.
\begin{lemma} \label{lem:fourier}
We have
$$
E(z|h) = a_{0,h}(y) + \sum_{\ell \neq 0} a_{\ell,h}(y) e(\ell x)
$$
where
$$
a_{\ell,h}(y) = \begin{cases}
\frac{H(-1)}{2\pi}  + O(\sqrt{y}) & \text{ if } \ell = 0, \\
O_A \Big(\sqrt{y} \frac{(\nu(\ell)+1) 2^{\nu(\ell)} d(\ell)}{1 + |\ell y|^A} \Big)& \text{ if } \ell \neq 0,
\end{cases}
$$
and $\nu(\ell)$ is the largest power of $2$ dividing $\ell$. 
\end{lemma}
\begin{proof}
The proof closely follows the argument of Holowinsky \cite{Holowinsky},
with appropriate modifications at the prime $p=2$.
Consider the Eisenstein series 
\[
E(z,s)=\sum_{ \gamma \in \Gamma_{\infty} \backslash \Gamma_0(4)} (\tmop{Im}(\gamma z))^s.
\]
At the cusp at $\infty$, $E(z,s)$ 
has the Fourier expansion 
\[
E(z,s)=
y^s+\varphi(s) y^{1-s}+2\sqrt{y}\sum_{\ell \neq 0} \varphi(s,\ell) K_{s-\frac12}(2\pi |\ell| y) e(n x),
\]
here $K_{\nu}(u)$ is the modified Bessel function of second kind and for $\tmop{Re}(s)>\tfrac12$
\[
\varphi(s)=\pi^{1/2}\frac{\Gamma(s-\tfrac12)}{4^{2s}\Gamma(s)}\sum_{n \ge 1} \frac{\phi(4 n )}{n^{2s}}= \pi^{1/2}\frac{\Gamma(s-\tfrac12)}{4^{2s} \Gamma(s)}\frac{2}{1-2^{-2s}} \frac{\zeta(2s-1)}{\zeta(2s)}
\]
and
\[
\varphi(s,\ell)=\frac{\pi^s}{4^{2s} \Gamma(s)} |\ell|^{s-\frac12} \sum_{n \ge 1} \frac{c_{4 n }(\ell)}{n^{2s}}
=\frac{\pi^s }{4^{2s}\Gamma(s)}  \frac{\mathcal L_2(s,\ell)}{\zeta(2s)}  \sum_{ab=|\ell|}\left(\frac{a}{b}\right)^{s-\frac12}
\]
where $c_q(n) $ is a Ramanujan sum
 (see \cite[Section 2.2]{Kubota} and \cite[Section 4.2]{TopicsIwaniec}). 
Writing
$\nu(\ell)$ for the largest power of $2$ dividing $\ell$
note that $\mathcal L_2(s,\ell)$ is given by 
\[
\mathcal L_2(s, \ell)=\frac{1-2^{1-2s}}{(1-2^{(\nu(\ell)+1)(1-2s)})(1-2^{-2s})} \sum_{j=0}^{\nu(\ell)} \frac{c_{2^{j+2}}(\ell)}{2^{2js}}=O((\nu(\ell)+1) 2^{\nu(\ell)})
\]
which holds uniformly for $\delta \le \tmop{Re}(s) \le 1$, for any $\delta>0$.

To obtain the Fourier series expansion of the incomplete
Eisenstein series $E(z|h)$ we write
\[
E(z|h)=\frac{1}{2\pi i} \int_{(2)} E(z,s) H(-s) \, ds
=a_0(y)+\sum_{\ell \neq 0} a_{\ell,h} (y) e(\ell x).
\]
By shifting contours of integration it follows that
\[
a_{0,h}(y)=\frac{1}{2\pi i} \int_{(2)} (y^s+\varphi(s)y^{1-s}) H(-s) \, ds
=\frac{H(-1)}{2\pi}+O(\sqrt{y}).
\]
To estimate $a_{\ell,h}(y)$ we use the 
bounds 
$
K_{it}(u) \ll \min( u^{-1/2} e^{-u}, e^{-\frac{\pi}{2}|t|})
$
(see \cite[Corollary 3.2]{GhoshReznikovSarnak}) and $H(-s) \ll (1+|s|)^{-A}$, 
to get that 
\begin{equation} \notag
\begin{split}
a_{\ell, h}(y) =&\frac{\sqrt{y}}{2\pi i} \int_{(2)} K_{s-\frac12}(2\pi |\ell| y)  \frac{\pi^s }{4^{2s}\Gamma(s)}  \frac{\mathcal L_2(s,\ell)}{\zeta(2s)}  \sum_{ab=|\ell|}\left(\frac{a}{b}\right)^{s-\frac12} H(-s) \, ds \\
\ll& \sqrt{y} (\nu(\ell)+1) 2^{\nu(\ell)} d(|\ell|) \int_{\mathbb R} \left|\frac{K_{i\tau}(2\pi|\ell|y)}{\Gamma(\frac12+i\tau)}\right|  (1+|\tau|)^{-A} \, d\tau \\
 \ll & \frac{\sqrt{y} (\nu(\ell)+1) 2^{\nu(\ell)} d(|\ell|) }{1 + |\ell y|^A},
\end{split}
\end{equation}
which gives the claim.
\end{proof}
We are now ready to prove Lemma \ref{lem:extend}.
As promised in order to arrive at (\ref{equ:extend}) 
we will evaluate (\ref{equ:start}) in two different ways.
\subsection{Proof of Lemma \ref{lem:extend}
\label{sec:firsteval}  Step 1: A first evaluation of (\ref{equ:start}).}
Note that, 
$$
E^Y(z | k) = \frac{1}{2\pi i} \int_{(2)} K(-s) Y^s E(z,s) ds
$$
with $K(s) := \int_{0}^{\infty} k(x) x^{s-1} dx$. Therefore
we can re-write (\ref{equ:start}) as 
$$
\frac{1}{2\pi i} \int_{(2)} K(-s)Y^s \Big ( \iint_{\Gamma_0(4)\backslash \mathbb H} E(z,s) E(z | h) |g(z)|^2 \cdot \frac{dx dy}{y^2} \Big ) ds.
$$
Shifting the contour to $\tmop{Re} (s) = \tfrac 12$ we collect a residue at $s = 1$ coming from $E(z,s)$. The 
value of the residue of $E(z,s)$ at $s = 1$ is $(\tmop{vol}(\Gamma_0(4)\backslash \mathbb H) )^{-1}= (2\pi )^{-1}$. 
Therefore we get that \eqref{equ:start} equals
\begin{align} \notag
K(-1) & \cdot \frac{Y}{2\pi} \iint_{\Gamma_0(4)\backslash \mathbb H} E(z | h) |g(z)|^2 \cdot \frac{dx dy}{y^2}  \\
& + \frac{1}{2\pi i}  \int_{(1/2)} K(-s)Y^s \Big ( \iint_{\Gamma_0(4)\backslash \mathbb H} E(z,s) E(z | h) |g(z)|^2 \cdot \frac{dx dy}{y^2} \Big ) ds. \label{eq:contourshift}
\end{align}
In the first term we apply the unfolding method, 
\begin{equation}\label{eq:unfoldedint}
\iint_{\Gamma_0(4)\backslash \mathbb H} E(z | h) |g(z)|^2 \cdot \frac{dx dy}{y^2} = \sum_{n \neq 0} |b_{g,\infty}(n)|^2
\int_{0}^{\infty} W_{\tfrac 14 \tmop{sgn}(n), it} (4\pi |n| y)^2 \cdot \frac{h(y) dy}{y^2}.
\end{equation}
In the second term, since $E(z | h)$ is compactly supported in $\Gamma_0(4) \backslash \mathbb{H}$, the $z$ variable is restricted to a compact set. 
Therefore we have the bound
$E(z,\tfrac 12 + iu) \ll (1 + |u|)^N$ for some large $N$, uniformly in the compact set to which $z$ belongs, with the 
implied constant depending only on $h$. Therefore, since $K(-\tfrac 12 - iu) \ll_A (1 + |u|)^{-A}$ 
the integral  in \eqref{eq:contourshift} is
$$
\ll \sqrt{Y} \iint_{\Gamma_0(4)\backslash \mathbb H} |E(z|h)| |g(z)|^2 \cdot \frac{dx dy}{y^2}  \ll
\sqrt{Y} \iint_{\Gamma_0(4)\backslash \mathbb H} |g(z)|^2 \cdot \frac{dx dy}{y^2} = \sqrt{Y}. 
$$
Combining this with \eqref{eq:contourshift} and (\ref{eq:unfoldedint}) we conclude that
\begin{align*}
& \iint_{\Gamma_0(4)\backslash \mathbb H} E^Y(z | k) E(z | h) |g(z)|^2 \cdot \frac{dx dy}{y^2}\\ 
 & \qquad \qquad =  
K(-1) \frac{Y}{2\pi}  
\sum_{n \neq 0} |b_{g,\infty}(n)|^2 \int_{0}^{\infty} W_{\tfrac 14 \tmop{sgn}(n), it}(4\pi |n| y)^2 \cdot \frac{h(y) dy}{y^2}+
 O(\sqrt{Y}).
\end{align*}
This gives the left-hand side of (\ref{equ:extend}). 
\subsection{Proof of Lemma \ref{lem:extend}. Step 2: A second evaluation of (\ref{equ:start})} \label{sec:secondeval}
By the unfolding method (\ref{equ:start}) is equal to
$$
\int_0^{\infty}
\int_{-1/2}^{1/2} k(Y y) E(z | h) |g(z)|^2 \cdot \frac{dx dy}{y^2} .
$$
Expanding $E(z | h)$ into the Fourier series given in Lemma \ref{lem:fourier} we see that the previous equation  equals
\begin{align}
\label{equ:2main}  \int_{0}^{\infty} &\int_{-1/2}^{1/2}  a_{0,h}(y) k(Y y) |g(z)|^2 \cdot \frac{dx dy}{y^2} \\
& \label{equ:2off} + \sum_{\ell \neq 0} \int_{0}^{\infty} \int_{-1/2}^{1/2} e(\ell x) k(Y y) a_{\ell,h}(y) |g(z)|^2 
\cdot \frac{dx dy}{y^2}.
\end{align}
We now investigate (\ref{equ:2main}) and (\ref{equ:2off}).  By
Lemma \ref{lem:fourier} the contribution of (\ref{equ:2main}) equals
$$
\frac{H(-1)}{2\pi}  \int_{0}^{\infty} \int_{-1/2}^{1/2} |g(z)|^2 k(Y y) \cdot \frac{dx dy}{y^2} + 
O \Big (  \int_{0}^{\infty} \int_{-1/2}^{1/2}\sqrt{y} k(Y y) |g(z)|^2 \cdot \frac{dx dy}{y^2} \Big ) .
$$
Since $k$ is compactly supported we may bound the term $\sqrt{y}$ in the integrand in the second term above as $O(Y^{-1/2})$. It follows from this and the non-negativity of $k$ that the above is equal to
\begin{align*}
( 1 + O(Y^{-1/2}) )  \frac{H(-1)}{2\pi}  \int_{0}^{\infty} \int_{-1/2}^{1/2} |g(z)|^2 k(Y y) \cdot \frac{dx dy}{y^2} .
\end{align*}
Upon expanding $g(z)$ into a Fourier series we see that the above integral equals
\begin{align*}
\frac{H(-1)}{2\pi} \sum_{n \neq 0} |b_{g,\infty}(n)|^2 \cdot\int_{0}^{\infty}
W_{\tfrac 14 \tmop{sgn}(n), it}(4\pi |n| y)^2 \cdot \frac{k(Y y) dy}{y^2} .
\end{align*}
This gives the main term on the right-hand side of \eqref{equ:extend}.

It remains to estimate \eqref{equ:2off}. First,
we expand $g(z)$ into a Fourier series once again to see that (\ref{equ:2off}) is equal to
\begin{align*}
\sum_{\ell \neq 0} \sum_{n \neq 0, -\ell} b_{g,\infty}(n) b_{g,\infty}(n + \ell)\int_{0}^{\infty} k(Y y) a_{\ell,h}(y) W_{\tfrac 14 \tmop{sgn}(n), it}(4\pi |n| y) W_{\tfrac 14 \tmop{sgn}(n + \ell), it}(4\pi |n + \ell|y ) \cdot \frac{dy}{y^2} .
\end{align*}
Moving the sum over $n$ inside, taking absolute values, and applying Lemma \ref{lem:fourier} we see that the above is 
\begin{align*}
\ll \sum_{\ell \neq 0} & (\nu(\ell)+1) 2^{\nu(\ell)} d(|\ell|)\int_{0}^{\infty} k(y Y) \cdot \frac{\sqrt{y}}{1 + |\ell y|^A}  
\times \\ &  \Big | \sum_{n \neq 0, -\ell} b_{g,\infty}(n) b_{g,\infty}(n + \ell)  W_{\tfrac 14 \text{sgn}(n), it} (4\pi |n| y) 
W_{\tfrac 14 \text{sgn}(n), it} (4\pi |n + \ell| y) \Big | \cdot \frac{dy}{y^2}.
\end{align*}
Since $k(\cdot) $ is compactly supported, in the integrand we may bound the term $\sqrt{y} / ( 1 + |\ell y|^A)$ as $O(Y^{-1/2} \cdot ( 1 + |\ell / Y|^A)^{-1})$.
This along with the estimate for \eqref{equ:2main} gives the right-hand side of (\ref{equ:extend}) and completes the proof 
of Lemma \ref{lem:extend}.

\section{The summation formula and proof of QUE}
First recall the standard notation which will be set in place for the duration of this section:
Let $g_j$ be a basis of $V^+$
consisting of simultaneous eigenfunctions of $\Delta_{1/2}$ and $T_{p^2}$, $p \neq 2$,  
with each $g_j$ normalized so that $\int_{\Gamma_0(4) \backslash \mathbb{H}} |g_j(z)|^2 d \tmop{vol}(z) = 1$.
For simplicity we will write $g = g_j $ and $t = t_j$ so that $\Delta_{1/2} g = - (\tfrac 14 + t^2) g$. 
Recall that $t \in \mathbb{R}$ since there are no forms corresponding to exceptional
eigenvalues in $V^+$.  

\subsection{Estimates for Whittaker functions.}
Before establishing the summation formula we first establish the following estimate
for Whittaker functions, which we will use repeatedly. 
\begin{lemma} \label{lem:whithaker}
Let $h$ be a smooth compactly supported function on $\mathbb R_{+}$. 
Then for $Y > 0$, 
\begin{align}
\label{equ:matthes}
\int_{\mathbb{R}} W_{\pm \tfrac 14, it} (4\pi |n|y )^2 \cdot \frac{h(y Y)}{y^2} dy \ll_{\varepsilon} (1 + |t|)^{\pm \tfrac 12} \cdot e^{-\pi |t|} \cdot Y \cdot 
\begin{cases}
|n / (t Y)|^{1 - \varepsilon} & \text{ if } |n| \leq |t| Y \\
|(t Y) / n|^{100} & \text{ if } |n| > |t| Y
\end{cases}
\end{align}
\end{lemma}
\begin{proof}
The proof of (\ref{equ:matthes}) relies on some computations of Matthes \cite{Matthes, Matthes2}
(further refined in a lemma of Luo-Rudnick-Sarnak \cite{LuoRudnickSarnak}).  
Let
$$
\mathcal{M}_{\kappa, it} (s) := \int_{0}^{\infty} W_{\kappa, it}(y)^2 \cdot y^{s - 2} dy \ , \ \tmop{Re} (s) > 0.
$$
In Lemma 1 of \cite{Matthes2},
Matthes proves that for $s = \sigma + i\tau$ with $\varepsilon < \sigma \leq 1000$, and $-\tfrac12 \le \kappa <\tfrac12$
\begin{equation} \label{eq:matthesbd}
|\mathcal{M}_{\kappa, it}(s)| \leq \mathcal{M}_{\kappa, it}(\sigma) \ll (1 + |t|)^{\sigma-1+2 \kappa} e^{-\pi |t|}. 
\end{equation}
In addition if $H(s) := \int_{0}^{\infty} h(y) y^{s - 1} dy $, then we notice that 
$$
\int_{0}^{\infty} W_{\kappa, it} ( 4\pi |n| y)^2 \cdot \frac{h(y Y)}{y^2} dy = \frac{1}{2\pi i} \int_{(-\varepsilon)}
H(s) Y^{-s} (4\pi |n|)^{s + 1} \mathcal{M}_{\kappa, it}(-s) ds.
$$
For the proof of the bound for $|n| \leq |t| Y$ we apply the triangle inequality and Matthes bound on the line
$\tmop{Re} (s) = -\varepsilon$, and for the proof of the bound for $|n| > |t| Y$ we shift to $\tmop{Re} (s) = -101$ and then 
bound trivially using Matthes's result. 
\end{proof}
\subsection{The summation formula}

In this section we derive the analogue of a convexity bound for the Dirichlet series with coefficients given by $|b_{g,\infty}(n)|^2$.  
This allows
us to estimate the average of $|b_{g,\infty}(n)|^2$ when we sum slightly more than $|t|$ of these coefficients. 
To rephrase this in terms of $L$-functions note that this is analogous to using the approximate functional equation for $L(\tfrac 12, f \otimes \chi_d)$ followed by an application of Poisson summation for $\chi_d$; as a result we would relate the sum of
$L(\tfrac 12, f \otimes \chi_d)$ with $|d| \leq X$ to a
a similar sum of length $1 + |t|^2 / X$, which is shorter when $X$ is slightly larger than $|t|$. From this we see that we expect to win when $X$ exceeds $|t|$ slightly. Since we can only
extend our sum to have length $|t|(\log |t|)^{\eta}$ for some small $\eta>0$, we require a convexity
bound which is stronger than what one obtains using the Phragmen-Lindel\"of principle. For general $L$-functions such convexity bounds have been given by Heath-Brown \cite{HeathBrown}, and a stronger ``weakly subconvex'' bound was obtained by Soundararajan \cite{weaksubconvex}.   
The main tool used to prove the convexity bound is the following functional equation. 
\begin{lemma} \label{lem:functequ}
Consider $$\mathcal{M}_{\kappa,it}(s) = \int_{0}^{\infty} \frac{W_{\kappa, it}(y)^2}{y} \cdot y^{s - 1} dy
\text{ and } 
R_{\pm, g}(s) = \sum_{n > 0} \frac{|b_{g,\infty}(\pm n)|^2}{n^{s - 1}}.
$$
Let
$$
G(s) = \sum_{\pm} R_{\pm, g}(s) \mathcal{M}_{\pm \tfrac 14, it}(s)
$$
Then $G(s)$ is a meromorphic function, whose only singularity in $\tmop{Re} (s) > \tfrac 12$ is a simple pole at $s = 1$ with residue $(2\pi )^{-1}$. In addition $G(s)$ has no singularities in $\tmop{Re}( s) < 0$. 
Finally, the completed Dirichlet series $\widetilde{G}(s) := \pi^{-2s}\Gamma(s) \zeta(2s) G(s)$ satisfies 
the functional equation
$
\widetilde{G}(s) = \widetilde{G}(1 - s)
$
and $s ( 1 - s) \widetilde{G}(s)$ is entire.
\end{lemma}
\begin{proof}
The proof of this is nearly identical to the proof given by Kohnen and Zagier \cite{KohnenZagier} for holomorphic half-integral
weight forms lying in the Kohnen plus space. For completeness we have included a proof in the appendix. 
\end{proof}
This allows us to prove the following ``convexity bound''. 
\begin{lemma} \label{lem:convexity} Assume the Generalized Riemann Hypothesis.
Let $G(s)$ be as in Lemma \ref{lem:functequ} and $s=\sigma+i\tau$. Then for  $  \tfrac12< \sigma <1$
$$
G(\sigma + i \tau) \ll_{\sigma} (\log |t|)^{\varepsilon} \cdot (1 + |\tau|)^{1 - \sigma + \varepsilon}. 
$$
\end{lemma}
\begin{proof}
Let $H(s) = e^{s^2}$, and $\Lambda(s) = s ( 1 - s ) \pi^{-2s} \Gamma(s) \zeta(2s) G(s)$. Then according to Lemma \ref{lem:functequ}
we have $\Lambda(s) = \Lambda(1 - s)$ and $\Lambda(s)$ is an entire function. Therefore, by a standard argument, for $\tmop{Re}( s) > \tfrac 12$, 
\begin{equation} \label{equ:approxfunct}
\Lambda(s) = \frac{1}{2\pi i} \int_{(1)} \Lambda(s + w) \frac{H(w) dw}{w} + \frac{1}{2\pi i} \int_{(1)} 
\Lambda(1 - s + w) \frac{H(w) dw}{w}.  
\end{equation}
Define $c_{\pm}(n)$ by, 
$$
\sum_{n \geq 1} \frac{c_{\pm}(n)}{n^{s}} = \zeta(2s) R_{\pm,g}(s) \ , \ c_{\pm}(n) := \sum_{k^2 \ell = n} |b_{g,\infty}(\pm \ell)|^2 \ell .
$$
Also,  let
\begin{equation} \label{eq:phidef}
\Phi_s(w)=\frac{ (s+w)(1-(s+w)) \pi^{-2w} \Gamma(s+w) }{s(1-s)} \cdot \frac{H(w)}{w}.
\end{equation}
Thus, expanding $\zeta(2s)R_{\pm,g}(s)$ as a Dirichlet series we have for $\tfrac12 < \sigma <1$
\begin{equation} \label{eq:approxfun}
\begin{split}
\zeta(2s)G(s)
=&\sum_{\pm} \sum_{n \ge1} \frac{c(\pm n)}{n^s}
\frac{1}{2 \pi i} \int_{(1)} n^{-w} \frac{\Phi_s(w)}{\Gamma(s)} \mathcal M_{\pm \frac14, it}(s+w) \, dw\\
&+\sum_{\pm} \sum_{n \ge1} \frac{c(\pm n)}{n^{1-s}}
\frac{1}{2 \pi i} \int_{(1)} n^{-w} \frac{\Phi_{1-s}(w)}{\Gamma(s)} \mathcal M_{\pm \frac14, it}(1-s+w) \, dw.
\end{split}
\end{equation}
For $s=\sigma+i\tau$ and $w=u+iv$ 
Stirling's formula gives that uniformly in the range $0< \varepsilon \le \sigma \le A$, $-\sigma+\varepsilon \le u \le A$, and $|s-1|\ge \varepsilon$
\begin{equation} \notag
|\Phi_s(w) | \ll  |\Gamma(s)| \cdot (1+|\tau|)^u e^{-v^2/2}.
\end{equation}
Using this and \eqref{eq:matthesbd} gives 
\begin{equation} \label{eq:weight}
\begin{split}
&\int_{(1)} n^{-w} \frac{\Phi_s(w)}{\Gamma(s)} \mathcal M_{\pm \frac14, it}(s+w) \, dw  \\
&  \ll 
\begin{cases} \displaystyle
\left| \frac{t}{n} \right|^{c} (1+|t|)^{\sigma-1\pm \frac12} e^{-\pi |t|} (1+|\tau|)^{c}  & \text{ if } |n| > |t| \text{ for any } 0<c<100, \\
 (1+|t|)^{\sigma-1\pm \frac12} e^{-\pi |t|}\left(1+\left|\frac{t}{n} \right|^{c} \right) & \text{ if } |n| \leq |t| \text{ for any }  -\sigma<c<0.
\end{cases}
 \end{split}
\end{equation}
Also, as a consequence of Lemma \ref{lem:secondmoment} we have under 
GRH that
\begin{equation} \label{eq:sumsofcoeff}
\sum_{X \leq n \leq 2X} c_{\pm}(n) \ll X \cdot (1 + |t|)^{\mp \tfrac 12} e^{\pi |t|} \cdot \begin{cases}
(\log X)^{\varepsilon} & \text{ if }   X \ge |t|^{\varepsilon},  \\
|t|^{\varepsilon} & \text{ if } X \le |t|^{\varepsilon}. 
\end{cases}
\end{equation}
Hence, when $|n| > |t|$ we pick $c = 1 - \sigma + \varepsilon$ in (\ref{eq:weight}) and using (\ref{eq:sumsofcoeff})
we get 
\begin{equation} \label{eq:weighted}
\begin{split}
& \sum_{\pm} \sum_{n \ge1} \frac{c(\pm n)}{n^s}
\frac{1}{2 \pi i} \int_{(1)} n^{-w} \frac{\Phi_s(w)}{\Gamma(s)} \mathcal M_{\pm \frac14, it}(s+w) \, dw \\
& \ll  (1+|\tau|)^{1-\sigma+\varepsilon} |t|^{\sigma-1} e^{-\pi |t|} \left( \sum_{\pm} |t|^ {\pm \frac12} \left(\sum_{ n \le |t|} \frac{c(\pm n)}{n^{\sigma}}+ \sum_{|t| \le n} \frac{c(\pm n)}{n^{\sigma}} \left|\frac{t}{n} \right|^{1-\sigma+\varepsilon}\right) \right).
\end{split} 
\end{equation} 
Let $\mathcal J= 1/(\varepsilon \log 2) $ and apply \eqref{eq:sumsofcoeff} to see that
\begin{equation} \label{eq:dyadic1}
\begin{split}
\sum_{ n \le |t|} \frac{c(\pm n)}{n^{\sigma}} \ll & 
\sum_{n \le |t|^{\varepsilon} } \frac{c(\pm n)}{n^{\sigma}} +\frac{1}{|t|^{\sigma}} \sum_{j=1}^{\mathcal J} 2^{j\sigma} \sum_{ \frac{|t|}{2^{j+1}} \le n \le \frac{|t|}{2^j}} c(\pm n) \\
\ll& e^{\pi |t|}|t|^{\mp \frac12}( |t|^{\varepsilon}+|t|^{1-\sigma} (\log |t|)^{\varepsilon}).
\end{split}
\end{equation} 
Similarly,
\begin{equation} \label{eq:dyadic2}
\begin{split}
\sum_{|t| \le n} \frac{c(\pm n)}{n^{\sigma}} \left|\frac{t}{n} \right|^{1-\sigma+\varepsilon}
\ll &  |t|^{-\sigma} \sum_{j = 0}^{\infty} \left( 2^j\right)^{-(1+\varepsilon)} \sum_{ 2^j |t| \le n \le 2^{j+1} |t|} c(\pm n) \\ 
\ll & e^{\pi |t|} |t|^{1-\sigma\mp \frac12}(\log |t|)^{\varepsilon}.
\end{split}
\end{equation} 
By \eqref{eq:dyadic1} and \eqref{eq:dyadic2} we see that \eqref{eq:weighted} is $\ll (1+|\tau|)^{1 - \sigma+\varepsilon} (\log |t|)^{\varepsilon}$.
By a similar argument we have
\[
\sum_{\pm} \sum_{n \ge1} \frac{c(\pm n)}{n^{1-s}}
\frac{1}{2 \pi i} \int_{(1)} n^{-w} \frac{\Phi_{1-s}(w)}{\Gamma(s)} \mathcal M_{\pm \frac14, it}(1-s+w) \, dw \ll (1+|\tau|)^{1 - \sigma+\varepsilon} (\log |t|)^{\varepsilon}.
\]
Using these estimates in \eqref{eq:approxfun} completes the proof.

\end{proof}
We are now ready to prove our summation formula.

\begin{lemma} \label{lem:summation}
Let $k$ be a smooth compactly supported function on $\mathbb R_{+}$. Then, for $Y \ge 1$, 
\begin{align*}
\sum_{n \neq 0} & |b_{g,\infty}(n)|^2 \int_{0}^{\infty} W_{\tfrac 14 \tmop{sgn}(n), it}(4\pi |n| y)^2 \cdot \frac{k(y Y) dy}{y^2}  = \frac{Y K(-1)}{2\pi} + O(Y^{1/2 + \varepsilon} (\log t)^{\varepsilon}). 
\end{align*}
\end{lemma}
\begin{proof}
Let us keep the notation of Lemma \ref{lem:functequ}. 
Applying Mellin inversion 
we find that
\begin{align*}
\sum_{n \neq 0} & |b_{g,\infty}(n)|^2 \int_{0}^{\infty} W_{\tfrac 14 \tmop{sgn}(n), it}(4\pi |n| y)^2 \cdot \frac{k(y Y) dy}{y^2}  \\
& = \frac{1}{2\pi i} \int_{(\sigma)} K(-s) Y^{s} (4\pi)^{1-s } G(s) ds \ , \ \sigma >  1.  
\end{align*}
Shifting the contour to $\sigma = \frac12+\varepsilon$ we collect a simple pole at $s = 1$ with residue $(2\pi)^{-1}$, and we get that the above integral equals
$$
\frac{1}{2\pi} \cdot K(-1) Y + \frac{1}{2\pi i} \int_{(\tfrac12+\varepsilon)} K(-s) Y^{s} (4\pi)^{1-s } G(s) ds .
$$
Applying Lemma \ref{lem:convexity}, the above integral is
\[
O\left( Y^{\frac12+\varepsilon} (\log |t|)^{\varepsilon}\right),
\] 
which gives the claim.
\end{proof}

\subsection{Proof of Quantum Unique Ergodicity}
Recall the discussion in Section
\ref{sec:weyl} where we saw that Theorem \ref{thm:maass}
follows from the estimates
\begin{equation} \label{equ:mainterm2}
\sum_{n \neq 0}
|b_{g, \infty}(n)|^2 \int_0^{\infty}W_{\frac14 \tmop{sgn}(n), it}(4\pi y |n|)^2 h(y) \frac{dy}{y^2}= \frac{1}{\tmop{vol}(\Gamma_0(4)\backslash \mathbb H)}
 \int_0^{\infty} h(y) \frac{dy}{y^2}+o(1)
\end{equation}
as $|t| \rightarrow \infty$
and for $\ell \neq 0$
\begin{equation} \notag
\begin{split}
\sum_{n \neq 0} b_{g, \infty}(n) b_{g, \infty}(n +\ell)\int_0^{\infty} W_{\frac14 \tmop{sgn}(n),it} (4\pi |n|y)
W_{\frac14 \tmop{sgn}(n+\ell),it} (4\pi |n+\ell|y) h(y) \frac{dy}{y^2}=o(1),
\end{split} 
\end{equation}
as $|t |\rightarrow \infty$.
To bound the off-diagonal contribution we note that $W_{\kappa, iu}(y) = W_{\kappa, - iu}(y)$ and that
$W_{\kappa, i u}(y) \in \mathbb{R}$. Using the bound $|ab| \leq a^2 + b^2$, we see that to control the off-diagonal
it will be enough to show that
\begin{equation} \label{equ:offdiagonal2}
\sum_{n \neq 0} |b_{g,\infty}(n) b_{g,\infty}(n + \ell)| \int_{0}^{\infty} W_{\tfrac 14 \tmop{sgn}(n), it}(4\pi |n| y)^2
\cdot \frac{h(y) dy}{y^2} = o(1)
\end{equation}
as $|t| \rightarrow \infty$ for $\ell \neq 0$. 
Thus the proof of the theorem reduces to establishing (\ref{equ:offdiagonal2}) and (\ref{equ:mainterm2}).

First let us consider (\ref{equ:offdiagonal2}). Combining Lemma \ref{lem:offdiagonal} and Lemma \ref{lem:whithaker} we have
\begin{equation} \label{eq:offdiagonalestimation}
\begin{split}
& \sum_{n \neq 0} |b_{g,\infty}(n) b_{g,\infty}(n + \ell)| \int_{0}^{\infty} W_{\tfrac 14 \tmop{sgn}(n), it}(4\pi |n| y)^2
\cdot \frac{h(y) dy}{y^2} \\
& \ll  |t|^{\pm \frac12} e^{-\pi |t|}
\left( \sum_{ |n| \le |t|} |b_{g,\infty}(n) b_{g,\infty}(n + \ell)| \left| \frac{n}{t}\right|^{1-\varepsilon} 
+\sum_{ |n| > |t|} |b_{g,\infty}(n) b_{g,\infty}(n + \ell)| \left| \frac{t}{n}\right|^{100} \right) \\
& \ll 
(\log |t|)^{-25/156+\varepsilon},
\end{split}
\end{equation}
for each fixed $\ell \neq 0$ and in the last step the terms with $|n| \le |t|^{\varepsilon}$ are handled trivially while in the other ranges we split up the sums into intervals of the form $\mathcal I_j=(|t|/2^{j+1},|t|/2^j]$, $j \in \mathbb{Z}$, apply Lemma \ref{lem:offdiagonal} to bound the sums over the intervals $\mathcal I_j$, then sum over $j$ (see \eqref{eq:dyadic1} and \eqref{eq:dyadic2} for a similar argument). 

The estimation of the LHS of \eqref{equ:mainterm2} is more intricate.
Given smooth functions $h,k$ compactly supported on $\mathbb R_{+}$ denote by $H(s), K(s)$ their respective Mellin transforms. 
Then, by Lemma \ref{lem:extend} the LHS of (\ref{equ:mainterm2}) is equal to 
\begin{equation} \label{equ:extend2a}
\begin{split}
& \frac{H(-1)}{K(-1)Y} \cdot (1 + O(Y^{-1/2})) \mathcal{S}_Y(0, t; k) + O(Y^{-1/2})   \\
& \qquad \qquad \qquad  +O\Big ( \frac{1}{Y^{3/2}} \sum_{\ell \neq 0} \frac{(\nu(\ell)+1) 2^{\nu(\ell)} d(|\ell|)}{1 + |\ell / Y|^A} \cdot
\mathcal{S}_Y(\ell, t; k) \Big ) 
\end{split}
\end{equation}
where $\nu(\ell)$ is the exponent of the largest power of $2$ dividing $\ell$ and $\mathcal{S}_Y(\ell, t; k)$ is defined as,
\begin{align*}                                                                                                                        
\int_{0}^{\infty} k(y Y) \Big | \sum_{n \neq 0, -\ell} b_{g,\infty}(n) b_{g,\infty}(n + \ell) W_{\tfrac 14 \tmop{sgn}(n), it}  
(4\pi |n| y) W_{\tfrac 14 \tmop{sgn}, \text{it}} (4\pi |n + \ell| y) \Big | \cdot \frac{dy}{y^2}  .
\end{align*}
Moving the absolute value inside and using Lemma \ref{lem:offdiagonal} and Lemma \ref{lem:whithaker} we have by proceeding as in \eqref{eq:offdiagonalestimation} that
$$
|\mathcal{S}_Y(\ell, t; k) |\ll_{\varepsilon} Y (\log Y|t|)^{-25/156 + \varepsilon}
$$
provided that $1 \leq Y \leq |t|$
 and $\ell \leq \log |t|$. In addition applying Cauchy-Schwarz and using Lemma \ref{lem:secondmoment}
and Lemma \ref{lem:whithaker} we get 
$$
|\mathcal{S}_Y(\ell, t; k)| \ll_{\varepsilon} Y (\log Y|t|)^{\varepsilon}
$$
for all $\ell$.
Choosing $Y = (\log |t|)^{\eta}$ for some exponent $\eta<1$ to be specified later
it follows that the error term in (\ref{equ:extend2a}) is 
\begin{equation} \label{eq:offdiagonalbdagain}
\ll Y^{1/2} \cdot (\log |t|)^{-25/156 + \varepsilon}.  
\end{equation}

To estimate the main term in \eqref{equ:extend2a} apply the summation formula in Lemma \ref{lem:summation} to see that 
\begin{equation} \label{eq:maintermagain}
\frac{H(-1)}{K(-1)Y} 
 \mathcal{S}_Y(0, t; k)=
\frac{H(-1)}{2\pi} + O(Y^{-1/2 + \varepsilon} (\log t)^{\varepsilon}).
\end{equation}
Using \eqref{eq:offdiagonalbdagain} and \eqref{eq:maintermagain} in \eqref{equ:extend2a}
it follows that
\begin{equation} \notag
\begin{split}
&\sum_{n \neq 0}
|b_{g, \infty}(n)|^2 \int_0^{\infty}W_{\frac14 \tmop{sgn}(n), it}(4\pi y |n|)^2 h(y) \frac{dy}{y^2}\\
& \qquad \qquad \qquad =
\frac{H(-1)}{2\pi} +O(Y^{-1/2+\varepsilon} (\log |t|)^{\varepsilon})+O(\sqrt{Y} (\log |t|)^{-25/156+\varepsilon}).
\end{split}
\end{equation}
Choosing $Y = (\log |t|)^{25/156}$ the error term is $O((\log |t|)^{-25/312+\varepsilon})=o(1)$, thereby establishing \eqref{equ:mainterm2}.
Additionally, we have already shown that the LHS of (\ref{equ:offdiagonal2}) is
$O( (\log  |t|)^{-25/156 + \varepsilon})$. 
Therefore,  we have proved Theorem \ref{thm:maass}.  

\section{Equidistribution of mass and zeros of half-integral weight holomorphic forms} \label{sec:holomorphic}
\subsection{Preliminaries}
Write $S_{k+1/2}(\Gamma_0(4))$ for the space of weight $k+\tfrac12$ holomorphic cusp forms for $\Gamma_0(4)$. 
Every $g \in S_{k+1/2}(\Gamma_0(4))$ has a Fourier expansion of the form
\[
g(z)=\sum_{n \ge 1} c(n)e(nz)
\]
and
for odd $p$ the Hecke operator $T_{p^2}$ is defined 
on $S_{k+1/2}(\Gamma_0(4))$ as
\[
T_{p^2} g(z)=\sum_{n \ge 1} \left( c(p^2n)+\left( \frac{(-1)^kn}{p}\right)p^{k-1}c(n)+p^{2k-1}c\left( \frac{n}{p^2}\right)\right)e(nz).
\]

Let us now recall some results of Kohnen \cite{Kohnen1} and Kohnen-Zagier \cite{KohnenZagier}.
The Kohnen plus space $S_{k+1/2}^+(\Gamma_0(4))$ denotes the subspace of $S_{k+1/2}(\Gamma_0(4))$ of cusp forms whose Fourier coefficients satisfy $c(n)=0$ unless $(-1)^k n \equiv 0,1 \pmod 4$
and has a basis 
consisting of simultaneous eigenfunctions of the operators $T_{p^2}$ for all odd $p$. For
such a $g \in S_{k+1/2}^+(\Gamma_0(4))$ with $T_{p^2} g=\lambda_g(p) g$
there exists a 
  Hecke cusp form $f \in S_{2k}(\text{SL}_2(\mathbb{Z}))$, such that $H_p f=\lambda_g(p) f$ where $H_p$ is the usual Hecke operator for $S_{2k}(\text{SL}_2(\mathbb{Z}))$. By the strong multiplicity one theorem this determines $f$ (up to scalar multiplication). In addition,
writing the Fourier expansion of $f$ as
 \[
 f(z)=\sum_{n \ge 1} a(n) e(nz)
 \]
 and normalizing with $a(1)=1$
 the Fourier coefficients of $f$ and $g$ are related by the formula
 \[
 c(|d|\delta^2)=c(|d|)\sum_{e|\delta} \mu(e) e^{k-1} \chi_d(e) a\left( \frac{\delta}{e}\right)
 \]
 where $d$ is a fundamental discriminant with $(-1)^k d>0$.
 With this correspondence, 
one has for $n=\delta^2 |d|$ with $d$ a fundamental discriminant such that $(-1)^k d>0$
\begin{equation} \label{eq:KZbd}
|c(n)| \ll  n^{\frac{k}{2}-\frac14} \delta^{\varepsilon} (4\pi)^{k/2} \left( \frac{ 1}{\Gamma(k+\tfrac12)} \right)^{1/2} \left(\frac{L(\tfrac12, f\otimes \chi_d)}{ L(1,\tmop{Sym}^2 f)} \right)^{1/2},
\end{equation}
where $g$ has been $L^2$-normalized, that is,
\[
\iint_{\Gamma_0(4) \backslash \mathbb H} y^{k+1/2} |g(z)|^2 \tmop{dvol}(z)=1,
\]
and recall $f$ is arithmetically normalized with $a(1)=1$ so that
\[
\iint_{\text{SL}_2(\mathbb{Z}) \backslash \mathbb H} y^{2k} |f(z)|^2 \tmop{dvol}(z)=\frac{(2k-1)! L(1, \tmop{Sym}^2 f)}{(4\pi)^{2k-1} 2\pi^2}.
\]
%

Kohnen-Zagier also showed that the function
\begin{equation} \notag
D(s)= (4 \pi)^{-k-\frac12}\Gamma(s+k-\tfrac12)
\sum_{ n \ge 1} \frac{|c(n)|^2}{n^{s+k-\frac12}} 
\end{equation}
is absolutely convergent for $\tmop{Re}(s)>1$,
admits a meromorphic continuation to the complex
plane with the only singularity in $\tmop{Re}(s) \ge \tfrac12$ being a simple pole at $s=1$ with residue $1/(2\pi)$ (note that $D(s)$ may have poles at $s=\rho/2$ for zeros $\rho$ of $\zeta(s)$). Moreover,
the completed function $\widetilde D(s)=\pi^{-2s} \Gamma(s) \zeta(2s)D(s)$ satisfies the functional equation $\widetilde D(s)=\widetilde D(1-s)$ and $s(1-s)\widetilde D(s)$ is an entire function.

Finally we record the simple consequence of Stirling's formula which will be used repeatedly
\begin{equation} \label{eq:stirling}
\frac{\Gamma(s+k)}{\Gamma(k)}=k^s(1+O((|s|+1)^2k^{-1}))
\end{equation}
for $|s|=o(\sqrt{k})$.

\subsection{Proof of Theorem \ref{thm:holomorphic}}
It suffices to show that
\begin{equation} \label{equ:holomorphicmain}
\sum_{n \geq 1}
|c(n)|^2 \int_0^{\infty} e^{-4\pi n y} y^{k+\frac12} h(y) \frac{dy}{y^2}=
 \frac{1}{\tmop{vol}(\Gamma_0(4)\backslash \mathbb H)}
 \int_0^{\infty} h(y)  \frac{dy}{y^2} +o(1)
\end{equation}
as $k \rightarrow \infty$
and for $\ell \neq 0$
\begin{equation} \label{equ:holomorphicoff}
\begin{split}
\sum_{n \ge 1 } c(n) c(n +\ell)\int_0^{\infty} e^{-2\pi(2n+\ell)y} y^{k+\frac12} h(y) \frac{dy}{y^2}=o(1),
\end{split} 
\end{equation}
as $k \rightarrow \infty$ (see Section \ref{sec:weyl}).

To estimate the sums of Fourier coefficients
we will proceed in the same way as for the Maa{\ss} forms. First,
we will use estimates for moments of $L(\tfrac12, f \otimes \chi_d)$
and \eqref{eq:KZbd}
to obtain \eqref{equ:holomorphicoff}.
Next, to estimate the main term \eqref{equ:holomorphicmain}
we will extend the length of summation then apply a summation formula, which is obtained through a convexity bound for $D(s)$. This will give \eqref{equ:holomorphicmain} thereby proving Theorem \ref{thm:holomorphic}. 

\subsubsection{ Estimates for moments}
Let $f$ be a weight $2k$ level 1 holomorphic Hecke cusp form.
We require analogues of Lemmas \ref{lem:shiftedmoment} and \ref{lem:firstmoment}. Here the weight plays the role of the spectral parameter $|t|$. Additionally, since 
we have Deligne's bound for the Hecke eigenvalues of $f$ it follows from GRH that
\begin{equation}\label{eq:deligne}
\frac{1}{\log \log k} \ll L(1, \tmop{Sym}^2 f) \ll \log \log k.
\end{equation}
Repeating the same argument
used in the proof of Lemma \ref{lem:shiftedmoment} 
  we have uniformly for $ X > k^{\varepsilon}$ and $1 \le a,b, |\ell| \le (\log X)^{100}$ that
$$
\sum_{\substack{d_1,d_2 \\ a|d_1| \le X  \\ ad_1=bd_2+\ell}}
L(\tfrac 12, f \otimes \chi_{d_1})^{1/2} L(\tfrac 12, f \otimes \chi_{d_2})^{1/2} \ll \frac{X}{[a,b]} (\log X)^{-1/4+\varepsilon}
$$ 
and
$$
\sum_{|d| \le X} L(\tfrac 12, f \otimes \chi_d) \ll  X (\log X)^{\varepsilon},
$$
where the summation in both sums is over fundamental discriminants.
In fact, in the holomorphic case the estimate \eqref{eq:deligne} allows us to omit entirely
the argument used in the proof of Lemma \ref{lem:firstmoment}, where we treated the small primes
separately.

From the moments estimates
we obtain
a bound on sums of Fourier
coefficients. Using \eqref{eq:KZbd}
in place of Corollary \ref{cor:bound} we proceed as in Section \ref{sec:smooth} 
to see that for $X \ge k^{\varepsilon} $ 
that
\begin{equation} \label{eq:momentshol}
\sum_{ X \le n \le 2X} \frac{|c(n)c(n+\ell)|}{(2 \pi (2n+\ell))^{k-\frac12}}
 \ll  \frac{X}{\Gamma(k+\tfrac12)} \cdot \begin{cases}
(\log X)^{-1/4+\varepsilon} & \text{ if } 0 \neq |\ell| \le \log X \\
(\log X)^{\varepsilon} & \text{ if } \ell=0.
\end{cases}
\end{equation}
Here one uses the elementary bound
\[
\left( \frac{\sqrt{n(n+\ell)}}{n+\ell/2}\right)^{k-\tfrac12} \le 1,
\]
which holds for $|\ell|<n/2$. Also, it follows from the Lindel\"of bound that the LHS of \eqref{eq:momentshol} is $\ll (X k (|\ell|+1))^{\varepsilon} X/\Gamma(k+\tfrac12)$ for all $X \ge 1$ and $\ell$.

\subsubsection{ Extending the length of summation}

Next, we need an analogue of Lemma \ref{lem:extend}.
Let $h,j \in C_c^{\infty}(\mathbb R_{ +})$, $j(x) \ge 0$,
with Mellin transforms $H(s), J(s)$ (resp.).
As before,
 we evaluate
\begin{equation} \label{eq:holomorphicsum}
\iint_{\Gamma_0(4) \backslash \mathbb H} E(z|h)  E^Y(z|j) y^{k+\frac12} |g(z)|^2 \, \tmop{dvol}(z)
\end{equation}
in two different ways.
Here $E(z|h)$ and $ E^Y(z|j)$
are as defined in \eqref{eq:incompleteEisenstein1} and \eqref{eq:incompleteEisenstein2} (resp.). First, we proceed as in Section \ref{sec:firsteval} rewriting $E^Y(z|j)$ in terms of $E(z,s)$, using Mellin inversion, then shifting contours of integration
 (taking advantage of the analytic continuation of
the Eisenstein series) to see that \eqref{eq:holomorphicsum}
equals
\[
J(-1) \frac{Y}{2\pi}  
\sum_{n \ge 1} |c(n)|^2 \int_{0}^{\infty} e^{-4\pi n y} \cdot y^{k-\frac32} h(y) \, dy+
 O(\sqrt{Y}).
\]
Next, as in Section \ref{sec:secondeval} we use the unfolding technique with $E^Y(z|j)$ and then expand $E(z|h)$ into a Fourier
series
as given in Lemma \ref{lem:fourier}
to see that \eqref{eq:holomorphicsum}
equals
\[
\sum_{n, \ell} c(n)c(n+\ell)
\int_0^{\infty} a_{\ell, h}(y)  e^{-2\pi(2n+\ell)y} y^{k-\frac32} j(yY)  dy.
\]
Applying the estimates for $a_{\ell,h}(y)$ given in Lemma \ref{lem:fourier} we get that the above equals 
\[
\frac{H(-1)}{2\pi} \sum_{n \ge 1} |c(n)|^2 \int_{0}^{\infty} e^{-4\pi n y} \cdot y^{k-\frac32} j(yY) \, dy \cdot 
\big( 1+\big( Y^{-1/2}\big)\big)+O(\mathcal E_Y)
\]
where
\begin{equation} \notag
\mathcal E_Y = \frac{1}{\sqrt{Y}} \sum_{\ell \neq 0 } \frac{(\nu(\ell)+\ell) 2^{\nu(\ell)}d(|\ell|)}{1+|\ell /Y|^A}
\sum_{n \ge 1} |c(n) c(n + \ell)|\int_{0}^{\infty}  e^{-2\pi(2n+\ell)y} \cdot y^{k-\frac32} j(yY)  \, dy 
\end{equation}
and
$\nu(\ell)$ is the larges power of $2$ dividing $\ell$. This completes our second evaluation of \eqref{eq:holomorphicsum}.
Combining the above formulas gives us the extended sum
\begin{align} \label{equ:extendholomorphic}
&H(-1)
\sum_{n \ge 1} |c(n)|^2 \int_{0}^{\infty} e^{-4\pi n y} \cdot y^{k-\frac32} j(Yy) \, dy  \cdot 
\big( 1+\big( Y^{-1/2}\big)\big) \\
& \quad  = J(-1) Y \sum_{n \ge 1} |c(n)|^2 \int_{0}^{\infty} e^{-4\pi n y} \cdot y^{k-\frac32} h(y) \, dy \nonumber +O(\mathcal E_Y)+O\Big( \sqrt{Y}\Big).
\end{align}
\subsubsection{ A summation formula}
We will show assuming GRH that for $\tfrac12< \sigma<1$
\begin{equation} \label{eq:convexityhol}
|D(\sigma+i\tau)| \ll (1+|\tau|)^{1-\sigma+\varepsilon} (\log k)^{\varepsilon}.
\end{equation}
By a contour integration argument
this gives the following summation formula
\begin{align} \label{equ:summationholomorphic}
\sum_{n \ge 1} |c(n)|^2 \int_{0}^{\infty} e^{-4\pi n y} \cdot y^{k-\frac32} j(Yy) \, dy  =& \frac{1}{2\pi i} \int_{(2)} J(-s) Y^{s} (4\pi)^{1-s } D(s) \, ds \\
=& \frac{Y J(-1)}{2\pi}+O\left(Y^{1/2+\varepsilon} (\log k)^{\varepsilon}\right) \notag.
\end{align}
To show (\ref{eq:convexityhol})
let
\[
\sum_{n \ge 1} \frac{\widetilde c(n)}{n^s}=
\zeta(2s) \sum_{ n \ge 1} \frac{|c(n)|^2}{n^{s+k-\frac12}} , \quad
\widetilde c(n)=\sum_{m^2 h=n} |c(h)|^2 h^{\frac12-k}
\]
and note that by \eqref{eq:momentshol} we have
\begin{equation} \label{eq:coefbdagain}
\sum_{X \le n \le 2X}
\widetilde c(n) \ll \frac{(4\pi)^k X}{\Gamma(k+\tfrac12)}
\cdot
\begin{cases}
(\log k)^{\varepsilon} & \text{ if } X > k^{\varepsilon}, \\
k^{\varepsilon} & \text{ if } X \le k^{\varepsilon}.
\end{cases}
\end{equation}
%
%
%
Also, let $\Phi_s(w)$ be defined as in \eqref{eq:phidef}. 
Since $s(1-s)\widetilde D(s)$ is invariant under $s \rightarrow 1-s$
a standard contour integration argument gives
\begin{equation} \label{eq:approxhol}
\begin{split}
(4\pi)^{k+1/2}\zeta(2s)D(s)
=& \sum_{n \ge 1}
 \frac{\widetilde c(n)}{n^s} \frac{1}{2\pi i}
\int_{(1)} n^{-w} \Phi_s(w) \frac{\Gamma(k+w+s-\tfrac12)}{\Gamma(s)} \, dw\\
&  + \sum_{n \ge 1}\frac{\widetilde c(n)}{n^{1-s}} \frac{1}{2\pi i}
\int_{(1)} n^{-w} \Phi_{1-s}(w) \frac{\Gamma(k+w-s+\tfrac12)}{\Gamma(s)} \, dw.
\end{split}
\end{equation}
Using an analogue of \eqref{eq:weight} (which is obtained by
applying \eqref{eq:stirling}) along with \eqref{eq:coefbdagain}
we see that
\begin{equation} \notag
\begin{split}
& \frac{1}{(4\pi)^{k+1/2}}
\sum_{n \ge 1} \frac{\widetilde{ c}(n)}{n^{s}} \int_{(1)} n^{-w} \Phi_s(w) \frac{\Gamma(k+s+w-\tfrac12)}{\Gamma(s)} \, dw \\
&  \qquad \ll  (1+|\tau|)^{1-\sigma+\varepsilon}(4\pi)^{-k} k^{\sigma-1} \Gamma(k+\tfrac12) \left(\sum_{ n \le k} \frac{\widetilde{ c}(n)}{ n^{\sigma}} +\sum_{ n \ge k} \frac{\widetilde{c}(n)}{ n^{\sigma}} \left|\frac{k}{n}\right|^{1-\sigma+\varepsilon}\right) \\
& \qquad
\ll  (\log k)^{\varepsilon}(1+|\tau|)^{1-\sigma+\varepsilon}.
\end{split}
\end{equation}
Similarly, the second term in \eqref{eq:approxhol} is $O((1+|\tau|)^{1-\sigma+\varepsilon} (\log k)^{\varepsilon})$. This proves \eqref{eq:convexityhol}.

\subsubsection{Completion of the proof}

First we will establish \eqref{equ:holomorphicoff}.  Note that for $h \in C_c^{\infty}(\mathbb R_+)$ 
using Mellin inversion together with \eqref{eq:stirling} then applying \eqref{eq:momentshol}
we have that
\begin{equation} \notag
\begin{split}
& \sum_{n \ge 1 } c(n) c(n +\ell)\int_0^{\infty} e^{-2\pi(2n+\ell)y} y^{k+\frac12} h(y) \frac{dy}{y^2} \\
& \qquad \ll  \Gamma(k-\tfrac12) \sum_{n \ge 1} \frac{|c(n)c(n+\ell)|}{(2\pi(2n+\ell))^{k-\frac12}} \left(\left|h\left( \frac{k-\frac12}{(2\pi(2n+\ell))}
\right) \right|+\frac{1}{k^{1-\varepsilon}} \cdot \frac{1}{1+(n/k)^A} \right)\\
& \qquad \ll  \frac{\Gamma(k-\tfrac12)}{\Gamma(k+\tfrac12)} \cdot k (\log k)^{-1/4+\varepsilon} \ll (\log k)^{-1/4+\varepsilon}.
\end{split}
\end{equation}
It remains to estimate the main term \eqref{equ:holomorphicmain} and to accomplish
this we will use \eqref{equ:extendholomorphic}. We first estimate $\mathcal E_Y$ by 
choosing $Y=(\log k)^{\eta}$ for some $0<\eta<1$ to be chosen later and
proceeding as above we get that
\begin{equation} \notag
\begin{split}
\mathcal E_Y=&\frac{1}{\sqrt{Y}} \sum_{\ell \neq 0 } \sum_{\ell \neq 0 } \frac{(\nu(\ell)+\ell) 2^{\nu(\ell)}d(|\ell|)}{1+|\ell /Y|^A}
\sum_{n \ge 1} |c(n) c(n + \ell)|\int_{0}^{\infty} j(yY) e^{-2\pi(2n+\ell)y} y^{k-3/2} \, dy\\
\ll& \frac{\Gamma(k-\tfrac12)}{\sqrt{Y}} \sum_{\ell \neq 0} \frac{(\nu(\ell)+1) 2^{\nu(\ell)}d(|\ell|)}{1+|\ell/ Y|^A} 
\sum_{n \ge 1} \frac{|c(n)c(n+\ell)|}{(2\pi(2n+\ell))^{k-\frac12}} \\
& \qquad \qquad \qquad \times \left(j\left( \frac{Y(k-\frac12)}{(2\pi(2n+\ell))}
\right) +\frac{1}{k^{1-\varepsilon}} \cdot \frac{1}{1+(n/(kY))^A} \right) \\
\ll& Y^{3/2} (\log k)^{-1/4+\varepsilon},
\end{split}
\end{equation}
where to handle the contribution of the terms with $\log k< \ell < k$ we used Cauchy-Schwarz and the second bound in \eqref{eq:momentshol} and for $\ell >k$ we use instead the trivial bound. Combining this along with \eqref{equ:extendholomorphic}
and
\eqref{equ:summationholomorphic}
gives
\[
\sum_{n \ge 1} |c(n)|^2 \int_{0}^{\infty} e^{-4\pi n y} \cdot y^{k-\frac32} h(y) \, dy =\frac{H(-1)}{2\pi}
+O(Y^{-1/2+\varepsilon} (\log k)^{\varepsilon})+O(\sqrt{Y} (\log k)^{-1/4+\varepsilon}).
\]
Taking $Y=(\log k)^{1/4}$ we get
\[
\sum_{n \ge 1} |c(n)|^2 \int_{0}^{\infty} e^{-4\pi n y} \cdot y^{k-\frac32} h(y) \, dy
=\frac{1}{\tmop{vol}(\Gamma_0(4) \backslash \mathbb H)}
\int_0^{\infty} h(y) \frac{dy}{y^2}+O((\log k)^{-1/8+\varepsilon}),
\]
thereby establishing \eqref{equ:holomorphicmain} and completing the proof of Theorem \ref{thm:holomorphic}.

\subsection{Equidistribution of zeros: Proof of Corollary \ref{cor:zeros} } \label{sec:zeros}

For a holomorphic function $f$ on $\mathbb H$
let $\tmop{ord}_{\varrho} (f)$
denote the order of vanishing of $f$ at $\varrho$.  Also let
\[
\Gamma_z=\{ \gamma \in \Gamma_0(4) / \{\pm 1\} : \gamma z=z \}
\]
be the stabilizer group of $z$ and
write $M_{k+1/2}(\Gamma_0(4))$ for the
space of weight $k+1/2$ holomorphic modular forms for $\Gamma_0(4)$.
For $g \in M_{k+1/2}(\Gamma_0(4))$
one trivially has $f(z):=g(z)^4 \in M_{4k+2}(\Gamma_0(4))$
(notice that $\Big(\Big( \frac{-1}{d}\Big)^k \overline{\varepsilon_d} \Big( \frac{c}{d}\Big)^{2k+1}\Big)^4=1$).
The standard valence formula (see \cite{Serre, TopicsIwaniec}) applied to $f$ gives
\[
N_f:= \sum_{\mathfrak a} \tmop{ord}_{\mathfrak a} (f)+
\sum_{ \varrho \in \Gamma_0(4)\backslash \mathbb H} \frac{1}{\# \Gamma_{\varrho}} 
\cdot  \tmop{ord}_{\varrho} (f)=[\tmop{SL}_2(\mathbb Z): \Gamma_0(4)] \cdot \frac{ \left(4k+2\right)}{12}
\]
where the first sum is over the cusps $\mathfrak a$ of $\Gamma_0(4)\backslash \mathbb H $ and $\tmop{ord}_{\mathfrak{a}} (f)$ is the order of vanishing of $f$ at  $\mathfrak{a}$.
Since $\tmop{ord}_{\varrho}(f)=4 \tmop{ord}_{\varrho}(g)$ we conclude that
the total number of weighted $\Gamma_0(4)$ nonequivalent zeros of $g$, which we denote by $N_g$, equals
\[
N_g=\frac{1}{2} \cdot \left(k+\frac{1}{2}\right)
\]
for any modular form $g \in M_{k+1/2}(\Gamma_0(4))$.

Given $\psi\in C_c^{\infty}( \mathbb H)$
define
\[
\Psi(z)=\sum_{\gamma \in \Gamma_0(4)} \psi(\gamma z),
\]
so that $\Psi(z) \in C_c^{\infty}(\Gamma_0(4)\backslash \mathbb H)$.
Again, for $g \in M_{k+1/2}(\Gamma_0(4))$ we note that $f(z)=g(z)^4 \in M_{4k+2}(\Gamma_0(4))$ so
Lemma 2.1 of Rudnick \cite{Rudnick} applies to $f$, from which we deduce that
for any $g \in M_{k+1/2}(\Gamma_0(4))$ 
\begin{equation} \label{eq:mainzeros}
\begin{split}
\frac{1}{N_g}
\sum_{ \varrho \in \Gamma_0(4)\backslash \mathbb H} \frac{1}{\#\Gamma_{\varrho}} \cdot \tmop{ord}_{\varrho} g(\varrho) \Psi(\varrho)
=&\frac{1}{\tmop{vol}(\Gamma_0(4)\backslash \mathbb H)} \iint_{\Gamma_0(4) \backslash \mathbb H}
\Psi(z) \, \frac{dx \, dy}{y^2} \\
&+\frac{2}{ (k+\frac12) \pi} \iint_{\Gamma_0(4) \backslash \mathbb H} \log\left( y^{k+1/2} |g(z)|^2\right) \Delta_{\mathbb H} 
\Psi(z) \, \frac{dx \, dy}{y^2}.
\end{split}
\end{equation}
The first integral above is our main term and it remains to estimate the second integral (here we will assume additional properties about $g$). Let
 $\{ g_k\}$ be a sequence of weight $k+1/2$ cusp forms lying in the Kohnen plus subspace that are also eigenfunctions of $T_{p^2}$ for all $p>2$. Corollary \ref{cor:zeros} follows from the estimate
\begin{equation} \label{eq:zerosform}
\left|
\frac{1}{ k} \iint_{\Gamma_0(4) \backslash \mathbb H} \log\left( y^{k+1/2} |g_k(z)|^2\right) \Delta_{\mathbb H} \Psi(z) \, \frac{dx \, dy}{y^2} \right|=o(1) \qquad (k \rightarrow \infty).
\end{equation}
Note that we can and will assume that $g_k$ is $L^2$-normalized.

To establish \eqref{eq:zerosform} one can follow either Rudnick's argument (see also the earlier proof of Shiffman-Zelditch \cite{Shiffman}) or the argument 
used to prove Theorem 2.1 of \cite{LesterMatomakiRadziwill}.  Both 
extend to a more general setting and 
 establish \eqref{eq:zerosform}  
provided that the following two criteria are met:
\begin{itemize}
\item[i)] Inside compact subsets of $\Gamma_0(4) \backslash \mathbb H$ one has
$y^{k+1/2} |g_k(z)|^2 \ll k^A$ for some $A>0$.
\item[ii)] For each hyperbolic ball $B(r,z) \subset \Gamma_0(4)\backslash \mathbb H$
there exists a point $z_0=x_0+iy_0 \in B(r,z)$ such that given $\varepsilon >0$ one has for all  sufficiently large $k$ that
\[
y_0^{k+1/2} |g_k(z_0)|^2 \ge e^{-k \varepsilon}.
\]
\end{itemize}
Following an argument of Iwaniec and Sarnak (see Lemma A.1 of \cite{IwaniecSarnak}), Rudnick (see Appendix A.2) showed $i)$ holds for integral weight holomorphic forms with $A=1$ and a straightforward adaptation of this method gives $i)$
 for any $g \in M_{k+1/2}(\Gamma_0(4))$. Moreover, for $g_k$ as above Steiner \cite{Steiner} proved that $i)$ holds with $A=\frac67+\varepsilon$ unconditionally and with $A=\frac12+\varepsilon$ under GRH. 
To establish $ii)$
we apply Theorem \ref{thm:holomorphic}, which gives a much better lower bound. This is the \textit{only} place where QUE is used in this argument. Since criteria $i)$ and $ii)$ hold the methods of either Rudnick and Shiffman-Zelditch, or, \cite{LesterMatomakiRadziwill} give \eqref{eq:zerosform}.
Thus, by \eqref{eq:mainzeros} and \eqref{eq:zerosform}, along with an approximation argument, we conclude that
for a compact subset $\mathcal D \subset \Gamma_0(4)\backslash \mathbb H$
with boundary measure zero we have as $k \rightarrow \infty$
\[
\frac{1}{N_{g_k}} \sum_{ \varrho \in \mathcal D} \frac{1}{\# \Gamma_{\varrho}} \cdot \tmop{ord}_{\varrho} (g_k) = \frac{\tmop{vol}(\mathcal D)}{\tmop{vol}(\Gamma_0(4)\backslash \mathbb H)}+o(1)
\]
thereby proving Corollary \ref{cor:zeros}.
Finally, note the above result shows that $\tmop{ord}_{\varrho}(g_k) =o(k)$, so since
 $\# \Gamma_z =1$ for all but $O(1)$ points
it follows that we can restate this as
$$
\sum_{ \substack{\varrho \in \mathcal D \\ g_{k}(\rho) = 0}} 1 = \frac{k}{2} \cdot \frac{\tmop{vol}(\mathcal D)}{\tmop{vol}(\Gamma_0(4)\backslash \mathbb H)}+o(k)
$$
where in the sum the zeros are counted with multiplicity. 

\section{Appendix: Proof of Proposition \ref{prop:shimura}} \label{app:one}
The formula for $|b_{g,\infty}(d)|$ with $d$ a fundamental discriminant follows from the main result of Baruch and Mao \cite{BaruchMao}. Given a fundamental discriminant $d$, consider the $d$th Shimura lift, 
$$                                                                          
\text{Sh}_d g(z) = \sqrt{y} \sum_{k \neq 0} 2 a_{\text{Sh}_d g}(|k|) K_{2 i t}(2\pi |k| y)  e(k x)
$$ 
where
\begin{equation} \notag
a_{\text{Sh}_d g}(k) =  k \sum_{\substack{m n = k \\ m, n > 0}} \frac{1}{n^{3/2}} \Big ( \frac{d}{n} \Big )   
b_{g,\infty}(d m^2).
\end{equation}
If $\text{Sh}_d g(z)$ is identically equal to zero, then by M\"obius inversion $b_{g,\infty}(d \delta^2)= 0$ for all $\delta > 0$ 
and therefore the bound $|b_{g,\infty}(d \delta^2)|^2 \ll |b_{g,\infty}(d)|^2 \delta^{2 \theta - 2 + \varepsilon}$ is vacuously true for that fundamental discriminant $d$ and all integers $\delta > 0$.  
On the other hand if $\text{Sh}_d g(z)$ is not identically zero, then following the proof of Proposition 6 in \cite{DukeImamogluToth} we conclude that $\text{Sh}_d g(z)$ is a weight $0$ Maa{\ss} form. Therefore there exists a 
Hecke normalized Hecke-Maa{\ss} form $\phi$ of weight zero, such that
$\langle \text{Sh}_d g , \phi \rangle \neq 0$ (where $\langle \cdot, \cdot \rangle$ corresponds to the Petersson inner product.
Consider now the weight $0$ Hecke operators $H_p$. These are self-adjoint for $\langle \cdot, \cdot \rangle$ and 
commute with the Shimura lift, in the sense that $H_p (\text{Sh}_d g(z)) = \text{Sh}_d ( T_{p^2} g(z))$ for all primes $p > 2$. 
Let $\lambda_p$ denote the $T_{p^2}$ eigenvalues of $g$, so that $T_{p^2} g = \lambda_p g$ for all $p > 2$, 
and denote by $\lambda_{\phi}(p)$ the Hecke eigenvalues of $\phi$, so that $H_p \phi = \lambda_{\phi}(p) \phi$ for all primes $p$.  
Using the just mentioned two facts about the Hecke operators $H_p$, we find that for all $p > 2$, 
$$
\lambda_p \langle \text{Sh}_d g, \phi \rangle  = \langle \text{Sh}_d (T_{p^2} g) , \phi \rangle
= \langle H_p (\text{Sh}_d g) , \phi \rangle = \langle \text{Sh}_d g, H_p \phi \rangle = \langle \text{Sh}_d g, \phi \rangle
\cdot \lambda_{\phi}(p).
$$
Since $\langle \text{Sh}_d g , \phi \rangle \neq 0$ we conclude that $\lambda_p = \lambda_{\phi}(p)$ for all $p > 2$, 
hence $T_{p^2} g = \lambda_{\phi}(p) g$ for all $p > 2$. Taking the Shimura lift on both sides of this equation, and 
using the relation $H_p (\text{Sh}_d g(z)) = \text{Sh}_d ( T_{p^2} g(z))$, we conclude that
$H_p (\text{Sh}_d g) = \lambda_{\phi}(p) \text{Sh}_d g$ for all $p > 2$. Hence by the strong multiplicity one theorem it follows
that $\text{Sh}_d g$ is a constant multiple of $\phi$. Since $\phi$ is Hecke normalized 
we conclude that $\text{Sh}_d g = b_{g,\infty}(d) \phi$, and that
by necessity $b_{g,\infty}(d) \neq 0$. 
Therefore, 
$$
k \sum_{\substack{m n = k \\ m, n > 0}} \frac{1}{n^{3/2}} \Big ( \frac{d}{n} \Big )                             
b_{g,\infty}(d m^2) = b_{g,\infty}(d) \lambda_{\phi}(k).
$$
Using M\"obius inversion
\[
mb_{g, \infty}(dm^2)=b_{g, \infty}(d) \sum_{k |m }
\frac{\mu(k)}{\sqrt{k}} \left( \frac{d}{k} \right) \cdot \lambda_{\phi}\left( \frac{m}{k} \right)
\]
and then applying the bound $|\lambda_{\phi}(k)| \ll k^{\theta + \varepsilon}$, where $\theta$ is the best exponent towards the Ramanujan-Petersson conjecture,  gives $|b_{g,\infty}(d \delta^2)|^2 \ll |b_{g,\infty}(d)|^2 \delta^{2 \theta - 2 + \varepsilon}$.

\section{Appendix: The functional equation}

As promised we include for completeness the proof of Lemma \ref{lem:functequ}. 
As already pointed out this follows very closely the proof given for holomorphic
forms by Kohnen and Zagier \cite{KohnenZagier}. 

\begin{proof} [Proof of Lemma \ref{lem:functequ}]
Let
$$
E^{(4)}_\infty(z,s) = \sum_{\gamma \in \Gamma_{\infty}\backslash \Gamma_0(4) } (\tmop{Im}( \gamma z))^s
$$
be the Eisenstein series for $\Gamma_0(4)$ at the cusp at $\infty$
and
$$
E_\infty(z,s) = \sum_{\gamma \in \Gamma_{\infty}\backslash \text{SL}_2(\mathbb{Z})} (\tmop{Im} (\gamma z))^s
$$
the Eisenstein series for the full modular group. 
By the unfolding method, 
\begin{equation} \label{equ:0rankin}
(4\pi)^{-s + 1} G(s) = \iint_{\Gamma_0(4)\backslash \mathbb H} |g(z)|^2 E^{(4)}_\infty(z,s) \tmop{dvol}(z).
\end{equation}
Since $E^{(4)}_\infty(z,s)$ has a simple pole at $s = 1$ with residue $\tmop{vol}(\Gamma_0(4)\backslash \mathbb H) = (2\pi)^{-1}$
we conclude that the residue of $G(s)$ at $s = 1$ is 
$$
\frac{1}{2\pi} \iint_{\Gamma_0(4) \backslash \mathbb{H}} |g(z)|^2 d \tmop{vol}(z) = \frac{1}{2\pi}.
$$
We would like to also conclude from the above representation that $G(s)$ has a functional equation relating $s$ to $1 - s$,
but this is not completely immediate since $E^{(4)}_\infty(z,1-s)$ is a linear combination of Eisenstein series
at the cusps of $\Gamma_0(4)$. In fact to establish the functional equation we will use the fact that
$g \in V^+$. Recall that $V^+$ corresponds to the subspace of $V$ with eigenvalue $1$ of
a certain operator $L$. Following Katok-Sarnak \cite{KatokSarnak} we will now provide an explicit description of $L$. 
Consider the involution
$$
\tau_2 F(z) = e^{i\pi / 4} \Big ( \frac{z}{|z|} \Big )^{-1/2} F \Big ( \frac{-1}{4z} \Big )
$$
and the Hecke operator
$$
\sigma F(z) = \frac{\sqrt{2}}{4} \sum_{v \pmod{4}} F \Big ( \frac{z + v}{4} \Big )
$$
Then $L = \tau_2 \sigma$. For a $g \in V^+$ we have $Lg = g$, and with the explicit
description of $L$ this implies an identity that we will use. First notice that,  
$$
\sigma g = \sqrt{2} g_0
$$
where for $j = 0,1$, 
$$
g_j(z) = \sum_{\substack{n \neq 0 \\ n \equiv j \pmod{4}}} b_{g,\infty}(n) W_{\tfrac 14 \tmop{sgn}(n), it} (\pi |n| y) e(n x / 4). 
$$
Therefore $L g = g$ is equivalent to $\sqrt{2} \tau_2 g_0(z) = g(z)$. Since $\tau_2$ is an involution this means that
$
\sqrt{2} g_0(z) = \tau_2 g(z)
$. 
Hence from $\sqrt{2} \tau_2 g_0(z) = g(z)$ and $\sqrt{2} g_0(z) = \tau_2 g(z)$ 
we obtain the following two identities, 
\begin{align} \notag
& \sqrt{2} g_0(z)  = e^{\pi i /4} \Big ( \frac{z}{|z|} \Big )^{-1/2} g \Big ( \frac{-1}{4z} \Big ) \\ \nonumber
& e^{\pi i /4} \Big ( \frac{z}{|z|} \Big )^{-1/2} g_0 \Big ( \frac{-1}{4z} \Big ) = \frac{1}{\sqrt{2}} g(z)
\end{align}
In addition, using the Fourier expansion of $g$, and the above two properties we get, 
\begin{align*} \label{equ:g1trans} \nonumber
e^{\pi i /4}  \Big ( \frac{z}{|z|} \Big )^{-1/2} \cdot g \Big ( \frac{1}{2} - \frac{1}{4 z} \Big ) & = e^{\pi i /4} \Big ( \frac{z}{|z|} \Big )^{-1/2}
\cdot \Big ( 2 g_0 \Big ( - \frac{1}{z} \Big ) - g \Big ( \frac{-1}{4 z} \Big ) \Big ) \\
& = \sqrt{2} \Big ( g \Big ( \frac{z}{4} \Big ) - g_0(z) \Big ) = \sqrt{2} g_1(z)
\end{align*}
using in both equations the fact that the coefficients $n \equiv 2,3 \pmod{4}$ of $g(z)$ are equal to zero. 
Now since $g_0$ and $g_1$ are automorphic forms of half-integral weight for $\Gamma_0(4)$ (since $\sigma: V \rightarrow V$, 
$\tau: V \rightarrow V$ and $V$ is a Hilbert space)
by the
Rankin-Selberg method we get
\begin{align*}
\pi^{-s + 1} G(s) & = \iint_{\Gamma_0(4)\backslash \mathbb H} |g_1(z)|^2 E^{(4)}_\infty(z,s) \tmop{dvol}(z) + \iint_{\Gamma_0(4)\backslash \mathbb H} |g_0(z)|^2 E^{(4)}_\infty(z,s) \tmop{dvol}(z) \\
& = \frac{1}{2} \iint_{\Gamma_0(4)\backslash \mathbb H} \Big | g \Big ( - \frac{1}{4z} \Big ) \Big |^2 E^{(4)}_\infty(z, s) \tmop{dvol}(z) \\
& \qquad \qquad + \frac{1}{2} \iint_{\Gamma_0(4)\backslash \mathbb H}
\Big | g \Big ( \frac{1}{2} - \frac{1}{4z} \Big ) \Big |^2 E^{(4)}_\infty(z,s) \tmop{dvol}(z) \\
& = \frac{1}{2} \iint_{\Gamma_0(4)\backslash \mathbb H} |g(z)|^2 \cdot \Big ( E^{(4)}_\infty \Big ( \frac{-1}{4z} , s \Big ) + E^{(4)}_\infty \Big ( \frac{-1}{4z + 2} , s\Big ) \Big ) \tmop{dvol}(z)
\end{align*}
Since (see for example Kohnen-Zagier, p. 191) $$
E^{(4)}_\infty \Big ( - \frac{1}{4z}, s \Big ) + E^{(4)}_\infty \Big ( - \frac{1}{4z + 2} , s \Big ) = 
\frac{1}{4^s - 1} (2^s E_\infty(2z,s) - E_\infty(4z, s))
$$
we conclude that 
\begin{equation} \label{equ:firstrankin}
\pi^{-s + 1} G(s) = \frac{1}{2} \iint_{\Gamma_0(4)\backslash \mathbb H} |g(z)|^2 \cdot \Big ( \frac{1}{4^s - 1} (2^s E_\infty(2z,s) - E_\infty(4z, s)) \Big )  \tmop{dvol}(z)
\end{equation}
But by (\ref{equ:0rankin}) and the identity 
$$
E^{(4)}_\infty(z,s) = \frac{1}{4^{s} - 1} (E_\infty(4z, s) - 2^{-s} E(2z, s))
$$
we also get
\begin{equation} \label{equ:secondrankin}
\pi^{-s + 1} G(s) = 4^{s - 1} \iint_{\Gamma_0(4)\backslash \mathbb H} |g(z)|^2 \cdot \Big ( \frac{1}{4^s - 1} (E_\infty(4z, s) - 2^{-s} E_\infty(2z,s) ) \Big ) \tmop{dvol}(z).
\end{equation}
Multiplying both sides of \eqref{equ:secondrankin} by two, then adding (\ref{equ:firstrankin}) and dividing both sides of the resulting formula by three, we get
$$
\pi^{-s + 1} G(s) = \frac{1}{6} \iint_{\Gamma_0(4)\backslash \mathbb H} |g(z)|^2 E_\infty(4z,s) \tmop{dvol}(z).
$$
Since $\pi^{-s} \Gamma(s) \zeta(2s) E_\infty(4z, s)$ is invariant under $s \mapsto 1 - s$ we conclude
that $$\widetilde{G}(s) = \pi^{-2s} \Gamma(s) \zeta(2s) G(s)$$ is invariant under $s \mapsto 1 - s$ as claimed.
\end{proof}

\section{Acknowledgments}
This paper is an outgrowth of our joint work with Kaisa Matom\"aki \cite{LesterMatomakiRadziwill}
and we are especially thankful to her for her numerous
ideas and suggestions.
We are also very grateful to Ze\'ev Rudnick and Peter Sarnak for many conversations related to this project and for
their encouragement. We would also like to thank Kannan Soundararajan for many discussions on moments 
and Valentin Blomer for very useful pointers to the literature on Whittaker functions. 

\bibliographystyle{amsplain}
\bibliography{automorphic}

\end{document}